\documentclass[12pt]{amsart}
\usepackage[T1]{fontenc}
\usepackage{amssymb}
\usepackage{a4wide}
\usepackage{graphicx}
\usepackage{verbatim}
\usepackage{amsmath}
\usepackage{version}
\usepackage{amsfonts}
\providecommand{\U}[1]{\protect\rule{.1in}{.1in}}

\newtheorem{theorem}{Theorem}[section]

\newtheorem{corollary}[theorem]{Corollary}

\newtheorem{lemma}[theorem]{Lemma}

\newtheorem{proposition}[theorem]{Proposition}
\newtheorem{remark}[theorem]{Remark}

\begin{document}
\title[Semiclassical measures and integrable systems]{Long-time dynamics of completely integrable Schr\"odinger flows on the torus}
\author[N. Anantharaman]{Nalini Anantharaman}
\address[N. Anantharaman]{Universit\'{e} Paris 11, Math\'{e}matiques, B\^{a}t. 425, 91405
ORSAY
Cedex, FRANCE} \email{Nalini.Anantharaman@math.u-psud.fr}
\author[C. Fermanian]{Clotilde~Fermanian-Kammerer}
\address[C. Fermanian]{LAMA UMR CNRS 8050,
Universit\'e Paris EST\\
61, avenue du G\'en\'eral de Gaulle\\
94010 CR\'ETEIL Cedex\\ FRANCE}
\email{clotilde.fermanian@univ-paris12.fr}
\author[F. Maci\`a]{Fabricio Maci\`a}
\address[F. Maci\`a]{Universidad Polit\'{e}cnica de Madrid. DCAIN, ETSI Navales. Avda. Arco de la
Victoria s/n. 28040 MADRID, SPAIN} \email{Fabricio.Macia@upm.es}
\thanks{N. Anantharaman wishes to acknowledge the support of Agence
Nationale de la Recherche, under the grant ANR-09-JCJC-0099-01. F.
Maci{\`a} takes part into the visiting faculty program of ICMAT
and is partially supported by grants MTM2010-16467 (MEC), ERC
Starting Grant 277778}

\begin{abstract}
In this article, we are concerned with long-time behaviour of
solutions to a semi-classical Schr\"odinger-type equation on the
torus. We consider time scales which go to infinity when the
semi-classical parameter goes to zero and we associate with each
time-scale the set of semi-classical measures associated with all
possible choices of initial data. We emphasize the existence of a
threshold~: for  time-scales below this threshold,  the set of
semi-classical measures contains measures which are singular with
respect to Lebesgue measure in the ``position'' variable, while at
(and beyond) the threshold, all  the semi-classical measures are
absolutely continuous in the ``position'' variable.

\end{abstract}
\maketitle

\newcommand{\nwc}{\newcommand}
\nwc{\nwt}{\newtheorem}
\nwt{coro}{Corollary}
\nwt{ex}{Example}
\nwt{prop}{Proposition}
\nwt{defin}{Definition}


\nwc{\mf}{\mathbf} 
\nwc{\blds}{\boldsymbol} 
\nwc{\ml}{\mathcal} 


\nwc{\lam}{\lambda}
\nwc{\del}{\delta}
\nwc{\Del}{\Delta}
\nwc{\Lam}{\Lambda}
\nwc{\elll}{\ell}

\nwc{\IA}{\mathbb{A}} 
\nwc{\IB}{\mathbb{B}} 
\nwc{\IC}{\mathbb{C}} 
\nwc{\ID}{\mathbb{D}} 
\nwc{\IE}{\mathbb{E}} 
\nwc{\IF}{\mathbb{F}} 
\nwc{\IG}{\mathbb{G}} 
\nwc{\IH}{\mathbb{H}} 
\nwc{\IN}{\mathbb{N}} 
\nwc{\IP}{\mathbb{P}} 
\nwc{\IQ}{\mathbb{Q}} 
\nwc{\IR}{\mathbb{R}} 
\nwc{\IS}{\mathbb{S}} 
\nwc{\IT}{\mathbb{T}} 
\nwc{\IZ}{\mathbb{Z}} 
\def\bbbone{{\mathchoice {1\mskip-4mu {\rm{l}}} {1\mskip-4mu {\rm{l}}}
{ 1\mskip-4.5mu {\rm{l}}} { 1\mskip-5mu {\rm{l}}}}}
\def\bbleft{{\mathchoice {[\mskip-3mu {[}} {[\mskip-3mu {[}}{[\mskip-4mu {[}}{[\mskip-5mu {[}}}}
\def\bbright{{\mathchoice {]\mskip-3mu {]}} {]\mskip-3mu {]}}{]\mskip-4mu {]}}{]\mskip-5mu {]}}}}
\nwc{\setK}{\bbleft 1,K \bbright}
\nwc{\setN}{\bbleft 1,\cN \bbright}
 \newcommand{\Lim}{\mathop{\longrightarrow}\limits}


\nwc{\va}{{\bf a}}
\nwc{\vb}{{\bf b}}
\nwc{\vc}{{\bf c}}
\nwc{\vd}{{\bf d}}
\nwc{\ve}{{\bf e}}
\nwc{\vf}{{\bf f}}
\nwc{\vg}{{\bf g}}
\nwc{\vh}{{\bf h}}
\nwc{\vi}{{\bf i}}
\nwc{\vI}{{\bf I}}
\nwc{\vj}{{\bf j}}
\nwc{\vk}{{\bf k}}
\nwc{\vl}{{\bf l}}
\nwc{\vm}{{\bf m}}
\nwc{\vM}{{\bf M}}
\nwc{\vn}{{\bf n}}
\nwc{\vo}{{\it o}}
\nwc{\vp}{{\bf p}}
\nwc{\vq}{{\bf q}}
\nwc{\vr}{{\bf r}}
\nwc{\vs}{{\bf s}}
\nwc{\vt}{{\bf t}}
\nwc{\vu}{{\bf u}}
\nwc{\vv}{{\bf v}}
\nwc{\vw}{{\bf w}}
\nwc{\vx}{{\bf x}}
\nwc{\vy}{{\bf y}}
\nwc{\vz}{{\bf z}}
\nwc{\bal}{\blds{\alpha}}
\nwc{\bep}{\blds{\epsilon}}
\nwc{\barbep}{\overline{\blds{\epsilon}}}
\nwc{\bnu}{\blds{\nu}}
\nwc{\bmu}{\blds{\mu}}
\nwc{\bet}{\blds{\eta}}



\nwc{\bk}{\blds{k}}
\nwc{\bm}{\blds{m}}
\nwc{\bM}{\blds{M}}
\nwc{\bp}{\blds{p}}
\nwc{\bq}{\blds{q}}
\nwc{\bn}{\blds{n}}
\nwc{\bv}{\blds{v}}
\nwc{\bw}{\blds{w}}
\nwc{\bx}{\blds{x}}
\nwc{\bxi}{\blds{\xi}}
\nwc{\by}{\blds{y}}
\nwc{\bz}{\blds{z}}


\nwc{\cA}{\ml{A}}
\nwc{\cB}{\ml{B}}
\nwc{\cC}{\ml{C}}
\nwc{\cD}{\ml{D}}
\nwc{\cE}{\ml{E}}
\nwc{\cF}{\ml{F}}
\nwc{\cG}{\ml{G}}
\nwc{\cH}{\ml{H}}
\nwc{\cI}{\ml{I}}
\nwc{\cJ}{\ml{J}}
\nwc{\cK}{\ml{K}}
\nwc{\cL}{\ml{L}}
\nwc{\cM}{\ml{M}}
\nwc{\cN}{\ml{N}}
\nwc{\cO}{\ml{O}}
\nwc{\cP}{\ml{P}}
\nwc{\cQ}{\ml{Q}}
\nwc{\cR}{\ml{R}}
\nwc{\cS}{\ml{S}}
\nwc{\cT}{\ml{T}}
\nwc{\cU}{\ml{U}}
\nwc{\cV}{\ml{V}}
\nwc{\cW}{\ml{W}}
\nwc{\cX}{\ml{X}}
\nwc{\cY}{\ml{Y}}
\nwc{\cZ}{\ml{Z}}

\nwc{\fA}{\mathfrak{a}}
\nwc{\fB}{\mathfrak{b}}
\nwc{\fC}{\mathfrak{c}}
\nwc{\fD}{\mathfrak{d}}
\nwc{\fE}{\mathfrak{e}}
\nwc{\fF}{\mathfrak{f}}
\nwc{\fG}{\mathfrak{g}}
\nwc{\fH}{\mathfrak{h}}
\nwc{\fI}{\mathfrak{i}}
\nwc{\fJ}{\mathfrak{j}}
\nwc{\fK}{\mathfrak{k}}
\nwc{\fL}{\mathfrak{l}}
\nwc{\fM}{\mathfrak{m}}
\nwc{\fN}{\mathfrak{n}}
\nwc{\fO}{\mathfrak{o}}
\nwc{\fP}{\mathfrak{p}}
\nwc{\fQ}{\mathfrak{q}}
\nwc{\fR}{\mathfrak{r}}
\nwc{\fS}{\mathfrak{s}}
\nwc{\fT}{\mathfrak{t}}
\nwc{\fU}{\mathfrak{u}}
\nwc{\fV}{\mathfrak{v}}
\nwc{\fW}{\mathfrak{w}}
\nwc{\fX}{\mathfrak{x}}
\nwc{\fY}{\mathfrak{y}}
\nwc{\fZ}{\mathfrak{z}}


\nwc{\tA}{\widetilde{A}}
\nwc{\tB}{\widetilde{B}}
\nwc{\tE}{E^{\vareps}}
\nwc{\tk}{\tilde k}
\nwc{\tN}{\tilde N}
\nwc{\tP}{\widetilde{P}}
\nwc{\tQ}{\widetilde{Q}}
\nwc{\tR}{\widetilde{R}}
\nwc{\tV}{\widetilde{V}}
\nwc{\tW}{\widetilde{W}}
\nwc{\ty}{\tilde y}
\nwc{\teta}{\tilde \eta}
\nwc{\tdelta}{\tilde \delta}
\nwc{\tlambda}{\tilde \lambda}
\nwc{\ttheta}{\tilde \theta}
\nwc{\tvartheta}{\tilde \vartheta}
\nwc{\tPhi}{\widetilde \Phi}
\nwc{\tpsi}{\tilde \psi}
\nwc{\tmu}{\tilde \mu}

\nwc{\To}{\longrightarrow} 

\nwc{\ad}{\rm ad}
\nwc{\eps}{\epsilon}
\nwc{\ep}{\epsilon}
\nwc{\vareps}{\varepsilon}

\def\bom{\mathbf{\omega}}
\def\om{{\omega}}
\def\ep{\epsilon}
\def\tr{{\rm tr}}
\def\diag{{\rm diag}}
\def\Tr{{\rm Tr}}
\def\i{{\rm i}}
\def\mi{{\rm i}}
\def\e{{\rm e}}
\def\sq2{\sqrt{2}}
\def\sqn{\sqrt{N}}
\def\vol{\mathrm{vol}}
\def\defi{\stackrel{\rm def}{=}}
\def\t2{{\mathbb T}^2}
\def\s2{{\mathbb S}^2}
\def\hn{\mathcal{H}_{N}}
\def\shbar{\sqrt{\hbar}}
\def\A{\mathcal{A}}
\def\N{\mathbb{N}}
\def\T{\mathbb{T}}
\def\R{\mathbb{R}}
\def\RR{\mathbb{R}}
\def\Z{\mathbb{Z}}
\def\C{\mathbb{C}}
\def\O{\mathcal{O}}
\def\Sp{\mathcal{S}_+}
\def\Lap{\triangle}
\nwc{\lap}{\bigtriangleup}
\nwc{\rest}{\restriction}
\nwc{\Diff}{\operatorname{Diff}}
\nwc{\diam}{\operatorname{diam}}
\nwc{\Res}{\operatorname{Res}}
\nwc{\Spec}{\operatorname{Spec}}
\nwc{\Vol}{\operatorname{Vol}}
\nwc{\Op}{\operatorname{Op}}
\nwc{\supp}{\operatorname{supp}}
\nwc{\Span}{\operatorname{span}}

\nwc{\dia}{\varepsilon}
\nwc{\cut}{f}
\nwc{\qm}{u_\hbar}

\def\hto0{\xrightarrow{\hbar\to 0}}
\def\htoo{\stackrel{h\to 0}{\longrightarrow}}
\def\rto0{\xrightarrow{r\to 0}}
\def\rtoo{\stackrel{r\to 0}{\longrightarrow}}
\def\ntoinf{\xrightarrow{n\to +\infty}}

\providecommand{\abs}[1]{\lvert#1\rvert}
\providecommand{\norm}[1]{\lVert#1\rVert}
\providecommand{\set}[1]{\left\{#1\right\}}

\nwc{\la}{\langle}
\nwc{\ra}{\rangle}
\nwc{\lp}{\left(}
\nwc{\rp}{\right)}

\nwc{\bequ}{\begin{equation}}
\nwc{\be}{\begin{equation}}
\nwc{\ben}{\begin{equation*}}
\nwc{\bea}{\begin{eqnarray}}
\nwc{\bean}{\begin{eqnarray*}}
\nwc{\bit}{\begin{itemize}}
\nwc{\bver}{\begin{verbatim}}

%\nwc{\eal}{\end{align}}
\nwc{\eequ}{\end{equation}}
\nwc{\ee}{\end{equation}}
\nwc{\een}{\end{equation*}}
\nwc{\eea}{\end{eqnarray}}
\nwc{\eean}{\end{eqnarray*}}
\nwc{\eit}{\end{itemize}}
\nwc{\ever}{\end{verbatim}}

\newcommand{\defeq}{\stackrel{\rm{def}}{=}}

\section{Introduction}

\subsection{The Schr\"odinger equation in the large time and high frequency
r\'egime}

This article is concerned with the dynamics of the linear equation
\begin{equation}
\left\{
\begin{array}
[c]{l}
ih\partial_{t}\psi_{h}\left(  t,x\right)  =H(hD_{x})\psi_{h}\left(
t,x\right)  ,\qquad(t,x)\in\R\times\T^{d},\\
{\psi_{h}}_{|t=0}=u_{h},
\end{array}
\right.  \label{e:eq}
\end{equation}
on the torus $\mathbb{T}^{d}:=\left(
\mathbb{R/}2\pi\mathbb{Z}\right)  ^{d}$, with $H$ a
smooth,\footnote{For the sake of simplicity, we shall assume that
$H\in\mathcal{C}^{\infty}\left(  \mathbb{R}^{d}\right)  $. However
the smoothness assumption on $H$ can be relaxed to
$\mathcal{C}^{k}$, where $k$ large enough, in most results of this
article.} real-valued function on $(\R^{d})^*$ (the dual of
$\R^{d}$), and $h>0$. In other words, $H$ is a function on the
cotangent bundle $T^*\T^d=\T^d\times (\R^{d})^*$ that does not
depend on the first variable, and thus gives rise to a completely
integrable Hamiltonian flow. We are interested in the simultaneous
limits $h\rightarrow0^{+}$ (high frequency limit) and
$t\rightarrow+ \infty$ (large time evolution). Our results give a
description of the limits of sequences of \textquotedblleft
position densities\textquotedblright\ $\left\vert \psi_{h}\left(
t_{h},x\right)  \right\vert ^{2}$ at times $t_{h}$ that tend to
infinity as $h\rightarrow0^{+}$.

To be more specific, let us denote by $S_{h}^{t}$ the propagator associated
with $H(hD_{x})$:
\[
S_{h}^{t}:=\mathrm{e}^{-i{\frac{t}{h}}H(hD_{x})}.
\]
Fix a $\emph{time}$ \emph{scale}, that is, a function
\begin{align*}
\tau:\IR_{+}^{\ast}  &  \longrightarrow\IR_{+}^{\ast}\\
h  &  \longmapsto\tau_{h},
\end{align*}
such that $\liminf_{h\rightarrow0^{+}}\tau
_{h}>0$ (actually, we shall be mainly concerned in functions that go to
$+\infty$ as $h\rightarrow0^{+}$). Consider a sequence of initial conditions
$(u_{h})$, normalised in $L^{2}(\mathbb{T}^{d})$: $\left\Vert u_{h}\right\Vert
_{L^{2}\left(  \mathbb{T}^{d}\right)  }=1$ for $h>0$, and $h$-oscillating in
the terminology of \cite{GerardMesuresSemi91, GerLeich93}, \emph{i.e.}:
\begin{equation}
\limsup_{h\rightarrow0^{+}}\left\Vert \mathbf{1}_{\left[  0,R\right]  }\left(
-h^{2}\Delta\right)  u_{h}\right\Vert _{L^{2}\left(  \mathbb{T}^{d}\right)
}\Lim_{R\To\infty}0, \label{e:hosc}
\end{equation}
where $\mathbf{1}_{\left[  0,R\right]  }$ is the characteristic function of
the interval $\left[  0,R\right]  $. Our main object of interest is the
density $\left\vert S_{h}^{t}u_{h}\right\vert ^{2}$, and we introduce the
probability measures on $\mathbb{T}^{d}$
\[
\mathbb{\nu}_{h}\left(  t,dx\right)  :=\left\vert S_{h}^{t}u_{h}(x)\right\vert
^{2}dx;
\]
the unitary character of $S_{h}^{t}$ implies that $\nu_{h}\in\mathcal{C}
\left(  \mathbb{R};\mathcal{P}\left(  \mathbb{T}^{d}\right)  \right)
$.\footnote{In what follows, $\mathcal{P}\left(  X\right)  $ stands for the
set of Radon probability measures on a Polish space $X$.} To study the
long-time behaviour of the dynamics, we rescale time by $\tau_{h}$ and look at
the time-scaled probability densities:
\begin{equation}
\nu_{h}\left(  \tau_{h}t,dx\right)  . \label{e:nuht}
\end{equation}
When $t\not =0$ is fixed and $\tau_{h}$ grows too rapidly, it is a notoriously
difficult problem to obtain a description of the limit points (in the
weak-$\ast$ topology) of these probability measures as $h\rightarrow0^{+}$,
for rich enough families of initial data $u_{h}$. See for instance
\cite{Schub-largetimes, Paul11} in the case where the underlying classical
dynamics is chaotic, the $u_{h}$ are a family of lagrangian states, and
$\tau_{h}=h^{-2+\epsilon}$. In completely integrable situations, such as the
one we consider here, the problem is of a different nature, but rapidly leads
to intricate number theoretical issues \cite{MarklofPoisson, MarklofPoisson2,
MarklofSquares}.

We soften the problem by considering the family of probability measures
\eqref{e:nuht} as elements of $L^{\infty}\left(  \mathbb{R};\mathcal{P}\left(
\mathbb{T}^{d}\right)  \right)  $. Our goal will be to give a precise
description of the set $\mathcal{M}\left(  \tau\right)  $ of their
accumulation points in the weak-$\ast$ topology for $L^{\infty}\left(
\mathbb{R};\mathcal{P}\left(  \mathbb{T}^{d}\right)  \right)  $, obtained as
$\left(  u_{h}\right)  $ varies among all possible sequences of initial data
$h$-oscillating and normalised in $L^{2}\left(  \mathbb{T}^{d}\right)  $.

The compactness of $\mathbb{T}^{d}$ ensures that $\mathcal{M}\left(
\tau\right)  $ is non-empty. Having $\nu\in\mathcal{M}\left(  \tau\right)  $
is equivalent to the existence of a sequence $(h_{n})$ going to $0$ and of a
normalised, $h_{n}$-oscillating sequence $\left(  u_{h_{n}}\right)  $ in
$L^{2}\left(  \mathbb{T}^{d}\right)  $ such that:
\begin{equation}
\label{e:averagelim}\lim_{n\rightarrow+\infty}\frac{1}{\tau_{h_{n}}}
\int_{\mathbb{\tau}_{h_{n}}a}^{\tau_{h_{n}}b}\int_{\mathbb{T}^{d}}\chi\left(
x\right)  \left\vert S_{h_{n}}^{t}u_{h_{n}}\left(  x\right)  \right\vert
^{2}dxdt=\int_{a}^{b}\int_{\mathbb{T}^{d}}\chi\left(  x\right)  \nu\left(
t,dx\right)  dt,
\end{equation}
for every real numbers $a<b$ and every $\chi\in\mathcal{C}\left(
\mathbb{T}^{d}\right)  $. In other words, we are averaging the densities
$\left\vert S_{h}^{t}u_{h}(x)\right\vert ^{2}$ over time intervals of size
$\tau_{h}$. This averaging, as we shall see, makes the study more tractable.

If case \eqref{e:averagelim} occurs, we shall say that $\nu$ is obtained
through the sequence $\left(  u_{h_{n}}\right)  $. To simplify the notation, when
no confusion can arise, we shall simply write that $h\To0^{+}$ to mean that we
are considering a (discrete) sequence $h_{n}$ going to $0^{+}$, and we shall
denote by $\left(  u_{h}\right)  $ (instead of $\left(  u_{h_{n}}\right)  $)
the corresponding family of functions.

\begin{remark}
\label{r:boundedtime}When the function $\tau$ is bounded, the convergence of
$\nu_{h}\left(  \tau_{h}t,\cdot\right)  $ to an accumulation point $\nu\left(
t,\cdot\right)  $ is locally uniform in $t$. Moreover, $\nu$ can be completely
described in terms of semiclassical defect measures of the corresponding
sequence of initial data $\left(  u_{h}\right)  $, transported by the
classical Hamiltonian flow $\phi_{s}:T^{\ast}\mathbb{T}^{d}\To T^{\ast
}\mathbb{T}^{d}$ generated by $H$, which in this case is completely integrable. Explicitly,
\[
\phi_{s}(x,\xi):=(x+sdH(\xi),\xi).
\]
This is nothing but a formulation of Egorov's theorem (see, for instance,
\cite{EvansZworski}). Consider for instance, the case where the initial data
$u_{h}$ are coherent states~: fix $\rho\in\mathcal{C}_{c}^{\infty}\left(
\mathbb{R}^{d}\right)  $ with $\left\Vert \rho\right\Vert _{L^{2}\left(
\mathbb{R}^{d}\right)  }=1$, fix $\left(  x_{0},\xi_{0}\right)  \in
\mathbb{R}^{d}\times\mathbb{R}^{d}$, and let $u_{h}\left(  x\right)  $ be the
$2\pi\mathbb{Z}^{d}$-periodization of the following coherent state:
\[
\frac{1}{h^{d/4}}\rho\left(  \frac{x-x_{0}}{\sqrt{h}}\right)  e^{i\frac
{\xi_{0}}{h}\cdot x}.
\]
Then $\nu_{h}\left(  t,\cdot\right)  $ converges, for every $t\in\mathbb{R}$,
to:
\[
\delta_{x_{0}+tdH\left(  \xi_{0}\right)  }\left(  x\right)  .
\]

\end{remark}

When the time scale $\tau_{h}$ is unbounded, the $t$-dependence of elements
$\nu\in\mathcal{M}\left(  \tau\right)  $ is not described by such a simple
propagation law. From now on we shall only consider the case where $\tau
_{h}\Lim_{h\To 0} +\infty$.

The problem of describing the elements in $\mathcal{M}\left(  \tau\right)  $
for some time scale $\left(  \tau_{h}\right)  $ is related to several aspects
of the dynamics of the flow $S_{h}^{t}$ such as dispersive effects and unique
continuation. In \cite{AnantharamanMaciaSurv, MaciaDispersion} the reader will
find a description of these issues in the case where the propagator $S_{h}
^{t}$ is replaced by the semiclassical Schr\"{o}dinger flow $e^{iht\Delta}$
corresponding to the Laplacian on an arbitrary compact Riemannian manifold
(corresponding to $H(h D_{x})=-h^{2}\Delta$ in the case of flat tori). In that
setting, the time scale $\tau_{h}=1/h$ appears in a natural way, since it
transforms the semiclassical propagator into the non-scaled flow
$e^{ih\tau_{h}t\Delta}=e^{it\Delta}$. The possible accumulation points of
sequences of probability densities of the form $|e^{it\Delta}u_{h}|^{2}$
depend on the nature of the dynamics of the geodesic flow in the manifold
under consideration. Even in the case that the geodesic flow is completely
integrable, different type of concentration phenomena may occur, depending on
fine geometrical issues (compare the situation in Zoll manifolds
\cite{MaciaAv} and on flat tori \cite{MaciaTorus, AnantharamanMacia}). When
the geodesic flow has the Anosov property, the results in \cite{AnRiv} rule
out concentration on sets of small dimensions, by proving lower bounds on the
Kolmogorov-Sinai entropy of semiclassical defect measures.

\subsection{Semiclassical defect measures}

Our results are more naturally described in terms of \emph{Wigner
distributions }and \emph{semiclassical measures} (these are the semiclassical
version of the \emph{microlocal defect measures} \cite{GerardMDM91, TartarH},
and have also been called \emph{microlocal lifts} in the recent
literature about quantum unique ergodicity, see for instance the celebrated paper \cite{LindenQUE}). The \emph{Wigner distribution} associated
to $u_{h}$ (at scale $h$) is a distribution on the cotangent bundle $T^{\ast
}\mathbb{T}^{d}$, defined by
\begin{equation}
\int_{T^{\ast}\mathbb{T}^{d}}a(x,\xi)w_{u_{h}}^{h}(dx,d\xi)=\left\langle
u_{h},\Op_{h}(a)u_{h}\right\rangle _{L^{2}(\mathbb{T}^{d})},\qquad
\mbox{ for all }a\in\mathcal{C}_{c}^{\infty}(T^{\ast}\mathbb{T}^{d}),
\label{e:initialwigner}
\end{equation}
where $\Op_{h}(a)$ is the operator on $L^{2}(\mathbb{T}^{d})$ associated to
$a$ by the Weyl quantization. The reader not familiar with these objects can
consult the appendix of this article. For the moment, just recall that if
$\chi$ is a smooth function on $T^{\ast}\mathbb{T}^{d}=\mathbb{T}^{d}
\times(\R^{d})^*$ that depends only on the first coordinate, then
\begin{equation}
\int_{T^{\ast}\mathbb{T}^{d}}\chi(x)w_{u_{h}}^{h}(dx,d\xi)=\int_{\mathbb{T}
^{d}}\chi(x)|u_{h}(x)|^{2}dx. \label{e:proj}
\end{equation}
The main object of our study will be the (time-scaled) Wigner distributions
corresponding to solutions to (\ref{e:eq}):
\[
w_{h}(t,\cdot):=w_{S_{h}^{\tau_{h}t}u_{h}}^{h}.
\]
The map $t\longmapsto w_{h}(t,\cdot)$ belongs to $L^{\infty}(\R;\cD^{\prime
}\left(  T^{\ast}\mathbb{T}^{d}\right)  )$, and is uniformly bounded in that
space as $h\To0^{+}$ whenever $\left(  u_{h}\right)  $ is normalised in
$L^{2}\left(  \mathbb{T}^{d}\right)  $. Thus, one can extract subsequences
that converge in the weak-$\ast$ topology on $L^{\infty}(\R;\cD^{\prime
}\left(  T^{\ast}\mathbb{T}^{d}\right)  )$. In other words, after possibly
extracting a subsequence, we have
\[
\int_{\R}\int_{T^*\T^d}\varphi(t)a(x,\xi)w_{h}(t,dx,d\xi)dt\Lim_{h\To0}\int_{\R}\int_{T^*\T^d }
\varphi(t)a(x,\xi)\mu(t,dx,d\xi)dt
\]
for all $\varphi\in L^{1}(\R)$ and $a\in\mathcal{C}_{c}^{\infty}(T^{\ast
}\mathbb{T}^{d})$, and the limit $\mu$ belongs to $L^{\infty}\left(
\mathbb{R};\mathcal{M}_{+}\left(  T^{\ast}\mathbb{T}^{d}\right)  \right)
$.\footnote{$\mathcal{M}_{+}\left(  X\right)  $ denotes the set of positive
Radon measures on a Polish space $X$.}

The set of limit points thus obtained, as $\left(  u_{h}\right)  $ varies
among normalised sequences, will be denoted by $\mathcal{\widetilde{M}}\left(
\tau\right)  $. We shall refer to its elements as (time-dependent)
semiclassical measures.

Moreover, if $\left(  u_{h}\right)  $ is $h$-oscillating, it follows that
$\mu\in L^{\infty}\left(  \mathbb{R};\mathcal{P}\left(  T^{\ast}\mathbb{T}
^{d}\right)  \right)  $ and identity (\ref{e:proj}) is also verified in the
limit~: if $\nu\left(  t,\cdot\right)  \ $is the image of $\mu\left(
t,\cdot\right)  $ under the projection map $\left(  x,\xi\right)  \longmapsto
x$, then%
\[
\int_{a}^{b}\int_{\mathbb{T}^{d}}\chi\left(  x\right)  |S_{h}^{\tau_{h}t}
u_{h}\left(  x\right)  |^{2}dxdt\Lim_{h\To0}\int_{a}^{b}\int_{T^{\ast
}\mathbb{T}^{d}}\chi\left(  x\right)  \mu\left(  t,dx,d\xi\right)  dt,
\]
for every $a<b$ and every $\chi\in\mathcal{C}^{\infty}\left(  \mathbb{T}
^{d}\right)  $. Therefore, $\mathcal{M}\left(  \tau\right)  $ coincides with
the set of projections onto $x$ of semiclassical measures in
$\mathcal{\widetilde{M}}\left(  \tau\right)  $ corresponding to $h$
-oscillating sequences (see~\cite{GerardMesuresSemi91,GerLeich93}).

It is also shown in the appendix that the elements of $\mathcal{\widetilde{M}
}\left(  \tau\right)  $ are measures that are invariant by the Hamiltonian
flow $\phi_{s}$:
\[
\int_{T^*\T^d}
a\circ \phi_s(x,\xi)\mu(t,dx,d\xi) = \int_{T^*\T^d}
a(x,\xi)\mu(t,dx,d\xi),\;\forall\mu\in\mathcal{\widetilde
{M}}\left(  \tau\right), \forall s\in \R, \mbox{a.e.} \,t\in \R^d.
\]

\subsection{Regularity of semiclassical measures}

The main results in this article are aimed at obtaining a precise description
of the elements in $\mathcal{\widetilde{M}}\left(  \tau\right)  $ (and, as a
consequence, of those of $\mathcal{M}\left(  \tau\right)  $). We first present
a regularity result which emphasises the critical character of the time scale
$\tau_{h}=1/h$ in situations in which the Hessian of $H$ is non-degenerate,
definite (positive or negative).

\begin{theorem}
\label{t:main}(1) If $\tau_{h}\ll1/h$ then $\mathcal{M}\left(
\tau\right)  $ contains elements that are singular with respect to
the Lebesgue measure $dtdx$. Besides,
$\mathcal{\widetilde{M}}\left(  \tau\right)  $ contains all
uniform orbit measures of $\phi_{s}$.

(2) Suppose $\tau_{h}\sim1/h$ or $\tau_{h}\gg1/h$. Assume that the Hessian
$d^{2}H(\xi)$ is definite for all~$\xi$. Then
\[
\mathcal{M}\left(  \tau\right)  \subseteq L^{\infty}\left(  \mathbb{R}
;L^{1}\left(  \mathbb{T}^{d}\right)  \right)  ,
\]
\emph{i.e. }the elements of $\mathcal{M}\left(  \tau\right)  $ are absolutely
continuous with respect to $dtdx$.
\end{theorem}

\begin{remark}
Theorem \ref{t:main}(2) applies in particular when the data $(u_{h})$ are
eigenfunctions of $H(hD_{x})$, and shows (assuming the Hessian of $H$ is
definite) that the weak limits of the probability measures $|u_{h}(x)|^{2} dx$
are absolutely continuous.
\end{remark}

\begin{remark}
The conclusion of Theorem \ref{t:main}(2) may fail if the condition on the
Hessian of $H$ is not satisfied. We give here two counter-examples.

Fix $\omega\in\mathbb{R}^{d}$ and take $H\left(  \xi\right)  =\xi\cdot\omega$.
Let $\mu_{0}$ be an accumulation point in $\cD^{\prime}\left(  T^{\ast
}\mathbb{T}^{d}\right)  $ of the Wigner distributions $\left(  w_{u_{h}}
^{h}\right)  $ defined in \eqref{e:initialwigner}, associated to the initial
data $(u_{h})$. Let $\mu\in\mathcal{\widetilde{M}}\left(  \tau\right)  $ be
the limit of $w_{S_{h}^{\tau_{h}t}u_{h}}^{h}$ in $L^{\infty}(\R;\cD^{\prime
}\left(  T^{\ast}\mathbb{T}^{d}\right)  )$. Then an application of Egorov's
theorem (actually, a particularly simple adaptation of the proof of Theorem 4
in \cite{MaciaAv}) gives the relation, valid for any time scale $\left(
\tau_{h}\right)  $~:
\[
\int_{T^{\ast}\mathbb{T}^{d}}a\left(  x,\xi\right)  \mu\left(  t,dx,d\xi
\right)  =\int_{T^{\ast}\mathbb{T}^{d}}\left\langle a\right\rangle \left(
x,\xi\right)  \mu_{0}\left(  dx,d\xi\right)  ,
\]
for any $a\in\mathcal{C}_{c}^{\infty}\left(  T^{\ast}\mathbb{T}^{d}\right)  $
and a.e. $t\in\mathbb{R}$. Here $\left\langle a\right\rangle $ stands for
the average of $a$ along the Hamiltonian flow $\phi_{s}$, that is in our case
\[
\left\langle a\right\rangle \left(  x,\xi\right)  =\lim_{T\rightarrow\infty
}\frac{1}{T}\int_{0}^{T}a\left(  x+s\omega,\xi\right)  ds.
\]
Hence, as soon as $\omega$ is resonant (in the sense of \S \ref{s:decompo})
and $\mu_{0}=\delta_{x_{0}}\otimes\delta_{\xi_{0}}$ for some $\left(
x_{0},\xi_{0}\right)  \in T^{\ast}\mathbb{T}^{d}$, the measure $\mu$ will be
singular with respect to $dtdx$.

It is also easy to provide counter-examples where the Hessian of $H$ is
non-degenerate, but not definite. On the two-dimensional torus $\IT^{2}$,
consider for instance $H(\xi)=\xi_{1}^{2}-\xi_{2}^{2}$, where $\xi=(\xi
_{1},\xi_{2})$. Take for $(u_{h}(x_{1},x_{2}))$ the periodization of
\[
\frac{1}{\left(  2\pi h\right)  ^{1/2}}\rho\left(  \frac{x_{1}-x_{2}}
{h}\right)
\]
where $\rho\in\mathcal{C}_{c}^{\infty}\left(  \mathbb{R}\right)  $ satisfies
$\left\Vert \rho\right\Vert _{L^{2}\left(  \mathbb{R}\right)  }=1$. Then the
functions $u_{h}$ are eigenfunctions of $H(hD_{x})$ and the measures
$|u_{h}(x_{1},x_{2})|^{2}dx_{1}\,dx_{2}$ obviously concentrate on the diagonal
$\{x_{1}=x_{2}\}$.
\end{remark}

Note that statement (2) of Theorem \ref{t:main} has already been
proved in the case $H(\xi)=|\xi|^{2}$ and $\tau_h=1/h$
in~\cite{BourgainQL97} and~\cite{AnantharamanMacia} with different
proofs. However, the extension of the methods in these references
to more general $H$ is not straightforward, even in the case where
$H(\xi)=\xi\cdot A\xi$, where $A$ is a symmetric linear map~:
$(\IR^{d} )^{*}\To \IR^{d}$ (i.e. the Hessian of $H$ is constant),
the difficulty arising when $A$ has irrational coefficients.

The proof in \cite{AnantharamanMacia}  extends to the
$(t,x)$-dependent Hamiltonian $\left\vert \xi\right\vert
^{2}+h^{2}V\left( t, x\right)  $ with $V$ continuous except for
points forming a set of zero-measure. Recently, this has been
extended in \cite{Burq12} to more general perturbations of the
Laplacian (allowing for potentials $V \in
L^{\infty}(\mathbb{R}\times\mathbb{T}^d)$) by means of an abstract
argument that uses the result in \cite{BourgainQL97,
AnantharamanMacia} for $V=0$ as a black-box. In fact, the proof of
the result in \cite{Burq12} applies to our context.

\begin{remark}
Theorem \ref{t:main} and \cite{Burq12} imply that statement (2) of
Theorem \ref{t:main} also holds for sequences of solutions to the
Schr\"odinger equation corresponding to the perturbed Hamiltonian
$H(hD_x) + h \tau_h^{-1} V(t)$ where $V \in L^{\infty}(\mathbb{R};
\mathcal{L}(L^2(\mathbb{T}^d)))$. The size $h \tau_h^{-1}$ of the
perturbation is in some sense optimal; in Section \ref{s:halpha}.3
we present an example communicated to us by J. Wunsch showing that
absolute continuity of the elements of $\mathcal{M}\left(
1/h\right)  $ may fail in the presence of a subprincipal symbol of
order $h^{\beta}$ with $\beta\in\left(  0,2\right)  $ even in the
case $H\left( \xi\right)  =\left\vert \xi\right\vert ^{2}$.
\end{remark}

Theorem \ref{t:main}(2) admits a microlocal refinement, which allows us to
deal with more general Hamiltonians $H$ whose Hessian is not necessarily
definite at every $\xi\in\mathbb{R}^{d}$. Given $\mu\in\mathcal{\widetilde{M}
}\left(  \tau\right)  $ we shall denote by $\bar{\mu}$ the image of $\mu$
under the map $\pi_{2}:(x,\xi)\longmapsto\xi$. It is shown in the appendix
that $\bar{\mu}$ does not depend on $t$ (it can be obtained as $\bar{\mu
}=\left(  \pi_{2}\right)  _{\ast}\mu_{0}$, where the measure $\mu_{0}$ is an
accumulation point in $\mathcal{D}^{\prime}\left(  T^{\ast}\mathbb{T}
^{d}\right)  $ of the sequence $\left(  w_{u_{h}}^{h}\right)  $).

\begin{theorem}
\label{t:lagrangian}Let $\mu\in\mathcal{\widetilde{M}}\left(  1/h\right)  $
and denote by $\mu_{\xi}(t,\cdot)$ the disintegration of $\mu(t,\cdot)$ with
respect to the variable $\xi$, \emph{i.e. }for every $\theta\in L^{1}\left(
\mathbb{R}\right)  $ and every bounded measurable function~$f$:
\[
\int_{\mathbb{R}}\theta(t)\int_{\mathbb{T}^{d}\times\mathbb{R}^{d}}f(x,\xi
)\mu(t,dx,d\xi)dt=\int_{\mathbb{R}}\theta\left(  t\right)  \int_{\mathbb{R}
^{d}}\left(  \int_{\mathbb{T}^{d}}f(x,\xi)\mu_{\xi}(t,dx)\right)  \bar{\mu
}(d\xi)dt.
\]
Then for $\bar{\mu}$-almost every $\xi$ where $d^{2}H(\xi)$ is definite, the
measure $\mu_{\xi}(t,\cdot)$ is absolutely continuous.
\end{theorem}

Let us introduce the closed set
\[
C_{H}:=\left\{  \xi\in\R^{d}:\, d^{2}H(\xi)\text{ is not definite}\right\}  .
\]
The following consequence of Theorem \ref{t:lagrangian} provides a refinement
on Theorem \ref{t:main}(2), in which the global hypothesis on the Hessian of
$H$ is replaced by a hypothesis on the sequence of initial data.

\begin{corollary}
\label{c:mainsharp}Suppose $\nu\in\mathcal{M}\left(  1/h\right)  $ is obtained
through an $h$-oscillating sequence $\left(  u_{h}\right)  $ having a
semiclassical measure $\mu_{0}$ such that $\mu_{0}\left(  \T^{d}\times
C_{H}\right)  =0$. Then $\nu$ is absolutely continuous with respect to $dtdx$.
\end{corollary}

We show in Section~\ref{s:halpha}.2 that absolute continuity may
fail for the elements of $\mathcal{M}\left(  1/h\right)  $ when
$H\left(  \xi\right)  =\left\vert \xi\right\vert ^{2k}$,
$k\in\mathbb{N}$ and $k>1$; a situation where the Hessian is
degenerate at $\xi=0$.

\subsection{Second-microlocal structure of the semiclassical measures}

Theorem \ref{t:lagrangian} is a consequence of a more detailed result on the
structure of the elements of $\mathcal{\widetilde{M}}\left(  1/h\right)  $. We
follow here the strategy of~\cite{AnantharamanMacia} that we adapt to a
general Hamiltonian $H(\xi)$. The proof relies on a decomposition of the
measure associated with the primitive submodules of $(\Z^{d})^*$. Before stating
it, we must introduce some notation.

Recall that $(\R^d)^*$ is the dual\footnote{Later in the paper, we will tend to identify both by working in the canonical bases of $\R^d$.} of $\R^d$. We will denote by $(\Z^d)^*$ the lattice in $(\R^d)^*$ defined by $(\Z^d)^*=\{\xi\in
(\R^d)^*, \xi.n\in\Z, \:\forall n\in \Z^d\}$.
We call a submodule $\Lambda\subset(\Z^{d})^*$ primitive if $\left\langle
\Lambda\right\rangle \cap(\mathbb{Z}^{d})^*=\Lambda$ (here $\left\langle
\Lambda\right\rangle $ denotes the linear subspace of $(\mathbb{R}^{d})^*$ spanned
by $\Lambda$). Given such a submodule we define:
\begin{equation}
\label{def:ILambda}I_{\Lambda}:=\left\{  \xi\in(\mathbb{R}^{d})^*: dH\left(
\xi\right)\cdot k  =0,\;\forall k\in\Lambda\right\}  .
\end{equation}
We note that $I_{\Lambda}\setminus C_{H}$ is a smooth submanifold.

We define also $L^{p}\left(  \mathbb{T}^{d},\Lambda\right)  $ for $p\in\left[
1,\infty\right]  $ to be the subspace of $L^{p}\left(  \mathbb{T}^{d}\right)
$ consisting of the functions~$u $ such that $\widehat{u}\left(  k\right)  =0$
if $k\in(\IZ^{d})^*\setminus\Lambda$ ($\widehat{u}\left(  k\right)  $ stand for
the Fourier coefficients of~$u$). Given $a\in\mathcal{C}_{c}^{\infty}\left(
T^{\ast}\mathbb{T}^{d}\right)  $ and $\xi\in\mathbb{R}^{d}$, denote by
$\left\langle a\right\rangle _{\Lambda}\left(  \cdot,\xi\right)  $ the
orthogonal projection of $a\left(  \cdot,\xi\right)  $ on $L^{2}\left(
\mathbb{T}^{d},\Lambda\right)  $:
\begin{equation}
\label{def:aLambda}\langle a\rangle_{\Lambda}= \sum_{k\in\Lambda} \widehat
a_{k}(\xi) {\frac{\mathrm{e}^{ikx}}{(2\pi)^{d}}}
\end{equation}
We denote by $m_{\left\langle a\right\rangle _{\Lambda}}\left(
\xi\right)  $ the operator acting on $L^{2}\left(  \mathbb{T}^{d}
,\Lambda\right)  $ by multiplication by $\left\langle a\right\rangle
_{\Lambda}\left(  \cdot,\xi\right)  $.

\begin{theorem}
\label{t:precise}Let $\mu\in\mathcal{\widetilde{M}}\left(  1/h\right)  $. For
every primitive submodule $\Lambda\subset(\Z^{d})^*$ there exist a positive
measure $\mu_{\Lambda}\in\mathcal{C}\left(  \mathbb{R};\mathcal{M}_{+}\left(
T^{\ast}\mathbb{T}^{d}\right)  \right)  $ supported on $\mathbb{T}^{d}\times
I_{\Lambda}$ and invariant by the Hamiltonian flow $\phi_{s}$ such that~: for
every $a\in\mathcal{C}_{c}^{\infty}\left(  T^{\ast}\mathbb{T}^{d}\right)  $
that vanishes on $\T^{d}\times C_{H}$ and every $\theta\in L^{1}\left(
\mathbb{R}\right)  $:
\begin{equation}
\int_{\mathbb{R}}\theta\left(  t\right)  \int_{T^{\ast}\mathbb{T}^{d}}a\left(
x,\xi\right)  \mu\left(  t,dx,d\xi\right)  dt=\sum_{\Lambda\subseteq
\mathbb{Z}^{d}}\int_{\mathbb{R}}\theta\left(  t\right)  \int_{\mathbb{T}
^{d}\times I_{\Lambda}}a\left(  x,\xi\right)  \mu_{\Lambda}\left(
t,dx,d\xi\right)  dt, \label{e:mh1}
\end{equation}
the sum being taken over all primitive submodules of $(\mathbb{Z}^{d})^*$.

In addition, there is a measure $\rho_{\Lambda}$ on $I_{\Lambda}$, taking
values on the set of non-negative, symmetric, trace-class operators acting on
$L^{2}\left(  \mathbb{T}^{d},\Lambda\right)  $, such that the following
holds:
\begin{equation}
\int_{\mathbb{T}^{d}\times I_{\Lambda}}a\left(  x,\xi\right)  \mu_{\Lambda
}\left(  t,dx,d\xi\right)  =\int_{I_{\Lambda}}\operatorname{Tr}\left(
e^{-\frac{it}{2}d^{2}H(\sigma)D_{y}\cdot D_{y}}m_{\left\langle a\right\rangle
_{\Lambda}}\left(  \sigma\right)  e^{\frac{it}{2}d^{2}H(\sigma)D_{y}\cdot
D_{y}}\rho_{\Lambda}(d\sigma)\right)  . \label{e:mh2}
\end{equation}

When the Hessian of $H$ is definite, formulae (\ref{e:mh1}), (\ref{e:mh2})
hold for every $a\in\mathcal{C}_{c}^{\infty}\left(  T^{\ast}\mathbb{T}
^{d}\right)  $ and therefore completely describe $\mu$.
\end{theorem}

Theorem \ref{t:precise} has been proved for $H\left(  \xi\right)  =\left\vert
\xi\right\vert ^{2}$ in \cite{MaciaTorus} for $d=2$ and in
\cite{AnantharamanMacia} for arbitrary dimension (there, a formula similar to
(\ref{e:mh2}) is proved for the $x$-dependent Hamiltonian $\left\vert
\xi\right\vert ^{2}+h^{2}V\left(  x\right)  $).

We see that Theorem \ref{t:precise} allows to describe the
dependence of $\mu$ on the parameter $t$. As was noticed in
\cite{MaciaAv, MaciaTorus}, the semiclassical measures of the
sequence of initial data $\left(  u_{h}\right) $ do not determine
uniquely the time dependent semiclassical measure $\mu$. On the
other hand, these are fully determined by the ``measures''
$\rho_{\Lambda }$, which are two-microlocal objects determined by
the initial data $(u_{h})$. The contents of \eqref{e:mh2} is that
the measure $\mu_{\Lambda}\left( t,dx,d\xi\right)  $ can be
obtained by transporting $\rho_{\Lambda}$ by a certain
Schr\"odinger flow, and then tracing out certain directions. The
$\rho_{\Lambda}$ are obtained by a process of successive
two-microlocalizations along nested sequences of submanifolds in
frequency space; this process gives an explicit construction of $\mu$ in terms of the initial data. This two-microlocal construction is in
the spirit of that done in \cite{NierScat, Fermanian2micro,
FermanianShocks} in Euclidean space. We also refer the reader to
the articles \cite{VasyWunsch09, VasyWunsch11, Wunsch10} for
related work regarding the study of the wave-front set of
solutions to semiclassical integrable systems.

\begin{remark}
The arguments in Section 6.1 of \cite{AnantharamanMacia} show that Theorem
\ref{t:lagrangian} is a consequence of Theorem~\ref{t:precise}. Therefore, in
this article only the proof of Theorem \ref{t:precise} will be presented.
\end{remark}

\begin{remark}
Theorem \ref{t:precise} holds for the time scale $\tau_{h}=1/h$.
If $\tau _{h}\ll1/h$, the elements of
$\mathcal{\widetilde{M}}\left(  \tau\right)  $ can be described by
a similar result (see Section \ref{s:rec}) involving expression
(\ref{e:mh1}). However, the propagation law appearing in the
formula replacing (\ref{e:mh2}) involves classical transport
rather than propagation along a Schr\"{o}dinger flow, and as a
result Theorem \ref{t:main}(2) does not hold for $\tau_{h}\ll1/h$.
\end{remark}

When the Hessian of $H$ is constant Theorem \ref{t:precise} gives a complement
to the results announced in \cite{AnantharamanMacia} (where the argument was
only valid when the Hessian has rational coefficients). The statement is as follows.

\begin{corollary}
Suppose $H\left(  \xi\right)  =\xi\cdot A\xi$ where $A: (\R^d)^*\To \R^d$ is a symmetric definite linear map. Given $\nu\in\mathcal{M}\left(  1/h\right)  $ there
exists, for each primitive module $\Lambda\subseteq(\mathbb{Z}^{d})^*$, a
non-negative, self-adjoint, trace-class operator $\Sigma_{\Lambda}$ acting on
$L^{2}\left(  \mathbb{T}^{d},\Lambda\right)  $ such that, for $b\in C\left(
\mathbb{T}^{d}\right)  $ and $\theta\in L^{1}\left(  \mathbb{R}\right)  $:
\begin{equation}
\int_{\mathbb{R}}\theta\left(  t\right)  \int_{\mathbb{T}^{d}}b\left(
x\right)  \nu\left(  t,dx\right)  dt=\sum_{\Lambda\subseteq\mathbb{Z}^{d}}
\int_{\mathbb{R}}\theta\left(  t\right)  \operatorname{Tr}\left(
m_{\left\langle b\right\rangle _{\Lambda}}e^{-itH\left(  D_{x}\right)  }
\Sigma_{\Lambda}e^{itH\left(  D_{x}\right)  }\right)  dt. \label{e:nu}
\end{equation}
In fact, $\Sigma_{\Lambda}:=\rho_{\Lambda}\left(  I_{\Lambda}\right)  $,
where $\rho_{\Lambda}$ is given by Theorem \ref{t:precise}.
\end{corollary}

Comparing the special case (\ref{e:nu}) to the general case \eqref{e:mh1},
\eqref{e:mh2}, we see that in the former case the propagation law involve the
constant propagator $e^{itH\left(  D_{x}\right)  }$, whereas in the latter
case we need a ``superposition'' of propagators $e^{-\frac{it}{2}d^{2}
H(\sigma)D_{y}\cdot D_{y}}$ depending on $\sigma\in I_{\Lambda}$.

\subsection{Hierarchy of time scales\label{subsec:hierarchy}}

In this section, we present a more detailed discussion on the dependence of
the set $\mathcal{M}\left(  \tau\right)  $ on the time scale $\tau$; we shall
also clarify the link between the time-dependent Wigner distributions and
those associated with eigenfunctions. Eigenfunctions are the most commonly
studied objects in the field of quantum chaos, however, we shall see that they do
not necessarily give full information about the time-dependent Wigner distributions.

A particular case of our problem is when the initial data $(u_{h})$ are
eigenfunctions of $H\left(  hD_{x}\right)  $. We note that the spectrum of
$H\left(  hD_{x}\right)  $ coincides with $H\left(  h\mathbb{Z}^{d}\right)  $;
given $E_{h}\in H\left(  h\mathbb{Z}^{d}\right)  $ the corresponding
normalised eigenfunctions are of the form:
\begin{equation}
u_{h}\left(  x\right)  =\sum_{H\left(  hk\right)  =E_{h}}c_{k}^{h}e^{ik\cdot
x},\quad\text{with}\quad\sum_{k\in\mathbb{Z}^{d}}\left\vert c_{k}
^{h}\right\vert ^{2}=\frac{1}{\left(  2\pi\right)  ^{d}}. \label{e:neigf}
\end{equation}
In addition, one has:
\[
\nu_{h}\left(  \tau_{h}t,\cdot\right)  =\left\vert S_{h}^{\tau_{h}t}
u_{h}\right\vert ^{2}=\left\vert u_{h}\right\vert ^{2},
\]
independently of $\left(  \tau_{h}\right)  $ and $t$. Let us denote by
$\mathcal{M}\left(  \infty\right)  $ the set of accumulation points in
$\mathcal{P}\left(  \mathbb{T}^{d}\right)  $ of sequences $\left\vert
u_{h}\right\vert ^{2}$ where $\left(  u_{h}\right)  $ varies among all
possible $h$-oscillating sequences of normalised eigenfunctions (\ref{e:neigf}). Denote by $\mathcal{M}_{\text{av}}\left(  \tau\right)  $ the subset of
$\mathcal{P}\left(  \mathbb{T}^{d}\right)  $ consisting of measures of the
form:\footnote{$\overline{\operatorname{Conv}X}$ stands for the closed convex
hull of a set $X\subset L^{\infty}\left(  \mathbb{R};\mathcal{P}\left(
\mathbb{T}^{d}\right)  \right)  $ with respect to the weak-$\ast$ topology.}
\[
\int_{0}^{1}\nu\left(  t,\cdot\right)  dt,\quad\text{where }\nu\in
\overline{\operatorname{Conv}\mathcal{M}\left(  \tau\right)  }\text{.}
\]

\begin{proposition}
Suppose $\left(  \tau_{h}\right)  $ and $\left(  \tau'_{h}\right)   $ are time
scales tending to infinity and such that $\tau'_{h}\ll\tau_{h}$. Then:
\[
\mathcal{M}\left(  \infty\right)  \subseteq\mathcal{M}\left(  \tau\right)
\subseteq L^{\infty}\left(  \mathbb{R};\mathcal{M}_{\mathrm{av}}\left(
\tau'\right)  \right)  .
\]

\end{proposition}

\begin{remark}
As a consequence of Theorem \ref{t:main}(2) we obtain that all eigenfunction
limits $\mathcal{M}\left(  \infty\right)  $ are absolutely continuous under
the definiteness assumption on the Hessian of $H$.
\end{remark}

A time scale of special importance is the one related to the minimal spacing of
eigenvalues~: define
\begin{equation}
\tau_{h}^{H}:=h\sup\left\{  \left\vert E_{h}^{1}-E_{h}^{2}\right\vert
^{-1}\;:\;E_{h}^{1}\neq E_{h}^{2},\;E_{h}^{1},E_{h}^{2}\in H\left(
h\mathbb{Z}^{d}\right)  \right\}  . \label{e:tauh}
\end{equation}
It is possible to have $\tau_{h}^{H}=\infty$: for instance, if $H\left(
\xi\right)  =\left\vert \xi\right\vert ^{\alpha}$ with $0<\alpha<1$ or
$H\left(  \xi\right)  =\xi\cdot A\xi$ with $A$ a real symmetric matrix that is
not proportional to a matrix with rational entries;\footnote{This is the
content of the Oppenheim conjecture, settled by Margulis \cite{MargulisDani,
MargulisOpp}.} in some other situations, such as $H\left(  \xi\right)
=\left\vert \xi\right\vert ^{\alpha}$ with $\alpha>1$, (\ref{e:tauh}) is
finite~:\ $\tau_{h}^{H}=h^{1-\alpha}$.

\begin{proposition}
\label{p:conv} If $\tau_{h}\gg\tau_{h}^{H}$ one has:
\[
\mathcal{M}\left(  \tau\right)  =\overline{\operatorname{Conv}\mathcal{M}
\left(  \infty\right)  }.
\]

\end{proposition}

This result is a consequence of the more general results presented in Section
\ref{s:hierarchy}.

Note that Proposition \ref{p:conv} allows to complete the description of
$\mathcal{M}\left(  \tau\right)  $ in the case $H\left(  \xi\right)
=\left\vert \xi\right\vert ^{2}$ as the time scale varies.

\begin{remark}
Suppose $H\left(  \xi\right)  =\left\vert \xi\right\vert ^{2}$, or more
generally, that $\tau_{h}^{H}\sim1/h$ and the Hessian of $H$ is definite.
Then:
\[
\begin{array}
[c]{ll}
\text{if }\tau_{h}\ll1/h, & \exists\nu\in\mathcal{M}\left(  \tau\right)
\text{ such that }\nu\perp dtdx;\smallskip\\
\text{if }\tau_{h}\sim1/h, & \mathcal{M}\left(  \tau\right)  \subseteq
L^{\infty}\left(  \mathbb{R};L^{1}\left(  \mathbb{T}^{d}\right)  \right)
;\smallskip\\
\text{if }\tau_{h}\gg1/h & \mathcal{M}\left(  \tau\right)  =\overline
{\operatorname{Conv}\mathcal{M}\left(  \infty\right)  }.
\end{array}
\]

\end{remark}

Note that in this case the regularity of semiclassical measures
can be precised. The elements in $\mathcal{M}\left( \infty\right)$
are trigonometric polynomials when $d=2$, as shown in
\cite{JakobsonTori97}; and in general they are more regular than
merely absolutely continuous, see \cite{Aissiou, JakobsonTori97,
JakNadToth01}. The same phenomenon occurs with those elements in
$\mathcal{M}\left( 1/h\right)$ that are obtained through sequences
whose corresponding semiclassical measures do not charge $\{ \xi=0
\}$, see \cite{AJM11}.

\subsection{Organisation of the paper}

The key argument of this article is a second microlocalisation on primitive
submodules which is the subject of Section~\ref{s:second} and leads to
Theorems~\ref{mu^Lambda} and~\ref{Thm Properties}. In Section~\ref{s:successive}, successive microlocalisations
allow to prove Theorem~\ref{t:precise} and Theorem~\ref{t:main}(2) when $\tau_h\sim 1/h$. Examples are developed in Section~\ref{s:halpha} in order to prove
Theorem~\ref{t:main}(1). Finally, the results concerning hierarchy of
time-scales are proved in Section~\ref{s:hierarchy} (and lead to Theorem~\ref{t:main} for $\tau_h\gg1/h$).

\subsection*{Acknowledgements}

The authors would like to thank Jared Wunsch for communicating to
them the construction in example (3) in Section \ref{s:halpha}.3. They are also grateful to  Luc Hillairet for
helpful discussions related to some of the results in
Section~\ref{s:hierarchy}.
Part of this work was done as F. Macia was visiting the Laboratoire
d'Analyse et de Math\'ematiques Appliqu\'ees at Universit\'e de
Paris-Est, Cr\'eteil. He wishes to thank this institution for its
support and hospitality.


\section{Two-microlocal analysis of integrable systems on $\mathbb{T}^{d}$}
\label{sec:reson}

\subsection{Invariant measures and a resonant partition of
phase-space\label{s:decompo}}

As in \cite{AnantharamanMacia}, the first step in our strategy to characterise the elements in $\widetilde
{\mathcal{M}}\left(  \tau\right)  $ consists in introducing a partition of
phase-space $T^{\ast}\mathbb{T}^{d}$ according to the order of ``resonance'' of $\xi$, that induces a decomposition of the
measures $\mu\in\widetilde{\mathcal{M}}\left(  \tau\right)  $. We say that a measure $\mu
\in\mathcal{M}_{+}\left(  T^{\ast}\mathbb{T}^{d}\right)  $ is a \emph{positive
$H$-invariant measure} on $T^{\ast}\mathbb{T}^{d}$ whenever $\mu$ is invariant
under the action of the hamiltonian flow
\begin{equation}
\left(  \phi_{s}\right)  _{\ast}\mu=\mu,\quad\text{with }\phi_{s}\left(
x,\xi\right)  =\left(  x+sdH(\xi),\xi\right)  .\label{inv}%
\end{equation}
 Recall that $\mathcal{L}$ is the family of all primitive submodules of
$(\mathbb{Z}^{d})^{\ast}$ and that with each $\Lambda\in\mathcal{L}$, we
associate the set $I_{\Lambda}$ defined in~(\ref{def:ILambda}):
if $\Lambda^{\bot}\subseteq\mathbb{R}^{d}$ is the orthogonal to $\Lambda$ with
respect to the duality in $(\mathbb{R}^{d})^{\ast}\times\mathbb{R}^{d}$ then
$I_{\Lambda}=dH^{-1}(\Lambda^{\bot})$. Denote by $\Omega_{j}\subset
\mathbb{R}^{d}$, for $j=0,...,d$, the set of resonant vectors of order exactly
$j$, that is:
\[
\Omega_{j}:=\left\{  \xi\in(\mathbb{R}^{d})^*:\operatorname*{rk}\Lambda_{\xi
}=d-j\right\}  ,
\]
where
\[
\Lambda_{\xi}:=\left\{  k\in(\mathbb{Z}^{d})^{\ast}:k\cdot dH(\xi)=0\right\}
.
\]
Note that the sets $\Omega_{j}$ form a partition of $(\mathbb{R}^{d})^*$, and that
$\Omega_{0}=dH^{-1}\left(  \left\{  0\right\}  \right)  $; more generally,
$\xi\in\Omega_{j}$ if and only if the Hamiltonian orbit $\left\{  \phi
_{s}\left(  x,\xi\right)  :s\in\mathbb{R}\right\}  $ issued from any
$x\in\mathbb{T}^{d}$ in the direction $\xi$ is dense in a subtorus of
$\mathbb{T}^{d}$ of dimension $j$.

The set $\Omega:=\bigcup_{j=0}^{d-1}\Omega_{j}$ is usually called the set of
\emph{resonant }momenta, whereas $\Omega_{d}=(\mathbb{R}^{d})^*\setminus\Omega$ is
referred to as the set of \emph{non-resonant } momenta.

Finally, write
\[
R_{\Lambda}:=I_{\Lambda}\cap\Omega_{d-\operatorname*{rk}\Lambda}.
\]
Saying that $\xi\in R_{\Lambda}$ is equivalent to any of the following statements:

\begin{enumerate}
\item[(i)] for any $x_{0}\in\IT^{d}$ the time-average $\frac{1}{T}\int_{0}
^{T}\delta_{x_{0}+tdH(\xi)}\left(  x\right)  dt$ converges weakly, as
$T\rightarrow\infty$, to the Haar measure on the torus $x_{0}+\mathbb{T}
_{\Lambda^{\perp}}\text{.}$ Here, we have used the notation $\mathbb{T}
_{\Lambda^{\perp}}:=\Lambda^{\perp}/\left(  2\pi\mathbb{Z}^{d}\cap
\Lambda^{\perp}\right)  $;

\item[(ii)] $\Lambda_{\xi}=\Lambda$.
\end{enumerate}

Moreover, if $\operatorname*{rk}\Lambda=d-1$ then $R_{\Lambda}=dH^{-1}\left(
\Lambda^{\perp}\setminus\{0\}\right)  =I_{\Lambda}\setminus\Omega_{0}$. Note
that,
\begin{equation}
(\mathbb{R}^{d})^*=\bigsqcup_{\Lambda\in\mathcal{L}}R_{\Lambda}, \label{part}
\end{equation}
that is, the sets $R_{\Lambda}$ form a partition of $(\mathbb{R}^{d})^*$. As a
consequence, any measure $\mu\in{\mathcal{M}}_{+}(T^{\ast}\R^{d})$ decomposes
as
\begin{equation}
\mu=\sum_{\Lambda\in\mathcal{L}}\mu\rceil_{\mathbb{T}^{d}\times R_{\Lambda}}.
\label{dec}
\end{equation}
Therefore, the analysis of a measure $\mu$, reduces to that of $\mu
\rceil_{\T^{d}\times R_{\Lambda}}$ for all primitive submodule $\Lambda$.
Given an $H$-invariant measure $\mu$, it turns out that $\mu\rceil
_{\T^{d}\times R_{\Lambda}}$ are utterly determined by the Fourier
coefficients of $\mu$. Indeed, define the complex measures on $\mathbb{R}^{d}
$:
\[
\widehat{\mu}_{k} :=\int_{\mathbb{T}^{d}}\frac{e^{-ik\cdot x}}{\left(
2\pi\right)  ^{d/2}}\mu\left(  dx,\cdot\right)  ,\quad k\in\mathbb{Z}^{d},
\]
so that, in the sense of distributions,
\[
\mu\left(  x,\xi\right)  =\sum_{k\in\mathbb{Z}^{d}}\widehat{\mu}_{k}\left(
\xi\right)  \frac{e^{ik\cdot x}}{\left(  2\pi\right)  ^{d/2}}.
\]
Then, the following Proposition holds.

\begin{proposition}
\label{prop:decomposition} Let $\mu\in\mathcal{M}_{+}\left(  T^{\ast
}\mathbb{T}^{d}\right)  $ and $\Lambda\in\mathcal{L}$. The distribution:
\[
\left\langle \mu\right\rangle _{\Lambda}\left(  x,\xi\right)  :=\sum
_{k\in\Lambda}\widehat{\mu}_{k}\left(  \xi\right)  \frac{e^{ik\cdot x}
}{\left(  2\pi\right)  ^{d/2}}
\]
is a finite, positive Radon measure on $T^{\ast}\mathbb{T}^{d}$.\newline
Moreover, if $\mu$ is a positive $H$-invariant measure on $T^{\ast}
\mathbb{T}^{d}$, then every term in the decomposition (\ref{dec}) is a
positive $H$-invariant measure, and
\begin{equation}
\mu\rceil_{\mathbb{T}^{d}\times R_{\Lambda}}=\left\langle \mu\right\rangle
_{\Lambda}\rceil_{\mathbb{T}^{d}\times R_{\Lambda}}. \label{eq:muRLambda}
\end{equation}
Besides, identity~(\ref{eq:muRLambda}) is equivalent to the fact that
$\mu\rceil_{\mathbb{T}^{d}\times R_{\Lambda}}$ is invariant by the
translations
\[
\left(  x,\xi\right)  \longmapsto\left(  x+v,\xi\right)  ,\quad\text{for every
}v\in\Lambda^{\perp}.
\]

\end{proposition}

The proof of Proposition~\ref{prop:decomposition} follows the lines of those
of Lemmas~6 and~7 of~\cite{AnantharamanMacia}. We also point out that this
decomposition depends on the function $H$ through the definition of
$I_{\Lambda}$. Consequently, it is possible to make different choices for
decomposing  a measure $\mu$; this fact will be exploited in
Section~\ref{s:successive}.


\subsection{Second microlocalization on a resonant  submanifold\label{s:second}}

 Let $\left(
u_{h}\right) $ be a bounded sequence in $L^{2}\left(
\mathbb{T}^{d}\right)  $ and suppose (after extraction of a
subsequence) that its Wigner distributions $w_{h}(t)=w^h_{
S^{t\tau_h}_{h}u_h}$ converge to a
semiclassical measure $\mu\in L^{\infty}\left(  \mathbb{R};\mathcal{M}
_{+}\left(  T^{\ast}\mathbb{T}^{d}\right)  \right)  $ in the
weak-$\ast$ topology of $L^{\infty}\left(
\mathbb{R};\mathcal{D}^{\prime}\left(  T^{\ast
}\mathbb{T}^{d}\right)  \right)  $.

From now on, we shall assume that the time scale $\left(
\tau_{h}\right)
$ satisfies:
\begin{equation}
(h\tau_{h}) \quad \text{is a bounded sequence.} \label{e:htaubdd}
\end{equation}
Given $\Lambda\in \cL$, the purpose of this section is to study
the measure
 $\mu\rceil_{\mathbb{T}^{d}\times R_{\Lambda}}$ by performing a second microlocalization along $I_\Lambda$ in
 the spirit of~\cite{Fermanian2micro, FermanianShocks, FermanianGerardCroisements, NierScat, MillerThesis}
and~\cite{AnantharamanMacia, MaciaTorus}. By
Proposition~\ref{prop:decomposition}, it suffices to characterize
the action of $\mu\rceil_{\mathbb{T}^{d}\times R_{\Lambda}}$ on
test functions having only $x$-Fourier modes in $\Lambda$. With
this in mind, we shall introduce two auxiliary ``distributions''
which describe more precisely how $w_{h}\left( t\right)  $
concentrates along $\mathbb{T}^{d}\times I_\Lambda$. They are
actually not mere distributions, but lie in the dual of the class
of symbols $\mathcal{S}_{\Lambda}^{1}$ that we define below.

In what follows, we fix $\xi_{0}\in R_{\Lambda}$ such that
$d^{2}H(\xi_{0})$ is definite and, without loss of
generality\footnote{This can be made by applying a cut-off in
frequencies to the data.}, we restrict our discussion to
normalised sequences of initial data $(u_{h})$ that satisfy:
\[
\widehat{u_{h}}\left(  k\right)  =0,\qquad\text{for }hk\in\mathbb{R}
^{d}\setminus B(\xi_{0};\eps/2),
\]
where $B(\xi_{0},\eps/2)$ is the ball of radius $\eps/2$
centered at $\xi_{0}$. The parameter $\eps>0$ is taken small
enough, in order that
\[
d^{2}H(\xi)\text{ is definite for  all }\xi\in B(\xi_{0},\eps)\text{;}
\]
this implies that $I_{\Lambda}\cap B(\xi_{0},\eps)$ is a
submanifold of dimension $d-\mathrm{rk}\Lambda$, everywhere
transverse to $\la\Lambda\ra$, the vector subspace of $(\R^{d})^*$
generated by $\Lambda$.\footnote{Note that this is achieved under
the weaker hypothesis that $d^{2}H(\xi)$ is non-singular and
defines a definite biliear form on $\la \Lambda \ra \times \la
\Lambda \ra$.} By eventually reducing $\eps$, we have
\[
B(\xi_{0},\eps/2)\subset(I_{\Lambda}\cap
B(\xi_{0},\eps))\oplus\la\Lambda\ra,
\]
by which we mean that any element $\xi\in B(\xi_{0},\eps/2)$ can
be decomposed in a unique way as $\xi=\sigma+\eta$ with $\sigma\in
I_{\Lambda}\cap B(\xi _{0},\eps)$ and $\eta\in\la\Lambda\ra$. We
thus get a map
\begin{align}
F:B(\xi_{0},\eps/2)  & \To\left(  I_{\Lambda}\cap
B(\xi_{0},\eps)\right)
\times\la\Lambda\ra\label{e:defF}\\
\xi & \longmapsto(\sigma(\xi),\eta(\xi))\nonumber
\end{align}

With this decomposition  of the space of frequencies, we associate
two-microlocal test-symbols. We denote by
$\mathcal{S}_{\Lambda}^{1}$ the
class of smooth functions $a\left(  x,\xi,\eta\right)  $ on $T^{\ast}
\mathbb{T}^{d}\times \la\Lambda\ra$ that are:

\begin{enumerate}
\item[(i)] compactly supported on $\left(  x,\xi\right)  \in
T^{\ast }\mathbb{T}^{d}$, $\xi\in B(\xi_0,\eps/2)$,\smallskip

\item[(ii)] homogeneous of degree zero at infinity w.r.t.
$\eta\in\la\Lambda\ra$, $\emph{i.e.}$ such that there exist
$R_{0}>0$ and $a_{\text{hom}}\in \cC_{c}^{\infty}\left(
T^{\ast}\mathbb{T}^{d}\times\mathbb{S}\la\Lambda\ra\right)  $
with
\[
a\left(  x,\xi,\eta\right)  =a_{\text{hom}}\left(
x,\xi,\frac{\eta }{\left\vert \eta\right\vert }\right)
\text{,\quad for }\left\vert
\eta\right\vert >R_{0}\text{ and }\left(  x,\xi\right)  \in T^{\ast}
\mathbb{T}^{d}
\]
(we have denoted by $\mathbb{S}\la\Lambda\ra$ the unit sphere in
$\la\Lambda\ra\subseteq(\mathbb{R}^{d})^{*}$);\smallskip

\item[(iii)] such that their non vanishing Fourier coefficients
(in the $x$
variable) correspond to frequencies $k\in\Lambda$:
\[
a\left(  x,\xi,\eta\right)
=\sum_{k\in\Lambda}\widehat{a}_{k}\left(\xi ,\eta\right)
\frac{e^{ik\cdot x}}{\left(  2\pi\right) ^{d/2}}.
\]
We will also express this fact by saying that $\emph{a}$\emph{ has
only $x$-Fourier modes in $\Lambda$.}
\end{enumerate}

Let $\chi\in \cC_{c}^{\infty}\left(\la\Lambda\ra\right)  $ be a
nonnegative cut-off function that is identically equal to one near
the origin. For $a\in\cS_{\Lambda}^{1}$,  $R>1$, $\delta<1$, we
decompose $a$ into:
$a(x,\xi,\eta)=\sum_{j=1}^3a_j(x,\xi,\eta)$ with
\begin{eqnarray}
\nonumber
 a_1(x,\xi,\eta)  & :=  & a(x,\xi,\eta)\left(1-\chi\left({\eta\over R}\right)\right)\left(1- \chi\left({\eta(\xi)\over\delta}\right) \right),\\
 \label{def:a2}
 a_2(x,\xi,\eta) & := & a(x,\xi,\eta)\left(1-\chi\left({\eta\over R}\right)\right) \chi\left({\eta(\xi)\over\delta}\right) ,\\
 \label{def:a3}
 a_3(x,\xi,\eta) & := & a(x,\xi,\eta)\chi\left({\eta\over R}\right).
 \end{eqnarray}
 This induces a decomposition of the Wigner
 distribution: $$w_h(t)=w_{h,R, \delta}^{I_\Lambda}\left(  t\right)+w_{I_\Lambda,h,R}\left(  t\right)+  w^{I_\Lambda^c}_{h,R, \delta}\left(  t\right)$$ (when testing against functions $a$ with Fourier modes in $\Lambda$), where:
\[
\left\langle w_{h,R, \delta}^{I_\Lambda}\left(  t\right)
,a\right\rangle :=\int _{T^{\ast}\mathbb{T}^{d}}   a_2\left(
x,\xi,\tau_h \eta(\xi) \right)  w_{h}\left(  t\right) \left(
dx,d\xi\right)  ,
\]
\begin{equation}
\left\langle w_{I_\Lambda,h,R}\left(  t\right)  ,a\right\rangle
:=\int_{T^{\ast }\mathbb{T}^{d}} a_3\left(  x,\xi,\tau_h\eta(\xi)
\right)  w_{h}\left(
t\right)  \left(  dx,d\xi\right)  , \label{subL}
\end{equation}
and
\begin{equation}
\left\langle w^{I_\Lambda^c}_{h,R, \delta}\left(  t\right)
,a\right\rangle :=\int_{T^{\ast }\mathbb{T}^{d}} a_1\left(
x,\xi,\tau_h\eta(\xi)\right)  w_{h}\left(
t\right)  \left(  dx,d\xi\right)  ,
\end{equation}
that we shall analyse in the limits $h\To 0^+$, $R\To +\infty$ and
$\delta\To 0$ (taken in that order). One sees that
\[
\lim_{\delta\To 0}\lim_{R\rightarrow\infty}\lim_{h\rightarrow
0}\int_{\mathbb{R}} \theta(t)\left\langle w^{I_\Lambda^c}_{h,R,
\delta}\left( t\right)  ,a\right\rangle dt =\int_{\mathbb{R}}\int
_{T^{*}\T^d} \theta(t) a_\infty\left(x,\xi,{\eta(\xi)\over
|\eta(\xi)|} \right)\mu(t, dx,d\xi)\rceil_{\T^d\times I_\Lambda^c
}dt \] where $\mu\in\mathcal{\widetilde{M}}\left( \tau_h\right) $ is
the semiclassical measure obtained through the sequence $(u_h)$.
The restriction of the measure thus obtained to $\T^d\times
R_\Lambda$ vanishes, and we do not need to further analyse
the term involving the distribution $w^{I_\Lambda^c}_{h,R,
\delta}\left( t\right)$.

For $a\in{\mathcal S}^1_\Lambda$, we introduce the notation
$$\Op^\Lambda_h(a(x,\xi,\eta)):=\Op_h\left(a\left(x,\xi,\tau_h\eta(\xi)\right)\right)$$
so that the distributions $w_{h,R,\delta}^{I_\Lambda}(t)$ and
$w_{I_\Lambda,h,R}(t)$ can be expressed for all $t\in\R$ by
\begin{eqnarray*}
\langle w_{h,R,\delta}^{I_\Lambda}(t) ,a\rangle & = & \langle u_h, S^{-\tau_ht}_{h}\Op^\Lambda_h(a_2)S^{\tau_ht}_{h}u_h\rangle_{L^{2}(\T^d)} ,\\
\langle w_{I_\Lambda,h,R}^{\Lambda}(t) ,a\rangle & = & \langle
u_h,
S^{-\tau_ht}_{h}\Op^\Lambda_h(a_3)S^{\tau_ht}_{h}u_h\rangle_{L^{2}(\T^d)}
.
\end{eqnarray*}
Notice that, for all
$\beta\in\N^{d}$,
\[
\left\Vert \partial_{\xi}^{\beta}\left(a\left(
x,h\xi,\tau_{h}{\eta(h\xi )}\right)\right)\right\Vert
_{L^{\infty}}\leq C_{\beta }\left(  \tau_{h}{h}\right)
^{|\beta|}.
\]
The Calder\'{o}n-Vaillancourt theorem (see~\cite{CV} or the
appendix of~\cite{AnantharamanMacia} for a precise statement)
therefore ensures that
there exist $N\in\N$ and $C_{N}>0$ such that
\begin{equation}
\forall a\in{\mathcal{S}}_{\Lambda}^{1},\quad\left\Vert
\Op_{h}^{\Lambda }(a)\right\Vert
_{{\mathcal{L}}(L^{2}(\R^{d}))}\leq C_{N}\sum_{|\alpha|\leq
N}\Vert\partial_{x,\xi,\eta}^{\alpha}a\Vert_{L^{\infty}},\label{eq:CVeta}
\end{equation}
since $\left(  h\tau_{h}\right)  $ is bounded.

As a consequence of~(\ref{eq:CVeta}), both $ w_{h,R, \delta}^{I_\Lambda}$ and
$w_{I_\Lambda,h,R}$ are
bounded in $L^{\infty}\left(  \mathbb{R};\left(  \mathcal{S}_{\Lambda}
^{1}\right)  ^{\prime}\right)  $. After possibly extracting
subsequences, we have for every $\varphi\in L^{1}\left(
\mathbb{R}\right)  $ and
$a\in\mathcal{S}_{\Lambda}^{1}$,
\[
\int_{\mathbb{R}}\varphi\left(  t\right)  \left\langle
\tilde{\mu}^{\Lambda
}\left(  t,\cdot\right)  ,a\right\rangle dt:=\lim_{\delta\To 0}\lim_{R\rightarrow\infty}
\lim_{h\rightarrow0^{+}}\int_{\mathbb{R}}\varphi\left( t\right)
\left\langle w_{h,R, \delta}^{I_\Lambda}\left( t\right)
,a\right\rangle dt,
\]
and
\begin{equation}
\int_{\mathbb{R}}\varphi\left(  t\right)  \left\langle
\tilde{\mu}_{\Lambda
}\left(  t,\cdot\right)  ,a\right\rangle dt:=\lim_{R\rightarrow\infty}
\lim_{h\rightarrow0^{+}}\int_{\mathbb{R}}\varphi\left( t\right)
\left\langle
w_{I_\Lambda,h,R}\left(  t\right)  ,a\right\rangle dt. \label{doublelim}
\end{equation}

\begin{remark}\label{rem:gain}
When
$\tau_{h}\ll1/h$ the quantization of our symbols
generates a semi-classical pseudodifferential calculus with gain
$h\tau_{h}$. The operators $\Op_{h}^{\Lambda
}(a)$ are semiclassical both in $\xi$ and $\eta$. This implies that the accumulation points
$\tilde\mu^\Lambda$ and $\tilde\mu_\Lambda$ are positive measures (see for instance
\cite{FermanianGerardCroisements}). \label{r:positiv}
\end{remark}

Because of the existence of
$R_{0}>0$ and of $a_{\text{hom}}\in C_{c}^{\infty}\left(
T^{\ast }\mathbb{T}^{d}\times\mathbb{S}\left\langle
\Lambda\right\rangle \right)  $ such that
\[
a\left(  x,\xi,\eta\right)  =a_{\text{hom}}\left(
x,\xi,\frac{\eta }{\left\vert \eta\right\vert }\right)
,\quad\text{for }\left\vert \eta\right\vert \geq R_{0},
\]
 for $R$ large enough, the value $\left\langle
w_{h,R,\delta}^{I_\Lambda}\left( t\right)  ,a\right\rangle $ only
depends on $a_{\text{hom}}$. Therefore, the limiting object
$\tilde{\mu}^{\Lambda}\left(  t,\cdot\right)\in \left(\mathcal{S}_{\Lambda}^{1}\right)'$ is
zero-homogeneous in the last variable $\eta\in\mathbb{R}^{d}$, supported at infinity, and,
by construction, it is supported on $\xi\in I_\Lambda$. This can
be also expressed as the fact that $\tilde\mu^\Lambda$ is a
``distribution'' on $\T^{d}\times I_{\Lambda }\times
\overline{\la\Lambda\ra}$ (where $\overline{\la\Lambda\ra}$ is the
compactification of $\la\Lambda\ra$ by adding the sphere
$\mathbb{S}\la\Lambda\ra$ at infinity) supported on
$\{\eta\in\mathbb{S}\la\Lambda\ra\}$. Besides, the distribution
$\tilde \mu_\Lambda$ is  supported on $\T^{d}\times I_{\Lambda
}\times \la \Lambda \ra$. Indeed, we have for all~$t$,
\begin{equation}\label{eq:a3decompose}
\langle w_{I_\Lambda, h, R}(t), a(x,\xi,\eta)\rangle = \langle
w_{I_\Lambda,h,R}(t), a(x,\sigma(\xi),\eta) \rangle+O(\tau_h^{-1})
\end{equation}
since, by~(\ref{eq:CVeta}),
\begin{eqnarray*}
\Op_h^\Lambda(a_3(x,\xi,\eta)) & = &
\Op_h^\Lambda(a(x,\sigma(\xi)+\tau_h^{-1} \eta,\eta)
\chi(\eta/R))\\ & = & \Op_h^\Lambda(a(x,\sigma(\xi),\eta)
\chi(\eta/R))+O(\tau_h^{-1})
\end{eqnarray*}
where the $O(\tau_h^{-1})$ term is understood in the sense of the
operator norm of ${\mathcal L}(L^2(\R^d))$ and depends on $R$ (the
fact that we first let $h$ go to $0^+$ is crucial here).

$ $

From the decomposition $w_h(t)=w_{h,R, \delta}^{I_\Lambda}\left(  t\right)+w_{I_\Lambda,h,R}\left(  t\right)+  w^{I_\Lambda^c}_{h,R, \delta}\left(  t\right)$ (when testing against symbols having Fourier modes in $\Lambda$), it is immediate that the measure $\mu\left(  t,\cdot\right)
\rceil_{\mathbb{T}^{d}\times R_{\Lambda}}$ is related to
$\tilde\mu^\Lambda$ and $\tilde\mu_\Lambda$ according to the
following Proposition.

\begin{proposition}\label{p:decomposition}
Let
\[
\mu^{\Lambda}\left(  t,\cdot\right)  :=\int_{\left\langle
\Lambda\right\rangle
}\tilde{\mu}^{\Lambda}\left(  t,\cdot,d\eta\right)  \rceil_{\mathbb{T}
^{d}\times R_{\Lambda}},\quad\mu_{\Lambda}\left(  t,\cdot\right)
:=\int_{\left\langle \Lambda\right\rangle
}\tilde{\mu}_{\Lambda}\left( t,\cdot,d\eta\right)
\rceil_{\mathbb{T}^{d}\times R_{\Lambda}}.
\]
Then both $\mu^{\Lambda}\left(  t,\cdot\right)  $ and
$\mu_{\Lambda}\left(
t,\cdot\right)  $ are $H$-invariant positive measures on $T^{\ast}\mathbb{T}^{d}$and satisfy:
\begin{equation}
\mu\left(  t,\cdot\right)  \rceil_{\mathbb{T}^{d}\times R_{\Lambda}}
=\mu^{\Lambda}\left(  t,\cdot\right)  +\mu_{\Lambda}\left(
t,\cdot\right)  . \label{eq:decomposition}
\end{equation}
\end{proposition}

This proposition motivates the analysis of the structure of the
accumulation points $\tilde{\mu}_{\Lambda}\left( t,\cdot\right)  $
and $\tilde{\mu}^{\Lambda}\left(  t,\cdot\right)  $.
 It turns out that both $\tilde{\mu}^{\Lambda}$
and $\tilde{\mu}_{\Lambda}$ have some extra regularity in the
variable $x$, although for two different reasons. Our next two
results form one of the key steps towards the proof of Theorem
\ref{t:main}.

Let us  first deal with $\tilde\mu^\Lambda(t,\cdot)$. We define, for $\left(  x,\xi,\eta\right)  \in T^{\ast}
\mathbb{T}^{d}\times \left(\la\Lambda\ra\setminus \{0\}\right)$ and $s\in\mathbb{R}$,
$$\displaylines{
\phi_{s}^{0}\left(  x,\xi,\eta\right)  :=\left(
x+sdH(\xi),\xi,\eta\right)  ,\cr \phi_{s}^{1}\left(
x,\xi,\eta\right)  :=\left(
x+sd^2H(\sigma(\xi))\cdot\frac{\eta}{|\eta|},\xi,\eta\right)
.\cr}$$
This second definition extends in an obvious way to $\eta\in \mathbb{S}\la\Lambda\ra$ (the sphere at infinity). On the other hand, the map
$\left(
x,\xi,\eta\right)\mapsto \phi_{s|\eta|}^{1}\left(
x,\xi,\eta\right)$ extends to $\eta=0$.

\begin{theorem}\label{mu^Lambda}
$\tilde{\mu}^{\Lambda}\left( t,\cdot\right)  $ is a positive
measure on $\T^{d}\times I_{\Lambda }\times
\overline{\la\Lambda\ra}$ supported on the sphere at infinity
$\mathbb{S}\la\Lambda\ra$ in the variable $\eta$. Besides, for
a.e. $t\in\mathbb{R}$, the measure $\tilde{\mu}^{\Lambda}\left(
t,\cdot\right)  $ satisfies the invariance properties:
\begin{equation}
\left(  \phi_{s}^{0}\right)  _{\ast}\tilde{\mu}^{\Lambda}\left(
t,\cdot\right)  =\tilde{\mu}^{\Lambda}\left(  t,\cdot\right)
,\quad\left( \phi_{s}^{1}\right)
_{\ast}\tilde{\mu}^{\Lambda}\left(  t,\cdot\right)
=\tilde{\mu}^{\Lambda}\left(  t,\cdot\right)  ,\quad
s\in\mathbb{R}.
\label{2minv}
\end{equation}
\end{theorem}

Note that this result holds whenever $\tau_h\ll 1/h$ or
$\tau_h=1/h$. This is in contrast with the situation we encounter
when dealing with $\tilde\mu_\Lambda(t,\cdot)$. The regularity of
this object indeed depends on the properties of the scale.

\begin{theorem}
\label{Thm Properties} (1) The distributions
$\tilde\mu_\Lambda(t,\cdot)$ are supported on $\T^d\times
I_\Lambda\times \langle \Lambda\rangle$ and are continuous with
respect to $t\in\R$. Moreover, they satisfy the following
propagation law:
\begin{equation}\label{eq:etafini}
\forall t\in\R,\qquad \tilde\mu_{\Lambda}(t,x,\xi,\eta)=(
\phi^1_{t|\eta|})_* \tilde\mu_{\Lambda}(0,x,\xi,\eta) .
\end{equation}
(2) If $\tau_h\ll1/h$ then $\tilde{\mu}_{\Lambda}\left(
t,\cdot\right)  $ is a positive measure. When $\tau_h=1/h$,
  the projection of $\tilde{\mu}_{\Lambda}\left(
t,\cdot\right)  $ on $T^{\ast}\mathbb{T}^{d}$ is a positive
measure, whose projection on $\mathbb{T}^{d}$ is absolutely
continuous with respect to the Lebesgue measure.
\end{theorem}

\begin{remark}
For $\tau_h=1/h$ the propagation law satisfied by distributions
$\tilde{\mu}_{\Lambda}\left( t,\cdot\right)$ can be interpreted in
terms of a Schr\"odinger flow type propagator. The precise
statement can be found in Proposition \ref{prop:opvame} in
Section~\ref{sec:mu_lambda}.
\end{remark}
\begin{remark}
\label{r:decRd}Note that for all $\xi\in\left(
\mathbb{R}^{d}\right)^*\setminus C_{H}$ (recall that
$C_{H}$ stands for the points where the Hessian
$d^{2}H\left(  \xi\right)  $ is not definite) we have $\mathbb{R}^{d}
=\Lambda^{\perp}\oplus d^{2}H\left(  \xi\right)  \left\langle
\Lambda \right\rangle $. Therefore, the flows $\phi_{s}^{0}$ and
$\phi_{s}^{1}$ are independent on $\mathbb{T}^{d}\times\left(
R_{\Lambda}\setminus C_{H}\right) \times\left\langle
\Lambda\right\rangle $.
\end{remark}
\begin{remark}
\label{r:nice} If $\operatorname*{rk}\Lambda=1$ then (\ref{2minv})
implies that, for a.e. $t\in\mathbb{R}$, and for any
$\nu\in\la\Lambda\ra$, the
measure $\tilde{\mu}^{\Lambda}\left(  t,\cdot\right)  \rceil_{\mathbb{T}
^{d}\times R_{\Lambda}\times\left\langle \Lambda\right\rangle }$
is invariant under
\[
(x,\sigma,\eta)\longmapsto(x+d^{2}H(\sigma)\cdot\nu,\sigma,\eta).
\]
On the other hand, the invariance by the Hamiltonian flow and
Proposition~\ref{prop:decomposition}, imply that
$\tilde{\mu}^{\Lambda}\left( t,\cdot\right)
\rceil_{\mathbb{T}^{d}\times R_{\Lambda}\times\left\langle
\Lambda\right\rangle }$ is also invariant under
\[
(x,\sigma,\eta)\longmapsto(x+v,\sigma,\eta)
\]
for every $v\in\Lambda^{\perp}$. Using Remark \ref{r:decRd} and
the fact that the Hessian
$d^{2}H\left(  \sigma\right)  $ is definite on the support of $\tilde{\mu}^{\Lambda}\left(
t,\cdot\right) \rceil_{\mathbb{T}^{d}\times
R_{\Lambda}\times\left\langle \Lambda \right\rangle }$, we conclude that the
measure $\tilde{\mu}^{\Lambda}\left(  t,\cdot\right)
\rceil_{\mathbb{T}^{d}\times R_{\Lambda}\times\left\langle \Lambda
\right\rangle }$ is constant in $x\in\mathbb{T}^{d}$ in this case.
\end{remark}

\begin{remark}\label{rem:decmuenL}
Consider the decomposition
$
\mu\left(  t,\cdot\right)
=\sum_{\Lambda\in\mathcal{L}}\mu^{\Lambda}\left( t,\cdot\right)
+\sum_{\Lambda\in\mathcal{L}}\mu_{\Lambda}\left( t,\cdot\right)  .
$
given by Proposition~\ref{p:decomposition}.
When $\tau_h=1/h$, Theorem \ref{Thm Properties} implies that the second term defines a
positive measure whose projection on $\mathbb{T}^{d}$ is
absolutely continuous with respect to the Lebesgue measure.
\end{remark}

We now give the proof of Theorems \ref{mu^Lambda} and \ref{Thm Properties}.


\subsection{Invariance properties of
$\tilde{\mu}^{\Lambda}$}\label{sec:infini}

In this section, we prove Theorem~\ref{mu^Lambda}. The positivity
of $\tilde{\mu}^{\Lambda}\left(  t,\cdot\right)  $ can be deduced
following the lines of \cite{FermanianGerardCroisements} \S 2.1,
or those of the proof of Theorem~1 in \cite{GerardMDM91}; see also
the appendix of~\cite{AnantharamanMacia}.

Let us now check the invariance property (\ref{2minv}). We use the
following Lemma which gives approximate transport equations by the
flow $\phi_{s}^{0}$.

\begin{lemma}
\label{Lemma Inv}For every $a\in\mathcal{S}_{\Lambda}^{1}$ and
$\varphi \in{\mathcal{C}}_{0}^{\infty}(\R)$, we have
\[
\int_{\mathbb{R}}\varphi(t)\langle u_{h},S_{h}^{-\tau_{h}t}\
\Op_{h}^{\Lambda
}(a)S_{h}^{\tau_{h}t}\ u_{h}\rangle_{L^{2}(\T^{d})}dt=\int_{\mathbb{R}}%
\varphi(t)\langle u_{h}\;,\;\Op_{h}^{\Lambda}\left(  a\circ\phi_{\tau_{h}%
t}^{0}\right)  u_{h}\rangle_{L^{2}(\T^{d})}dt+o(1).
\]
\end{lemma}

\begin{remark}
\label{r:phi1h}Consider the $h$-dependent flow
\[
\phi_{s}^{1,h}\left(  x,\xi,\eta\right)  :=\left(
x+s\tau_{h}\left( dH(\xi)-dH(\sigma(\xi))\right)  ,\xi,\eta\right)
,
\]
then if $a$ has only Fourier modes in $\Lambda$, we have
\begin{multline*}
\qquad\qquad a\circ\phi_{s}^{1,h}\left(  x,\xi,\eta\right)
=\sum_{k\in\Lambda}\widehat
{a}_{k}(\xi,\eta)\mathrm{e}^{ik\cdot\left(  x+{s\tau}_{h}\left(
dH(\xi)-dH(\sigma(\xi)\right)  \right)  }\\
=\sum_{k\in\Lambda}\widehat{a}_{k}(\xi,\eta)\mathrm{e}^{ik\cdot\left(
x+{s\tau}_{h}dH(\xi)\right)  }=a\circ\phi_{\tau_{h}s}^{0}\left(
x,\xi ,\eta\right)  .\qquad\qquad
\end{multline*}
This comes from the fact that for every $k\in\Lambda$ and
$\xi\in\R^{d}$, one has $k\cdot dH(\sigma(\xi))=0$.
\end{remark}

We postpone the proof of Lemma \ref{Lemma Inv} to the end of this
section and
we start proving~(\ref{2minv}). The invariance by the ``geodesic flow'' $\phi_s^0$ is standard and can be proved following the lines of the proof of property 3 in the appendix. Using~(\ref{def:a2}), we have
\begin{align}
\int_{\mathbb{R}}\varphi(t)\langle\tilde{\mu}^{\Lambda}(t,\cdot),a\rangle
dt &
=\lim_{\delta\rightarrow0}\lim_{R\rightarrow+\infty}\lim_{h\rightarrow
0}\int_{\mathbb{R}}\varphi(t)\left\langle
w_{h,R,\delta}^{I_{\Lambda}}\left(
t\right)  ,a\right\rangle dt\label{e:limmul}\\
&
=\lim_{\delta\rightarrow0}\lim_{R\rightarrow+\infty}\lim_{h\rightarrow
0}\int_{\mathbb{R}}\varphi(t)\langle w_{h}(t),a_{2}\left(  x,\xi,\tau_{h}
\eta(\xi)\right)  \rangle dt\nonumber
\end{align}
(along subsequences).
Notice that the symbol
\[
a_{2}\circ\phi_{s}^{1}(x,\xi,\eta)=a_{2}\left(
x+sd^{2}H(\sigma(\xi )){\frac{\eta}{|\eta|}},\xi,\eta\right) ,
\]
is a well-defined element of ${\mathcal{S}}_{\Lambda}^{1}$, since, for fixed $R$,
$a_{2}$ is identically equal to zero near $\eta=0$; moreover
\[
\forall\omega\in{\mathbb{S}}\la\Lambda\ra ,\;\;(a_{2}
\circ\phi_{s}^{1})_{\mathrm{hom}}(x,\xi,\omega)=a_{\mathrm{hom}}
(x+sd^{2}H(\sigma(\xi))\omega,\xi,\omega);
\]
therefore,%
\begin{equation}
\int_{\mathbb{R}}\varphi(t)\langle\tilde{\mu}^{\Lambda}(t,\cdot),a\circ
\phi_{s}^{1}\rangle dt=\lim_{\delta\rightarrow0}\lim_{R\rightarrow+\infty}
\lim_{h\rightarrow0}\int_{\mathbb{R}}\varphi(t)\left\langle
w_{h,R,\delta }^{I_{\Lambda}}\left(  t\right)
,a\circ\phi_{s}^{1}\right\rangle
dt.\label{eq:bs}%
\end{equation}
In order to relate (\ref{eq:bs}) to (\ref{e:limmul}) we note that the symbol:
\[
a_{2}\circ\phi_{s/|\eta|}^{1,h}\left( x,\xi,{\eta}\right)  ,
\]
satisfies:
\begin{eqnarray}
\nonumber
a_{2}\circ\phi_{s/|\eta|}^{1,h}\left( x,\xi,{\eta}\right)  &
= & a_{2}\left(
x+{\frac{s}{|\eta(\xi)|}}(dH(\xi)-dH(\sigma(\xi)),\xi,\tau_{h}{\eta(\xi
)}\right)  \\
\nonumber
&  = & a_{2}\left(  x+sd^{2}H(\sigma(\xi)){\frac{\eta(\xi)}{|\eta(\xi)|}}
+{\frac{s}{|\eta(\xi)|}}G(\xi)[\eta(\xi),\eta(\xi)],\xi,\tau_{h}{\eta(\xi
)}\right)  \\
&  = & \left(  a_{2}\circ\phi_{s}^{1}\right)  \left(
x,\xi,\tau_{h}\eta (\xi)\right)  +O(\delta);\label{e:1stwl}
\end{eqnarray}
we have used that, on $\operatorname*{supp}a_{2}$, we have
$|\eta(\xi)|\leq
C\delta$ and the function
\begin{equation}
G(\xi)=\int_{0}^{1}d^{3}H(\sigma(\xi)+t\eta(\xi))(1-t)dt,
\label{e:G}
\end{equation}
is uniformly bounded.
On the other hand, by Lemma~\ref{Lemma Inv} and Remark
\ref{r:phi1h}, we have
\begin{equation}
\left\langle w_{h,R,\delta}^{I_{\Lambda}}\left(  t\right)
,a\right\rangle =\left\langle w_{h}\left(  t\right)  ,a_{2}\left(
x,\xi,\tau_{h}\eta
(\xi)\right)  \right\rangle =\left\langle w_{h}\left(  0\right)  ,a_{2}
\circ\phi_{t}^{1,h}\left(  x,\xi,\tau_{h}\eta(\xi)\right)
\right\rangle
+o_{h}(1).\label{e:ta}
\end{equation}
Therefore, combining (\ref{e:1stwl}) and (\ref{e:ta}) we obtain:
\begin{equation}
\begin{split}
\int_{\mathbb{R}} & \varphi(t)\left\langle
w_{h,R,\delta}^{I_{\Lambda}}\left(
t\right)  ,a\circ\phi_{s}^{1}\right\rangle dt =\int_{\mathbb{R}}
\varphi(t)\la w_{h}\left(  0\right)
,a_2\circ \phi_{s}^{1}\circ\phi_{t}^{1,h}\left(
x,\xi,\tau_{h}\eta\left(  \xi\right)  \right)  \ra dt+O(\delta)+o_{h}(1)\\
&=\int_{\mathbb{R}}\varphi\left(  t-\frac{s}{|\eta|}\right)  \la
w_{h}\left(  0\right)  ,a_{2}\circ\phi_{t}^{1,h}\left(
x,\xi,\tau_{h}\eta (\xi)\right)  \ra dt+O(\delta)+o_{h}(1).
\end{split}\label{e:bs3}
\end{equation}
Since $|\eta|>R$ on the support of $a_{2}$, we have for all
$K\in\N$
\[
\lim_{R\rightarrow+\infty}\int_{\mathbb{R}}\sup_{x,\xi,\eta}\;\sup_{\alpha
\in\N^{3d},\,|\alpha|\leq K}\left\vert
\partial_{x,\xi,\eta}^{\alpha}\left[
\left(  \varphi(t-s/|\eta|)-\varphi(t)\right)  (a_{2}\circ\phi_{t}
^{1,h})\left(  x,\xi,\eta\right)  \right]  \right\vert dt=0,
\]
which implies, in view of (\ref{e:bs3}):
\begin{multline*}
\lim_{\delta\rightarrow0}\lim_{R\rightarrow+\infty}\lim_{h\rightarrow0}
\int_{\mathbb{R}}\varphi(t)\left\langle
w_{h,R,\delta}^{I_{\Lambda}}\left(
t\right)  ,a\circ\phi_{s}^{1}\right\rangle dt\\
=\lim_{\delta\rightarrow0}\lim_{R\rightarrow+\infty}\lim_{h\rightarrow0}
\int_{\mathbb{R}}\varphi(t)\la w_{h}\left(  0\right)
,a_{2}\circ\phi _{t}^{1,h}\left(  x,\xi,\tau_{h}\eta(\xi)\right)
\ra dt.
\end{multline*}
Applying again (\ref{e:ta}) to the left hand side of the above
identity concludes the proof of Theorem~\ref{mu^Lambda}.

Let us now prove Lemma~\ref{Lemma Inv}.

\begin{proof}
[Proof of Lemma ~\ref{Lemma Inv}.]Write
\begin{equation}
\begin{split}
\int_{\mathbb{R}}\varphi(t) &  \langle u_{h},S_{h}^{-\tau_{h}t}\ \Op_{h}
^{\Lambda}(a)S_{h}^{\tau_{h}t}\ u_{h}\rangle_{L^{2}(\T^{d})}dt\\
&
=\sum_{k,j\in\mathbb{Z}^{d},k-j\in\Lambda}\widehat{\varphi}\left(
\tau _{h}\frac{H\left(  hk\right)  -H\left(  hj\right)
}{h}\right)  \widehat {u_{h}}\left(  k\right)
\overline{\widehat{u_{h}}\left(  j\right)
}\,\widehat{a}_{j-k}\left(  {h\frac{k+j}{2}},{\tau}_{h}\eta\left(
h{\frac{k+j}{2}}\right)  \right)
\end{split}
\label{e:expli}
\end{equation}
where $a(x,\xi,\eta)=\sum_{k\in\Lambda}\widehat{a}_{k}(\xi,\eta)\mathrm{e}
^{ik\cdot x}$. Notice that,
\[
H\left(  hk\right)  -H\left(  hj\right)  =h\,dH\left(
h\frac{k+j}{2}\right) \cdot\left(  k-j\right)  +h^3r_{h}\left(h{k+j\over 2},k-j\right)  ,
\]
where  $r_{h}(\xi,\ell)$ satisfies, for every
$K\subset\mathbb{R}^{d}$ compact and convex and for every $\beta\in\N^d$,
\[
\left\vert \partial_\xi^\beta r_{h}\left(  \xi,\ell\right)  \right\vert \leq
C_{K,\beta}\left\vert \ell\right\vert ^{3},\quad \xi,h\ell\in K.
\]
Therefore, since $\widehat{\varphi}$ is uniformly Lipschitz, we
have, for $hk,hj\in K$
\[
 \widehat{\varphi}\left(  \tau_{h}\frac{H\left(
hk\right) -H\left(  hj\right)  }{h}\right)
-\widehat{\varphi}\left(  \tau_{h}dH\left( h\dfrac{k+j}{2}\right)
\cdot\left(  k-j\right)  \right) +h^2\tau_h
M_{h}\left(h{k+j\over 2} ,k-j\right)
\]
with
$$M_h(\xi,\ell)=r_h(\xi,\ell)\int_0^1\widehat \varphi '\left(\tau_hdH(\xi)\cdot \ell+sh^2\tau_h r_h(\xi,\ell)\right)ds.$$
Plugging this equation in \eqref{e:expli} we obtain:
\[
\displaylines{\int_{\mathbb{R}}\varphi(t)\langle
u_{h},S_{h}^{-\tau_ht}\ \Op_{h}^{\Lambda }(a)S_{h}^{\tau_ht}\
u_{h}\rangle_{L^{2}(\T^{d})} dt \hfill\cr\hfill =
\int_{\mathbb{R}} \varphi(t) \la u_h,
\Op^\Lambda_h\left(a\left(x+{\tau_ht} dH(\xi),\xi,\eta\right)
\right) u_h\ra_{L^2(\T^d)} dt +h^2\tau_h \langle u_h\;,\;\Op^\Lambda_h(R_{a}^{h})u_h\rangle_{L^2(\T^d)}, \cr}
\]
where the symbol $R_a^h(x,\xi,\eta)$ is characterized by its Fourier coefficients :
$$
\widehat{R_a^h}(\ell,\xi,\eta)= M_h(\xi,\ell)\widehat a_{\ell}(\xi,\eta).
 $$
 Since the function $a$ is compactly supported in $(x,\xi)$, we deduce from the properties of $M_h$ and $r_h$ that for all $\beta\in\N^d$, there exists $C_\beta>0$ such that
 $$\left| \partial_\xi^\beta R_h(x,h\xi,\tau_h\eta(h\xi))\right| \leq C_\beta,\;\;(x,\xi)\in T^*\T^d.$$
 Passing to the limit $h\rightarrow 0^+$ allows to conclude the proof of the Lemma.
\end{proof}


\subsection{Propagation and regularity of $\tilde{\mu}_{\Lambda}$ }\label{sec:mu_lambda}

 This section is devoted to proving
Theorem~\ref{Thm Properties} and showing that, when $\tau_h=1/h$,
the distribution $\tilde{\mu}_{\Lambda}$
satisfies a propagation law that involves the family of unitary propagators%
\[
e^{-\frac{it}{2}d^{2}H(\sigma)D_{y}\cdot D_{y}},\quad\sigma\in
I_{\Lambda}.
\]
We start by proving Theorem~\ref{Thm Properties}(1). The statement
on the support of $\tilde{\mu}_{\Lambda}$ was already proved in
Section \ref{s:second}. The propagation law (and hence, the
continuity with respect to $t$) comes from the following result.
\begin{proposition}\label{prop:propagation}
For every $a\in\mathcal{S}_{\Lambda}^{1}$ and every
$t\in\mathbb{R}$ the
following holds:
\[
\left\langle w_{I_{\Lambda},h,R}\left(  t\right)  ,a\right\rangle
=\left\langle w_{I_{\Lambda},h,R}\left(  0\right)
,a\circ \tilde{\phi}_t^1\right\rangle +o_{h}\left(  1\right)  ,
\]
where
\begin{equation}
\tilde{\phi}_t^1\left(  x,\xi,\eta\right)  :=\left(
x+td^{2}H(\sigma(\xi))\eta,\xi ,\eta\right)  . \label{e:a_t}
\end{equation}

\end{proposition}

\begin{proof}
We use Lemma~\ref{Lemma Inv} and Remark \ref{r:phi1h} to
conclude~:
\[
\left\langle w_{I_{\Lambda},h,R}\left(  t\right)  ,a\right\rangle
=\left\langle w_{h}\left(  0\right)
,a_{3}\circ\phi_{t}^{1,h}(x,\xi,\tau _{h}\eta(\xi))\right\rangle
+o_{h}(1),
\]
where $a_3$ is defined by (\ref{def:a3}). By definition of
$\phi_{t}^{1,h}$ and by Taylor expansion, we obtain
\begin{align*}
a_{3}\circ\phi_{t}^{1,h}\left(  x,\xi,\tau_{h}{\eta(\xi)}\right)
& =a_{3}\left(
x+{t\tau}_{h}(dH(\xi)-dH(\sigma(\xi)),\xi,\tau_{h}{\eta(\xi
)}\right)  \\
&  =a_{3}\left(  x+td^{2}H(\sigma(\xi))\tau_{h}{\eta(\xi)}+{t\tau}_{h}
G(\xi)[\eta(\xi),\eta(\xi)],\xi,\tau_{h}{\eta(\xi)}\right)  \\
&  =:b_{h}\left(  t,x,\xi,\tau_{h}{\eta(\xi)}\right),
\end{align*}
where $G$ is defined by (\ref{e:G}) and is bounded and smooth on
the support of $a_{3}$. Therefore,
\begin{align}
\label{def:bh}
b_h(t,x,\xi,\eta)= a_3 \left( x+td^2H(\sigma(\xi))\eta
+\tau_h^{-1} t G(\xi)[\eta,\eta],\xi,\eta\right)
\\
\nonumber = a_3\left(
x+td^2H(\sigma(\xi))\eta,\xi,\eta\right) +O(\tau_h^{-1}),
\end{align}
from which the result follows.
\end{proof}

Let us now focus on statement (2) of Theorem \ref{Thm Properties}. The result concerning time scales $\tau_h\ll1/h$ was already discussed in Remark \ref{r:positiv}.

From now on, suppose $\tau_h=1/h$. Let us introduce some notations. We consider the set $L^{2}
(\T^{d},\Lambda)\subseteq L^{2}(\T^{d})$ consisting of functions
having only Fourier modes in $\Lambda$. For any
$a\in{\mathcal{S}}_{\Lambda}^{1}$, and for any $\sigma\in
I_{\Lambda}$, we can associate the operator $\Op_{1}(a_{\sigma
}(y,\eta))$ which acts on $L^{2}(\T^{d},\Lambda)$, defined as
follows. For $\Phi\in L^{2}(\T^{d},\Lambda)$,
\[
\Op_{1}(a_{\sigma}(y,\eta))\Phi=\sum_{\lambda,\upsilon\in\Lambda}\widehat
{a}_{\lambda}\left(  \sigma,\upsilon+\frac{\lambda}{2}\right)
\widehat{\Phi
}\left(  \upsilon\right)  {e^{i(\upsilon+\lambda)y},}
\]
where
$a(x,\xi,\eta)=\sum_{k\in\Lambda}\widehat{a}_{k}(\xi,\eta)\frac
{\mathrm{e}^{ik\cdot x}}{(2\pi)^{d/2}}$. Thus
$\Op_{1}(a_{\sigma}(y,\eta))$ is the Weyl quantization of the
symbol $a_{\sigma}(y,\eta):=a(y,\sigma,\eta)$. In particular, for
$a=a(x,\xi)$ independent of the variable $\eta$,
$\Op_{1}(a_{\sigma}(y))$ is the multiplication operator
\begin{equation}\label{def:asigma}
\Op_{1}(a_{\sigma}(y))\Phi=a(y,\sigma)\Phi
\end{equation}
(we will simply denote by $a_{\sigma}(y)$ this multiplication
operator). We have the expression similar to formula
\eqref{e:Weylq} in the appendix,
\[
\la\Phi,\Op_{1}(a_{\sigma}(y,\eta))\Phi\ra_{L^{2}(\T^{d}, \Lambda)}=\frac{1}
{(2\pi)^{d/2}}\sum_{\upsilon,\upsilon^{\prime}\in\Lambda}\widehat{a}
_{\upsilon^{\prime}-\upsilon}\left(  \sigma,\frac{\upsilon+\upsilon^{\prime}
}{2}\right)  \widehat{\Phi}\left(  \upsilon\right)  \overline{{\widehat{\Phi}
}\left(  \upsilon^{\prime}\right)  }.
\]
Then, the last statement of Theorem~\ref{Thm Properties}(2) is a
consequence of the following
Proposition.\footnote{Recall that given a Hilbert space $H$, $\mathcal{L}
^{1}\left(  H\right)  $ stands for the space of bounded
trace-class operators acting on $H$ and, for a Polish space $X$,
$\mathcal{M}_{+}\left( X;\mathcal{L}^{1}\left(  H\right)  \right)
$ denotes the set of positive measures taking values on
$\mathcal{L}^{1}\left(  H\right)  $.}

\begin{proposition}
\label{prop:opvame}There exists $M\in\cM_{+}\left(  I_{\Lambda}
;\cL^{1}\left(  L^{2}(\T^{d},\Lambda)\right)  \right)  $ such that
for all $a\in{\mathcal{C}}_{0}^{\infty}(T^{\ast}\T^{d})$ with
Fourier modes in $\Lambda$ and all $\varphi\in L^{1}(\R)$,
\[
\int_{\mathbb{R}}\varphi\left(  t\right)
\langle\tilde{\mu}_{\Lambda}\left(
t,\cdot\right)  ,a\rangle dt=\int_{\R}\varphi(t)\int_{I_{\Lambda}}
\mathrm{Tr}\left(  e^{-\frac{it}{2}d^{2}H(\sigma)D_{y}\cdot
D_{y}}a_{\sigma }(y)e^{\frac{it}{2}d^{2}H(\sigma)D_{y}\cdot
D_{y}}M(d\sigma)\right)  dt.
\]

\end{proposition}

\begin{remark}
(i) The operator-valued measure $M$ is globally defined, it
describes the limit of $\langle w_{I_{\Lambda},h,R}(t),a\rangle$
for symbols $a=a(x,\xi)$. For symbols
$a=a(x,\xi,\eta)\in{\mathcal{S}}_{\Lambda}^{1}$,  one cannot build such a global measure (see
Remark~\ref{rem:holonomy} which emphasizes the technical obstruction).\newline
(ii) Note that for any given $\sigma$, the operator $e^{\frac{it}{2}
d^{2}H(\sigma)D_{y}\cdot D_{y}}$ obviously preserves
$L^{2}(\T^{d},\Lambda)$.
\end{remark}

The proof of Proposition~\ref{prop:opvame} relies on three steps:

\begin{enumerate}
\item We first define an operator $U_{h}$ which maps
$(2\pi\Z^{d})$-periodic functions on $(2\pi\Z^{d})$-periodic
functions with Fourier frequencies only in $\Lambda$.

\item Then, we express $w_{I_{\Lambda},h,R}(t)$ in terms of
$U_{h}$ and the operators $e^{\frac{it}{2}d^{2}H(\sigma)D_{y}\cdot
D_{y}}$.

\item We then conclude by passing to the limit when
$h\rightarrow0$ and $R\rightarrow+\infty$.
\end{enumerate}

\paragraph{\emph{First Step: Construction of the operator $U_{h}$}}

We introduce an auxiliary lattice $\tilde{\Lambda}\subset\Z^{d}$
such that $\Lambda^{\perp}\oplus\tilde{\Lambda}=\Z^{d}$ (recall that $\Lambda^{\perp}$ is the orthogonal of $\Lambda$ in the duality sense). We denote
by $\alpha$ the projection on $\la\tilde{\Lambda}\ra$, in the
direction of $\Lambda^{\perp}$. We have
$\alpha(\Z^{d})=\tilde{\Lambda}\subset\Z^{d}$. For $\sigma\in
(\R^{d})^{\ast}$, we shall denote by
$\sigma^{\alpha}\in\la\Lambda\ra$ the linear form
$\sigma^{\alpha}(y)=\sigma\cdot\alpha(y)$. We fix a bounded
fundamental domain $D_{\Lambda}$ for the action of $\Lambda$ on
$\la\Lambda \ra$. For~$\eta\in\la\Lambda\ra$, there is a unique
$\{\eta\}\in D_{\Lambda}$ (the \textquotedblleft fractional
part\textquotedblright\ of $\eta$) such that
$\eta-\{\eta\}\in\Lambda$. Finally, take
$b\in{\mathcal{C}}_{0}^{\infty }((\R^{d})^{\ast})$ supported in
the ball $B(\xi_{0},\eps)\subset (\R^{d})^{\ast}$, and identically
equal to~$1$ on $B(\xi_{0},\eps/2)$. We set for $f\in
L^{2}(\T^{d})$, $\sigma\in I_{\Lambda}$, $y\in\T^{d}$,
\begin{align*}
U_{h}f(\sigma,y)  &  =(2\pi)^{-\frac{d}{2}}e^{i\{\frac{\sigma^{\alpha}}
{h}\}\cdot
y}\int_{x\in\T^{d}}f(x)\sum_{\eta\in\la\Lambda\ra,\,\left(
\sigma,\eta\right)  \in F\left(  h\Z^{d}\right)
}b(\sigma+\eta)e^{{\frac
{i}{h}}\eta\cdot y}e^{-\frac{i}{h}(\sigma+\eta)\cdot x}dx\\
&  =(2\pi)^{-\frac{d}{2}}e^{i\{\frac{\sigma^{\alpha}}{h}\}\cdot
y}\sum _{\eta\in\la\Lambda\ra,\,\left(  \sigma,\eta\right)  \in
F\left(
h\Z^{d}\right)  }b(\sigma+\eta)\widehat{f}\left(  {\frac{\sigma+\eta}{h}
}\right)  e^{{\frac{i}{h}}\eta\cdot y}.
\end{align*}
Recall that $F$ is the local coordinate system defined in
(\ref{e:defF}), with the property that if $F\left(  \xi\right)
=\left(  \sigma\left(  \xi\right) ,\eta\left(  \xi\right)  \right)
$ then $\xi=\sigma\left(  \xi\right) +\eta\left(  \xi\right)  $.
Note that $U_{h}f(\sigma,y)=0$ if $\left( \sigma,\eta\right)
\notin F(h\Z^{d})$ for every $\eta\in\left\langle
\Lambda\right\rangle $ (since the sum has an empty index set).
The role of the term
$e^{i\{\frac{\sigma^{\alpha}}{h}\}\cdot y}$ becomes clear in the
following lemma.

\begin{lemma}
If $f$ is $(2\pi\Z)^{d}$-periodic, then $U_{h} f$ is
$(2\pi\Z)^{d}$-periodic and has only frequencies in $\Lambda$.
Therefore, $U_{h}$ maps $L^{2}(\T^{d})$ into the subspace
$L^{2}(\T^{d},\Lambda)$ of $L^{2}(\T^{d})$.
\end{lemma}

\begin{proof}
It is enough to show that for any $\sigma\in I_{\Lambda}$,
$\eta\in\la\Lambda\ra$ such that $\sigma+\eta\in h\Z^{d}$,
\[
\frac{\eta}{h}+\left\{  \frac{\sigma^{\alpha}}{h}\right\}
\in\Lambda.
\]
By definition,
$\frac{\eta}{h}+\{\frac{\sigma^{\alpha}}{h}\}\in\la\Lambda\ra$,
and we want to prove that for any $k\in2\pi\mathbb{Z}^{d}$,
$\left( \frac{\eta}{h}+\{\frac{\sigma^{\alpha}}{h}\}\right)  \cdot
k\in2\pi\Z$. We write
\[
\left(  \frac{\eta}{h}+\left\{  \frac{\sigma^{\alpha}}{h}\right\}
\right)
\cdot k=\left(  \frac{\sigma+\eta}{h}+\left\{  \frac{\sigma^{\alpha}}
{h}\right\}  -\frac{\sigma}{h}\right)  \cdot k
\]
and we know that $\frac{\sigma+\eta}{h}\cdot k\in2\pi\mathbb{Z}$.
We then use the fact that there exists $\lambda\in\Lambda$ such
that $\left\{ {\frac{\sigma^{\alpha}}{h}}\right\}
={\frac{\sigma^{\alpha}}{h}}+\lambda$, and write
\[
\left\{  {\frac{\sigma^{\alpha}}{h}}\right\}  \cdot k-{\frac{\sigma\cdot k}
{h}}={\frac{\sigma}{h}}\cdot\alpha(k)+\lambda\cdot
k-{\frac{\sigma}{h}}\cdot
k={\frac{\sigma}{h}}\cdot(\alpha(k)-k)+\lambda\cdot k.
\]
Since $\sigma+\eta=hl$ for some $l\in \Z^{d}$ and since
$k\in(2\pi\Z)^{d}$, we obtain
${\frac{\sigma}{h}}\cdot(\alpha(k)-k)=l \cdot(\alpha(k)-k)$ with $\alpha(k)-k\in\Lambda^\perp\cap(2\pi\Z^d)$ (since $\alpha(\Z^d)=\Lambda\subset\Z^d$). Finally, we get ${\frac{\sigma}{h}}\cdot(\alpha(k)-k)\in2\pi\Z$  and we also have
$\lambda\cdot k\in2\pi\Z$, which concludes the proof.
\end{proof}

Note that if $f$ is $(2\pi\Z)^{d}$-periodic, the Fourier
coefficients of $U_{h}f$ satisfy
\[
\forall\eta\in\Lambda,\;\;\widehat{U_{h}f}(\sigma,\eta)=b\left(
\sigma +h\eta-h\left\{  {\frac{\sigma^{\alpha}}{h}}\right\}
\right)  \widehat
{f}\left(  {\frac{\sigma}{h}}+\eta-\left\{  {\frac{\sigma^{\alpha}}{h}
}\right\}  \right)
\]
and we have the following Plancherel-type formula.

\begin{lemma}
\label{lem:plancherel} If $f$ is $(2\pi\Z)^{d}$-periodic, then we
have
\[
\forall\sigma\in I_{\Lambda},\;\;\sum_{k\in\Z^{d}}|\hat{f}(k)b(hk)|^{2}
=\sum_{\sigma\in
\sigma(h\Z^{d})}\int_{\T^{d}}|U_{h}f(\sigma,y)|^{2}dy.
\]

\end{lemma}

\begin{proof}
We have for all $(2\pi\Z^{d})$-periodic function $f$,
\begin{align*}
\sum_{k\in\Z^{d}}|\widehat{f}(k)b(hk)|^{2}  &  =\,\sum_{\sigma\in
I_{\Lambda },\eta\in\la\Lambda\ra\,\sigma+\eta\in
h\Z^{d}}|b(\sigma+\eta)|^{2}\left\vert
\widehat{f}\left(  {\frac{\sigma+\eta}{h}}\right)  \right\vert ^{2}\\
&  =\frac{1}{(2\pi)^{d}}\sum_{\sigma\in \sigma(h\Z^{d})}\int_{\T^{d}}
\sum_{\sigma+\eta,\sigma+\eta^{\prime}\in
h\Z^{d}}b(\sigma+\eta)\overline
{b(\sigma+\eta^{\prime})}\widehat{f}\left(
{\frac{\sigma+\eta}{h}}\right)
\overline{\widehat{f}\left(  {\frac{\sigma+\eta^{\prime}}{h}}\right)  }\\
&  \qquad\qquad\times\,\mathrm{exp}\left(  {\frac{i}{h}}\left(
y\cdot
(\eta-\eta^{\prime})\right)  \right)  dy\\
&  =\sum_{\sigma\in
\sigma(h\Z^{d})}\int_{\T^{d}}|U_{h}f(\sigma,y)|^{2}dy.
\end{align*}
\end{proof}

\paragraph{\emph{Second step: Link between $w_{I_{\Lambda},h,R}$ and
$U_{h}$}}

It is in this step that we really see the relevance of the objects
introduced previously. It comes from the two following lemmas~:

\begin{lemma}
\label{lem:16} For any $a\in\cS_{\Lambda}^{1}$,
\[
\displaylines{ \int_{T^{\ast}\mathbb{T}^{d}}a_3\left(x,\xi,
\frac{\eta(\xi)}h\right)w_{u_{h}}^{h}(dx,d\xi)\hfill\cr\hfill =
\sum_{\sigma\in \sigma(h\Z^d)} \la U_h u_h(\sigma, y)
,\Op_1\left(a_\sigma\left(y,\eta-\left\{{\sigma^\alpha\over
h}\right\}\right)
\chi\left(\frac{\eta-\left\{\frac{\sigma^\alpha}h\right\}}R\right)\right)U_h
u_h(\sigma, y)\ra_{L^2(\T^d)}\cr\hfill +O(h)\qquad.\cr}
\]

\end{lemma}

\begin{proof}
We have by~(\ref{eq:a3decompose}),
\[
\int_{T^{\ast}\T^{d}}a_{3}\left(
x,\xi,{\frac{\eta(\xi)}{h}}\right)
w_{u_{h}}^{h}(dx,d\xi)=\int_{T^{\ast}\T^{d}}a_{3}\left(
x,\sigma(\xi ),{\frac{\eta(\xi)}{h}}\right)
w_{u_{h}}^{h}(dx,d\xi)+O(h).
\]
Then, using~(\ref{e:Weylq})
\[
\displaylines{
\int_{T^{\ast}\mathbb{T}^{d}}a_3\left(x,\sigma(\xi),
\frac{\eta(\xi)}h\right)w_{u_{h}}^{h}(dx,d\xi)\hfill\cr\hfill=\frac{1}{(2\pi)^{d/2}}
\sum_{k-k'\in \Lambda} \widehat u_h(k) \overline{\widehat u_h(k')}
\widehat a_{k'-k} \left(\sigma\left(h{k+k'\over 2} \right),{1\over
h} \eta\left(h{k+k'\over 2}\right) \right)\chi \left({1\over hR}
\eta\left(h{k+k'\over 2}\right) \right).\cr}
\]
We write
\[
hk=\sigma+h\eta,\;hk^{\prime}=\sigma+h\eta^{\prime}\;\;,\,\left(
\sigma ,h\eta\right)  ,\left(  \sigma,h\eta^{\prime}\right)  \in
F(h\Z^{d}),
\]
using the fact that $k^{\prime}-k\in\Lambda$. In particular,
$\sigma\left( h{\frac{k+k^{\prime}}{2}}\right)  =\sigma$. Then,
\begin{multline*}
\int_{T^{\ast}\mathbb{T}^{d}}a_{3}\left(  x,\sigma(\xi),\frac{\eta(\xi)}
{h}\right)  w_{u_{h}}^{h}(dx,d\xi)\hfill\cr\hfill=\frac{1}{(2\pi)^{d/2}}
\sum_{\sigma+h\eta,\sigma+h\eta^{\prime}\in
h\Z^{d}}\;\widehat{u}_{h}\left( \frac{\sigma}{h}+\eta\right)
\overline{\widehat{u}_{h}\left(  \frac{\sigma
}{h}+\eta^{\prime}\right)  }\widehat{a}_{\eta^{\prime}-\eta}\left(
\sigma,{\frac{\eta+\eta^{\prime}}{2}}\right)  \chi\left(
\frac{\eta +\eta^{\prime}}{2R}\right)
\cr\hfill=\frac{1}{(2\pi)^{d/2}}\sum_{\sigma\in
\sigma(h\Z^{d})}\sum_{\eta+\left\{
{\frac{\sigma^{\alpha}}{h}}\right\} ,\eta^{\prime}+\left\{
{\frac{\sigma^{\alpha}}{h}}\right\}  \in\Lambda
}\widehat{U_{h}u_{h}}\left(  \sigma,\eta+\left\{  {\frac{\sigma^{\alpha}}{h}
}\right\}  \right)  \overline{\widehat{U_{h}u_{h}}\left(
\sigma,\eta^{\prime
}+\left\{  {\frac{\sigma^{\alpha}}{h}}\right\}  \right)  }\\
\hfill\widehat{a}_{\eta^{\prime}-\eta}\left(
\sigma,{\frac{\eta+\eta^{\prime }}{2}}\right)  \chi\left(
\frac{\eta+\eta^{\prime}}{2R}\right)
\cr\hfill=\frac{1}{(2\pi)^{d/2}}\sum_{\sigma\in
\sigma(h\Z^{d})}\sum_{\eta
,\eta^{\prime}\in\Lambda}\widehat{U_{h}u_{h}}\left(
\sigma,\eta\right)
\overline{\widehat{U_{h}u_{h}}\left(  \sigma,\eta^{\prime}\right)  }\hfill\\
\widehat{a}_{\eta^{\prime}-\eta}\left(  \sigma,{\frac{\eta+\eta^{\prime}}{2}
}-\left\{  {\frac{\sigma^{\alpha}}{h}}\right\}  \right)
\chi\left(  \frac
{1}{R}\left(  \frac{\eta+\eta^{\prime}}{2}-\left\{  {\frac{\sigma^{\alpha}}
{h}}\right\}  \right)  \right)  .
\end{multline*}
which is the desired expression.
\end{proof}

To simplify the notation in the computations that follow, we set:
\[
A\left(  \sigma,\eta\right)
:=\frac{1}{2}d^{2}H(\sigma)\eta\cdot\eta.
\]

\begin{lemma}
For any $a\in\cS_{\Lambda}^{1}$, for any $t\in\R$,
\begin{multline*}
\int_{T^{\ast}\mathbb{T}^{d}}a_{3}\left(
x,\xi,\frac{\eta(\xi)}{h}\right)
w_{h}(t,dx,d\xi)\\
=\sum_{\sigma\in \sigma(h\Z^{d})}\la e^{-itA\left(  \sigma,D_{y}\right)  }
U_{h}u_{h}(\sigma,y),\\
\hfill\Op_{1}\left(  a_{\sigma}\left(  y,\eta-\left\{  {\frac{\sigma^{\alpha}
}{h}}\right\}  \right)  \chi\left(  \frac{\eta-\left\{  \frac{\sigma^{\alpha}
}{h}\right\}  }{R}\right)  \right)  e^{itA\left(  \sigma,D_{y}\right)  }
U_{h}u_{h}(\sigma,y)\ra_{L^{2}(\T^{d})}\\
+o(1)
\end{multline*}

\end{lemma}

\begin{proof}
We use Proposition \ref{prop:propagation} and apply
Lemma \ref{lem:16} to the symbol $b_0(t,x,\xi,\eta)$ which was defined in the proof of Proposition
\ref{prop:propagation} (see~(\ref{def:bh}) with $h=0$). Then,  the result follows by using the fundamental property of the Weyl
quantization, namely that, since $A\left(  \sigma,\eta\right)  $
is quadratic in $\eta$ and
does not depend on $y$:
\[
\Op_{1}\left(  b_{0}(t,y,\sigma,\eta)\right)  =e^{-itA\left(
\sigma ,D_{y}\right)  }\Op_{1}\left(  a_{\sigma}\left(
y,\eta-\left\{ {\frac {\sigma^{\alpha}}{h}}\right\}  \right)
\chi\left( \frac{\eta-\left\{ \frac{\sigma^{\alpha}}{h}\right\}
}{R}\right) \right)  e^{itA\left( \sigma,D_{y}\right)  }.
\]
\end{proof}

\paragraph{\emph{Third step: Passing to the limit}}

If $a\in\cS_{\Lambda}^{1}$ is compactly supported in $\eta$, the
map $\sigma\longmapsto\Op_{1}(a_{\sigma}(y,\eta))$ belongs to the
Banach space
$\mathcal{C}_{c}(I_{\Lambda};\cK(L^{2}(\T^{d},\Lambda)))$.\footnote{Here
$\mathcal{K}\left(  H\right)  $ denotes the space of compact
operators acting on a Hilbert space $H$.}
The dual of this space is $\mathcal{M}_{\Lambda}:=\cM(I_{\Lambda}%
;\cL^{1}(L^{2}(\T^{d},\Lambda)))$, the space of trace-class
operator valued measures. Let us consider the element
$\rho^{h}\in\mathcal{M}_{\Lambda}$, defined by letting
\[
\la\rho^{h},K\ra=\sum_{\sigma\in \sigma(h\Z^{d})}\la
U_{h}u_{h}(\sigma ,y),K\left(  \sigma\right)
U_{h}u_{h}(\sigma,y)\ra
\]
for all $K\in\mathcal{C}_{c}\left(  I_{\Lambda};\mathcal{K}\left(
L^{2}\left(  \mathbb{T}^{d},\Lambda\right)  \right)  \right) $.
Lemma~\ref{lem:plancherel} implies that $\left( \rho^{h}\right) $
is bounded in $\cM_{\Lambda}$ if $\left(  u_{h}\right)  $ is
bounded in $L^{2}(\T^{d})$. Besides, each $\rho^{h}$ is positive
(meaning that $\la\rho^{h},K\ra\geq0$ if $K\left(  \sigma\right)
\geq0$ for all $\sigma$). We consider $M(\sigma)$ a weak-$\ast$
limit of the family $\rho^{h}$ in $\mathcal{M}_{\Lambda}$ and we
now restrict our attention to symbols $a\in\cS_{\Lambda}^{1}$ that
do not depend on $\eta$. We write
\[
\chi\left(  {\frac{\eta-\{\sigma^{\alpha}/h\}}{R}}\right)
-\chi\left(
{\frac{\eta}{R}}\right)  =-{\frac{1}{R}}\int_{0}^{1}d\chi\left(  {\frac{\eta}{R}}-{\frac{s}{R}
}\left\{  {\frac{\sigma^{\alpha}}{h}}\right\}  \right) \cdot\left\{  {\frac{\sigma^{\alpha}}{h}
}\right\}   ds,
\]
and we obtain
\[
\Op_{1}\left(  b_{0}(t,y,\sigma,\eta)\right)  =e^{-itA\left(
\sigma ,D_{y}\right)  }\Op_{1}\left(
a_{\sigma}(y)\chi(\eta/R)\right)  e^{itA\left( \sigma,D_{y}\right)
}+O(1/R).
\]
whence
\[
\displaylines{\qquad \lim_{R\rightarrow +\infty}
\lim_{h\rightarrow 0} \int_{T^{\ast}\mathbb{T}^{d}}a_3(x,\xi,
\frac{\eta(\xi)}h)w_h(t, dx,d\xi)\hfill\cr\hfill
=\lim_{R\rightarrow +\infty}   \left\la M(\sigma),
e^{-itA(\sigma,D_{y}} \Op_1\left( a_\sigma(y)\chi(\eta/R)
\right)e^{itA(\sigma,D_{y}} \right\ra.\quad\cr}
\]

\begin{remark}
\label{rem:holonomy} The arguments used to get rid of the term
$\{{\frac{\sigma^{\alpha}}{h}}\}$ crucially exploit the
presence of a factor $1/R$ in front of it. Such an argument cannot be used for symbols $a\in{\mathcal{S}}
^{1}_{\Lambda}$ which depends non trivially of the variable $\eta$
as in Lemma~\ref{lem:16}. In such a situation, by working in
$L^{2}(\R^{d})$, one can define {\em locally} an operator-valued measure
$M$; however, this object cannot be globally defined on
the torus.
\end{remark}


\section{An iterative procedure for computing $\mu$}
\label{s:successive}

\subsection{First step of the construction} \label{s:firststep}

What was done in the previous section can be considered as the first step of
an iterative procedure that allows to effectively compute $\mu(t,\cdot)$
solely in terms of the sequence of initial data $\left(  u_{h}\right)  $.
Recall that we assumed in\S \ref{s:second}, without loss of generality, that the projection on $\xi$ of $\mu\left(
t,\cdot\right)  $ was supported in a ball contained in $\mathbb{R}
^{d}\setminus C_{H}$. We have decomposed this measure as a sum
\[
\mu(t,\cdot)=\sum_{\Lambda\in{\mathcal{L}}}\mu_{\Lambda}(t,\cdot
)+\sum_{\Lambda\in{\mathcal{L}}}\mu^{\Lambda}(t,\cdot),
\]
where $\Lambda$ runs over the set of primitive submodules of $\IZ^{d}$, and
where
\[
\mu_{\Lambda}(t,\cdot)=\int_{\left\langle \Lambda\right\rangle }\tilde{\mu
}_{\Lambda}(t,\cdot,d\eta)\rceil_{\IT^{d}\times R_{\Lambda}},\qquad
\mu^{\Lambda}(t,.)=\int_{\overline{\left\langle \Lambda\right\rangle }}
\tilde{\mu}^{\Lambda}(t,\cdot,d\eta)\rceil_{\IT^{d}\times R_{\Lambda}}.
\]
From Theorem \ref{Thm Properties}, the distributions $\tilde{\mu}_{\Lambda}$ have the following
properties~:\smallskip

\begin{enumerate}
\item $\tilde{\mu}_{\Lambda}(t,dx,d\xi,d\eta)$ is in $\mathcal{C}\left(
\IR; (\cS_\Lambda^1)'\right)
 $ and all its $x$-Fourier modes are in $\Lambda$; with respect to the variable $\xi$, $\tilde{\mu}_{\Lambda}
(t,dx,d\xi,d\eta)$ is supported in $I_{\Lambda}$;\smallskip\smallskip

\item if $\tau_{h}\ll1/h$ then for every $t\in\mathbb{R}$, $\tilde{\mu
}_{\Lambda}\left(  t,\cdot\right)  $ is a positive measure and:
\[
\tilde{\mu}_{\Lambda}\left(  t,\cdot\right)  =\left(  \tilde{\phi}_{t}
^{1}\right)  _{\ast}\tilde{\mu}_{\Lambda}\left(  0,\cdot\right)  ,
\]
where:
\[
\tilde{\phi}_{s}^{1}:(x,\xi,\eta)\longmapsto(x+sd^{2}H(\sigma(\xi))\eta
,\xi,\eta);
\]

\item if $\tau_{h}=1/h$ then$\int_{\left\langle \Lambda\right\rangle }
\tilde{\mu}_{\Lambda}(t,\cdot,d\eta)$ is in $\mathcal{C}(\IR;\cM_{+}(T^{\ast
}\IT^{d}))$ and $\int_{\IR^{d}\times\left\langle \Lambda\right\rangle }
\tilde{\mu}_{\Lambda}(t,\cdot,d\xi,d\eta)$ is an absolutely continuous measure
on $\IT^{d}$. In fact, with the notations of Section \ref{sec:mu_lambda}, we
have, for every $a\in\mathcal{C}_{c}^{\infty}\left(  T^{\ast}\IT^{d}\right)  $
with Fourier modes in $\Lambda$,
\[
\int_{\IT^{d}\times I_{\Lambda}\times\left\langle \Lambda\right\rangle
}a(x,\xi)\tilde{\mu}_{\Lambda}(t,dx,d\xi,d\eta)=\int_{I_{\Lambda}}\Tr\left(
a_{\sigma}e^{-i{\frac{t}{2}}d^{2}H(\sigma)D_{y}\cdot D_{y}}M(d\sigma
)e^{i{\frac{t}{2}}d^{2}H(\sigma)D_{y}\cdot D_{y}}\right)
\]
where $M\in{\mathcal{M}}_{+}\left(  I_{\Lambda};{\mathcal{L}}^{1}\left(
L^{2}(\T^{d},\Lambda)\right)  \right)  $ and $a_{\sigma}$ is the
multiplication operator by $a(\cdot,\sigma)$, acting on $L^{2}(\T^{d}
,\Lambda)$.
\end{enumerate}

On the other hand, the measures $\tilde{\mu}^{\Lambda}$ satisfy:\smallskip

\begin{enumerate}
\item for $a\in\cS_{\Lambda}^{1}$, $\la\tilde{\mu}^{\Lambda}(t,dx,d\xi
,d\eta),a(x,\xi,\eta)\ra$ is obtained as the limit of
\[
\left\langle w_{h,R,\delta}^{I_{\Lambda}}\left(  t\right)  ,a\right\rangle
=\int_{T^{\ast}\mathbb{T}^{d}}\chi\left(  \frac{\eta\left(  \xi\right)
}{\delta}\right)  \left(  1-\chi\left(  \frac{\tau_{h}\eta(\xi)}{R}\right)
\right)  a\left(  x,\xi,\tau_{h}\eta(\xi)\right)  w_{h}\left(  t\right)
\left(  dx,d\xi\right)  ,
\]
in the weak-$\ast$ topology of $L^{\infty}(\IR,\left(  \cS_{\Lambda}
^{1}\right)  ^{\prime})$, as $h\To0^{+}$, $R\To+\infty$ and then
$\delta\To0^{+}$ (possibly along subsequences);\smallskip

\item $\tilde{\mu}^{\Lambda}(t,dx,d\xi,d\eta)$ is in $L^{\infty}
(\IR,\cM_{+}(T^{\ast}\IT^{d}\times\la\Lambda\ra))$ and all its $x$-Fourier modes are
in $\Lambda$. With respect to the variable $\eta$, the measure $\tilde{\mu}^{\Lambda
}(t,dx,d\xi,d\eta)$ is $0$-homogeneous and supported at infinity~: we see it as a measure on the sphere at infinity $\mathbb{S}\la\Lambda\ra$. With respect to $\xi$ it is supported on $\{\xi\in
I_{\Lambda} \}$;\smallskip

\item $\tilde{\mu}^{\Lambda}$ is invariant by the two flows,
\[
\phi_{s}^{0}:(x,\xi,\eta)\longmapsto(x+sdH(\xi),\xi,\eta),\text{\quad
and\quad}\phi_{s}^{1}:(x,\xi,\eta)\longmapsto(x+sd^{2}H(\sigma(\xi))\frac
{\eta}{\left\vert \eta\right\vert },\xi,\eta).
\]

\end{enumerate}

This is the first step of an iterative procedure; the next step is
to decompose the measure $\mu^{\Lambda}(t,\cdot)$ according to
primitive submodules of $\Lambda$. We need to adapt the discussion
of~\cite{AnantharamanMacia}; to this aim, we introduce some
additional notation.

Fix a primitive submodule $\Lambda\subseteq\mathbb{Z}^{d}$ and $\sigma\in
I_{\Lambda}\setminus C_{H}$.
For $\Lambda_{2}\subseteq\Lambda_{1} \subseteq \Lambda$ primitive submodules of $(\Z^d)^*$,
for $\eta\in\la\Lambda_1\ra$, we denote
\begin{eqnarray*}
\Lambda_{\eta}\left(  \sigma, \Lambda_1\right)  &:=& \left(\Lambda_1^\perp \oplus \R\, d^2H(\sigma).\eta\right)^\perp
\cap (\Z^d)^*  \\
&=& \left( \R\, d^2H(\sigma).\eta\right)^\perp \cap \Lambda_1,
\end{eqnarray*}
where the orthogonal is always taken in the sense of duality. We note that $\Lambda_{\eta}\left(  \sigma, \Lambda_1\right) $ is a primitive submodule of $\Lambda_1$, and that the inclusion $\Lambda_{\eta}\left(  \sigma, \Lambda_1\right) \subset \Lambda_1$ is strict if $\eta\not= 0$ since $d^2H(\sigma)$ is definite.
We define:
\[
R_{\Lambda_{2}}^{\Lambda_{1}}(\sigma):= \{\eta\in \la\Lambda_1\ra, \Lambda_{\eta}\left(  \sigma, \Lambda_1\right)=\Lambda_2\}.
\]
Because $d^2H(\sigma)$ is definite, we have the decomposition $(\R^d)^*= (d^2H(\sigma).\Lambda_2)^\perp\oplus \la\Lambda_2\ra.$
We define $P^\sigma_{\Lambda_{2}}$ to be
the projection onto $ \la\Lambda_2\ra$ with respect to this decomposition.

\subsection{Step $k$ of the construction\label{s:stepk}}
In the following, we set $\Lambda=\Lambda_{1}$, corresponding to step $k=1$.
We now describe the outcome of our decomposition at step $k$ ($k\geq 1$); we will indicate
in \S \ref{s:rec} how to go from step $k$ to $k+1$, for $k\geq 1$.

At step $k$, we have decomposed $\mu(t,\cdot)$ as a sum
\[
\mu(t,\cdot)=\sum_{1\leq l\leq k}\sum_{\Lambda_{1}\supset\Lambda_{2}
\supset\ldots\supset\Lambda_{l}}\mu_{\Lambda_{l}}^{\Lambda_{1}\Lambda
_{2}\ldots\Lambda_{l-1}}(t,\cdot)+\sum_{\Lambda_{1}\supset\Lambda_{2}
\supset\ldots\supset\Lambda_{k}}\mu^{\Lambda_{1}\Lambda_{2}\ldots\Lambda_{k}
}(t,\cdot),
\]
where the sums run over the \emph{strictly decreasing} sequences of primitive
submodules of $(\IZ^{d})^*$ (of lengths $l\leq k$ in the first term, of length $k$
in the second term). We have
\[
\mu_{\Lambda_{l}}^{\Lambda_{1}\Lambda_{2}\ldots\Lambda_{l-1}}(t,x,\xi
)=\int_{R_{\Lambda_{2}}^{\Lambda_{1}}(\xi)\times\ldots\times R_{\Lambda_{l}
}^{\Lambda_{l-1}}(\xi)\times \la\Lambda_l\ra}\tilde{\mu}_{\Lambda_{l}}^{\Lambda
_{1}\Lambda_{2}\ldots\Lambda_{l-1}}(t,x,\xi,d\eta_{1},\ldots,d\eta_{l}
)\rceil_{\IT^{d}\times R_{\Lambda_{1}}},
\]
\[
\mu^{\Lambda_{1}\Lambda_{2}\ldots\Lambda_{k}}(t,x,\xi)=\int_{R_{\Lambda_{2}
}^{\Lambda_{1}}(\xi)\times\ldots\times R_{\Lambda_{k}}^{\Lambda_{k-1}}
(\xi)\times \mathbb{S}\la\Lambda_k\ra}\tilde{\mu}^{\Lambda_{1}\Lambda_{2}\ldots\Lambda_{k}
}(t,x,\xi,d\eta_{1},\ldots,d\eta_{k})\rceil_{\IT^{d}\times R_{\Lambda_{1}}}.
\]
The distributions $\tilde{\mu}_{\Lambda_{l}}^{\Lambda_{1}\Lambda_{2}
\ldots\Lambda_{l-1}}$ have the following properties~:

\begin{enumerate}
\item $\tilde{\mu}_{\Lambda_{l}}^{\Lambda_{1}\Lambda_{2}\ldots\Lambda_{l-1}}\in \mathcal{C}\left(  \IR,\mathcal{D}^{\prime}\left(  T^{\ast}\T^d\times
\mathbb{S}\la\Lambda_1\ra\times \ldots \times\mathbb{S}\la\Lambda_{l-1}\ra \times \la\Lambda_l\ra
 \right)  \right)  $ and all its $x$-Fourier
modes are in $\Lambda_{l}$; with respect to $\xi$ it is supported
in $I_{\Lambda_{1}}$;\smallskip
\item for every $t\in\mathbb{R}$, $\tilde{\mu
}_{\Lambda_{l}}^{\Lambda_{1}\Lambda_{2}\ldots\Lambda_{l-1}}(t,\cdot)$ is invariant under the flows $\phi_s^j$ ($j=0,1, \ldots, l-1$) defined by
$$\phi_s^0(x,\xi,\eta_{1},...,\eta_{l})=(x+sdH(\xi),\xi,\eta_{1},...,\eta_{l-1},\eta_{l});$$
$$\phi_s^j(x,\xi,\eta_{1},..., \eta_{l})=(x+sd^2H(\xi)\frac{\eta_j}{|\eta_j|},\xi,\eta_{1},..., \eta_{l});$$

\item if $\tau_{h}\ll1/h$ then for every $t\in\mathbb{R}$, $\tilde{\mu
}_{\Lambda_{l}}^{\Lambda_{1}\Lambda_{2}\ldots\Lambda_{l-1}}(t,\cdot)$ is a
positive measure and
\[
\tilde{\mu}_{\Lambda_{l}}^{\Lambda_{1}\Lambda_{2}\ldots\Lambda_{l-1}}
(t,\cdot)=\left(  \tilde{\phi}_{t}^{l}\right)  _{\ast}\tilde{\mu}_{\Lambda
_{l}}^{\Lambda_{1}\Lambda_{2}\ldots\Lambda_{l-1}}(0,\cdot),
\]
where, for $\left(  x,\xi,\eta_{1},..,\eta_{l}\right)  \in T^{\ast}\T^d\times
\mathbb{S}\la\Lambda_1\ra\times \ldots \times\mathbb{S}\la\Lambda_{l-1}\ra \times \la\Lambda_l\ra
  $ we define:
\[
\tilde{\phi}_{s}^{l}:(x,\xi,\eta_{1},...,\eta_{l})\longmapsto(x+sd^{2}
H(\xi)\eta_{l},\xi,\eta_{1},...,\eta_{l});
\]

\item if $\tau_{h}=1/h$ then $\int_{\la\Lambda_l\ra}\tilde{\mu}_{\Lambda_{l}}
^{\Lambda_{1}\Lambda_{2}\ldots\Lambda_{l-1}}(t,\cdot,d\eta_{l})$ is in
$\mathcal{C}(\IR,\cM_{+}(T^{\ast}\IT^{d}\times\mathbb{S}\la\Lambda_1\ra\times \ldots \times\mathbb{S}\la\Lambda_{l-1}\ra))$ and the measure
$\int_{(\R^d)^*\times\mathbb{S}\la\Lambda_1\ra\times \ldots \times\mathbb{S}\la\Lambda_{l-1}\ra \times \la\Lambda_l\ra}\tilde{\mu}_{\Lambda_{l}}^{\Lambda_{1}\Lambda_{2}
\ldots\Lambda_{l-1}}(t,\cdot,d\xi,d\eta_{1},\ldots,d\eta_{l})$ is an
absolutely continuous measure on $\IT^{d}$. In fact, for $a\in\mathcal{C}
_{c}^{\infty}\left(  T^{\ast}\IT^{d}\right)  $ has only
Fourier modes in $\Lambda_{l}$, it admits the expression
\begin{multline*}
\int_{T^{\ast}\T^d\times
\mathbb{S}\la\Lambda_1\ra\times \ldots \times\mathbb{S}\la\Lambda_{l-1}\ra \times \la\Lambda_l\ra}a(x,\xi)\tilde{\mu}
_{\Lambda_{l}}^{\Lambda_{1}\Lambda_{2}\ldots\Lambda_{l-1}}(t,dx,d\xi,d\eta
_{1},\ldots,d\eta_{l})=\\
\Tr\left(  \int_{I_{\Lambda_{1}}\times\mathbb{S}\la\Lambda_1\ra\times \ldots \times\mathbb{S}\la\Lambda_{l-1}\ra}a_{\sigma} e^{i{\frac{t}{2}}d^{2}H(\sigma)D_{y}\cdot D_{y}}\tilde{\rho
}_{\Lambda_{l}}^{\Lambda_{1}\Lambda_{2}\ldots\Lambda_{l-1}}\left(
d\sigma,d\eta_{1},\ldots,d\eta_{l-1}\right)  e^{-i{\frac{t}{2}}d^{2}
H(\sigma)D_{y}\cdot D_{y}}\right)  ,
\end{multline*}
where $a_\sigma$ is the multiplication operator defined in~(\ref{def:asigma}), $\tilde{\rho}_{\Lambda_{l}}^{\Lambda_{1}\Lambda_{2}\ldots\Lambda_{l-1}}$
is a positive operator valued measure on $I_{\Lambda_{1}}\times\mathbb{S}\la\Lambda_1\ra\times \ldots \times\mathbb{S}\la\Lambda_{l-1}\ra$ taking
values in
$\cL^{1}(L^{2}(\T^{d},\Lambda_{l}))$.
\end{enumerate}

On the other hand $\tilde{\mu}^{\Lambda_{1}\Lambda_{2}\ldots\Lambda_{k}}$ satisfy:

\begin{enumerate}
\item $\tilde{\mu}^{\Lambda_{1}\Lambda_{2}\ldots\Lambda_{k}}$ is in
$L^{\infty}(\IR,\cM_{+}(T^{\ast}\IT^{d}\times\mathbb{S}\la\Lambda_1\ra\times \ldots \times\mathbb{S}\la\Lambda_{k}\ra))$ and all its
$x$-Fourier modes are in $\Lambda_{k}$;\smallskip

\item $\tilde{\mu}^{\Lambda_{1}\Lambda_{2}\ldots\Lambda_{k}}$ is invariant by
the $k+1$ flows, $\phi_{s}^{0}:(x,\xi,\eta)\mapsto(x+sdH(\xi),\xi,\eta
_{1},\ldots,\eta_{k})$, and $\phi_{s}^{l}:(x,\xi,\eta_{1},\ldots,\eta
_{k})\longmapsto(x+sd^{2}H(\sigma(\xi))\frac{\eta_{l}}{\left\vert \eta
_{l}\right\vert },\xi,\eta_{1},\ldots,\eta_{k})$ (where $l=1,\ldots,k$).

\end{enumerate}
Finally, we define the space $\cS_{\Lambda_{k}}^{k}$ which is the
class of smooth functions $a(x,\xi ,\eta_{1},\ldots,\eta_{k})$ on
$T^{\ast}\IT^{d}\times \la \Lambda_1\ra\times\ldots\times
\la\Lambda_k\ra$ that are

\begin{enumerate}
\item[(i)] smooth and compactly supported in $(x, \xi)\in T^{*}\IT^{d}$;

\item[(ii)] homogeneous of degree $0$ at infinity in each variable $\eta_{1},
\ldots, \eta_{k}$;

\item[(iii)] such that their non-vanishing $x$-Fourier coefficients correspond to
frequencies in $\Lambda_{k}$.
\end{enumerate}

\subsection{From step $k$ to step $k+1$ ($k\geq 1$)\label{s:rec}}

After step $k$, we leave untouched
the term $\sum_{1\leq l\leq k}
\sum_{\Lambda_{1}\supset\Lambda_{2}\supset\ldots\supset\Lambda_{l}}
\mu_{\Lambda_{l}}^{\Lambda_{1}\Lambda_{2}\ldots\Lambda_{l-1}}$ and decompose further  $\sum_{\Lambda
_{1}\supset\Lambda_{2}\supset\ldots\supset\Lambda_{k}}\mu^{\Lambda_{1}
\Lambda_{2}\ldots\Lambda_{k}}$. Using the positivity of $\tilde{\mu}^{\Lambda_{1}
\Lambda_{2}\ldots\Lambda_{k}}$, we use the procedure described in Section
\ref{s:decompo} to write
\begin{equation}\label{restriction}
\tilde{\mu}^{\Lambda_{1}\Lambda_{2}\ldots\Lambda_{k}}(\sigma,\cdot
)=\sum_{\Lambda_{k+1}\subset\Lambda_{k}}\tilde{\mu}^{\Lambda_{1}\Lambda
_{2}\ldots\Lambda_{k}}\rceil_{\eta_{k}\in R_{\Lambda_{k+1}}^{\Lambda_{k}
}(\sigma)},
\end{equation}
where the sum runs over all primitive submodules $\Lambda_{k+1}$ of
$\Lambda_{k}$. Moreover, by Proposition~\ref{prop:decomposition}, all the
$x$-Fourier modes of $\tilde{\mu}^{\Lambda_{1}\Lambda_{2}\ldots\Lambda_{k}
}\rceil_{\eta_{k}\in R_{\Lambda_{k+1}}^{\Lambda_{k}
}(\sigma)}$ are in $\Lambda_{k+1}$.
To generalize the analysis of Section \ref{s:second}, we consider
test functions in $\cS_{\Lambda_{k+1}}^{k+1}$. We
let
$$\displaylines{\qquad
w_{h,R_{1},\ldots,R_{k+1}}^{\Lambda_{1}\Lambda_{2}\ldots
\Lambda_{k+1}}\left(  t,x,\xi,\eta_1,\cdots,\eta_{k+1}\right)
:=
\left(  1-\chi\left(  {\eta_{k+1}\over R_{k+1}} \right)  \right)  \hfill\cr\hfill \times\,
 w^{\Lambda_1\Lambda_2\cdots\Lambda_k}_{h,R_1,\cdots,R_k}(t,x,\xi,\eta_1,\cdots,\eta_k)\otimes\delta_{ P^{\xi}_{\Lambda_{k+1}}(\eta_k)}(\eta_{k+1})   ,\qquad\cr}
 $$
and
$$\displaylines{\qquad
 w_{\Lambda_{k+1}h,R_{1},\ldots,R_{k+1}}^{\Lambda_{1}\Lambda_{2}\ldots
\Lambda_{k}}\left(  t,x,\xi,\eta_1,\cdots,\eta_{k+1}\right)
:=
\chi\left(  { \eta_{k+1}\over R_{k+1}}\right)    \hfill\cr\hfill \times\,
 w^{\Lambda_1\Lambda_2\cdots\Lambda_k}_{h,R_1,\cdots,R_k}(t,x,\xi,\eta_1,\cdots,\eta_k)\otimes\delta_{ P^{\xi}_{\Lambda_{k+1}}(\eta_k)}(\eta_{k+1} )  .\qquad\cr}
$$

By the Calder\'{o}n-Vaillancourt theorem, both $w_{\Lambda_{k+1}
,h,R_{1},\ldots,R_{k}}^{\Lambda_{1}\Lambda_{2}\ldots\Lambda_{k}}$ and
$w_{h,R_{1},\ldots,R_{k}}^{\Lambda_{1}\Lambda_{2}\ldots\Lambda_{k+1}}$ are
bounded in $L^{\infty}(\IR,(\cS_{\Lambda_{k+1}}^{k+1})^{\prime}).$ After
possibly extracting subsequences, we can take the following limits~:
\[
\lim_{R_{k+1}\To+\infty}\cdots\lim_{R_{1}\To+\infty}\lim_{h\To0}\left\langle
w_{h,R_{1},\ldots,R_{k}}^{\Lambda_{1}\Lambda_{2}\ldots\Lambda_{k+1}}\left(
t\right)  ,a\right\rangle =:\left\langle \tilde{\mu}^{\Lambda_{1}\Lambda
_{2}\ldots\Lambda_{k+1}}(t),a\right\rangle ,
\]
and
\[
\lim_{R_{k+1}\To+\infty}\cdots\lim_{R_{1}\To+\infty}\lim_{h\To0}\left\langle
w_{\Lambda_{k+1},h,R_{1},\ldots,R_{k}}^{\Lambda_{1}\Lambda_{2}\ldots
\Lambda_{k}}\left(  t\right)  ,a\right\rangle =:\left\langle \tilde{\mu
}_{\Lambda_{k+1}}^{\Lambda_{1}\Lambda_{2}\ldots\Lambda_{k}}(t),a\right\rangle
.
\]
Then the properties listed in the preceeding subsection  are  a direct generalisation of Theorems~\ref{mu^Lambda} and~\ref{Thm Properties} (see also~\cite{AnantharamanMacia}, Section 4) and writing
\begin{eqnarray}\label{decomposition}
\widetilde\mu^{\Lambda_1\Lambda_2\cdots \Lambda_k}(t,.)\rceil_{\eta_{k}\in R_{\Lambda_{k+1}}^{\Lambda_{k}
}(\sigma)} & = &
\int_{\langle \Lambda_{k+1}\rangle}
\widetilde \mu^{\Lambda_1\Lambda_2\cdots \Lambda_{k+1}}(t,.,d\eta_{k+1})
\rceil_{\eta_{k}\in R_{\Lambda_{k+1}}^{\Lambda_{k}}(\sigma)} \\
\nonumber &  & +
\int_{\langle \Lambda_{k+1}\rangle}
\widetilde \mu^{\Lambda_1\Lambda_2\cdots \Lambda_{k}}_{\Lambda_{k+1}}(t,.,d\eta_{k+1})
\rceil_{\eta_{k}\in R_{\Lambda_{k+1}}^{\Lambda_{k}}(\sigma)} .
\end{eqnarray}

\begin{remark}
By construction, if $\Lambda_{k+1}=\{0\}$, we have $\tilde\mu^{\Lambda
_{1}\Lambda_{2}\ldots\Lambda_{k+1}}=0$, and the induction stops. Similarly to
Remark \ref{r:nice}, one can also see that if $\operatorname*{rk}
\Lambda_{k+1}=1$, the invariance properties of $\tilde\mu^{\Lambda_{1}
\Lambda_{2}\ldots\Lambda_{k+1}}$ imply that it is constant in $x$.
\end{remark}

\begin{remark}
Note that in the preceding definition of $k$-microlocal Wigner transform for $k\geq 1$, we did not use a parameter $\delta$ tending to $0$ as we did when $k=0$ in order to isolate the part of the limiting measures supported above $R^{\Lambda_k}_{\Lambda_{k+1}}(\sigma)$. This comes directly from the restrictions made in~(\ref{restriction}) and~(\ref{decomposition}).
\end{remark}

\subsection{Proof of  Theorem \ref{t:precise}}

This iterative procedure allows to decompose $\mu$ along decreasing sequences of submodules. In particular, when $\tau_h\sim 1/h$, it implies Theorem~\ref{t:precise}.
 Indeed, to end the proof of Theorem
\ref{t:precise}, we let
\begin{eqnarray*}
\mu_{\Lambda}(t,\cdot)&=&\sum_{0\leq k\leq d}\sum_{\Lambda_{1}\supset
\Lambda_{2}\supset\cdots\supset\Lambda_{k}\supset\Lambda}
\mu_{\Lambda}^{\Lambda_{1}\Lambda_{2}\ldots\Lambda_{k}}(t,\cdot)\\
&=& \sum_{0\leq k\leq d}\sum_{\Lambda_{1}\supset
\Lambda_{2}\supset\cdots\supset\Lambda_{k}\supset\Lambda}\int_{ R_{\Lambda_{2}}^{\Lambda_{1}}(\xi)\times\ldots\times R_{\Lambda_{}
}^{\Lambda_{k}}(\xi)\times \la\Lambda\ra}
\tilde{\mu}_{\Lambda_{}}^{\Lambda
_{1}\Lambda_{2}\ldots\Lambda_{k}}(t,\cdot,d\eta_{1},\ldots,d\eta_{k}
)\rceil_{\T^d\times R_{\Lambda_1}},
\end{eqnarray*}
where $\Lambda_{1},\ldots,\Lambda_{k}$ run over the set of strictly decreasing
sequences of submodules ending with $\Lambda$. We know that $\mu_{\Lambda}^{\Lambda_{1}\Lambda_{2}\ldots\Lambda_{k}}$ is supported on $\{\xi \in I_{\Lambda_1}\}$, and since $\Lambda
\subset \Lambda_1$ we have  $ I_{\Lambda_1}\subset I_{\Lambda}.$
We also let
\[
\rho_{\Lambda}(\sigma)=\sum_{0\leq k\leq d}\sum_{\Lambda_{1}\supset\Lambda
_{2}\supset\cdots\supset\Lambda_{k}\supset\Lambda}\int_{R_{\Lambda_{2}}^{\Lambda_{1}}(\xi)\times\ldots\times R_{\Lambda_{}
}^{\Lambda_{k}}(\xi) }\tilde{\rho}_{\Lambda
}^{\Lambda_{1}\Lambda_{2}\ldots\Lambda_{k}}\left(  \sigma,d\eta_{1},\ldots
,d\eta_{k}\right) \rceil_{\sigma\in R_{\Lambda_1}} ,
\]
where the $\tilde{\rho}_{\Lambda}^{\Lambda_{1}\Lambda_{2}\ldots\Lambda_{k}}$
are the operator-valued measures appearing in \S\ref{s:stepk}.

As already mentioned, Theorem~\ref{t:precise} implies Theorem \ref{t:main} in the case $\tau_h\sim 1/h$. The proof of Theorem~\ref{t:main} in the case $\tau_h \ll 1/h$ is discussed in Section \ref{s:halpha} and in the case $\tau_h\gg 1/h$, in Section \ref{s:hierarchy}.

\section{Some examples of singular concentration\label{s:halpha}}

\subsection{Singular concentration for time scales $\tau_{h}\ll1/h$}

In this section, we focus on the case $\tau_{h}\ll1/h$ and prove
Theorem~\ref{t:main}(1).

Consider $\rho\in\cS\left(  \mathbb{R}^{d}\right)  $ with $\left\Vert \rho\right\Vert
_{L^{2}\left(  \mathbb{R}^{d}\right)  }=1$ and such that the Fourier transform $\widehat{\rho}$ is compactly supported. Let $\left(  x_{0},\xi_{0}\right)
\in\mathbb{R}^{d}\times\mathbb{R}^{d}$ and $\left(  \varepsilon_{h}\right)  $
a sequence of positive real numbers that tends to zero as $h\longrightarrow
0^{+}$. Form the wave-packet:%
\begin{equation}
v_{h}\left(  x\right)  :=\frac{1}{\left(  \varepsilon_{h}\right)  ^{d/2}}%
\rho\left(  \frac{x-x_{0}}{\varepsilon_{h}}\right)  e^{i\frac{\xi_{0}}{h}\cdot
x}. \label{e:defv_h}%
\end{equation}
Define
\[
u_{h}:=\mathbf{P}v_{h},
\]
where $\mathbf{P}$ denotes the periodization operator $\mathbf{P}v\left(
x\right)  :=\sum_{k\in\mathbb{Z}^{d}}v\left(  x+2\pi k\right)  $. Since $\rho$ is rapidly decrasing, we have
$\left\Vert u_{h}\right\Vert _{L^{2}\left(  \mathbb{T}^{d}\right)
}\Lim_{h\To 0} 1$. It
is not hard to check that $\left(  u_{h}\right)  $ is $h$-oscillatory.

Theorem \ref{t:main}(1) is a consequence of our next result.

\begin{proposition}
\label{prop:pWP}Let $\left(  \tau_{h}\right)  $ be such that $\lim
_{h\rightarrow0^{+}}h\tau_{h}=0$; suppose that $\varepsilon_{h}\gg h\tau_{h}$.
Then the Wigner distributions of the solutions $S_{h}^{\tau_{h}t}u_{h}$
converge weakly-$\ast$ in $L^{\infty}\left(  \mathbb{R};\mathcal{D}^{\prime
}\left(  T^{\ast}\mathbb{T}^{d}\right)  \right)  $ to $\mu_{\left(  x_{0}%
,\xi_{0}\right)  }$, defined by:%
\begin{equation}
\int_{T^{\ast}\mathbb{T}^{d}}a\left(  x,\xi\right)  \mu_{\left(  x_{0},\xi
_{0}\right)  }\left(  dx,d\xi\right)  =\lim_{T\rightarrow\infty}\frac{1}%
{T}\int_{0}^{T}a\left(  x_{0}+tdH\left(  \xi_{0}\right)  ,\xi_{0}\right)
dt,\quad\forall a\in\cC_{c}(T^{\ast}\T^{d}). \label{e:orbitm}%
\end{equation}

\end{proposition}

\begin{proof}
Start noticing that the sequence $\left(  u_{h}\right)  $ has the unique
semiclassical measure $\mu_{0}=\delta_{x_{0}}\otimes\delta_{\xi_{0}}$. Using
property (4) in the appendix, we deduce that the image $\overline{\mu}$ of
$\mu\left(  t,\cdot\right)  $ by the projection from $\mathbb{T}^{d}%
\times\mathbb{R}^{d}$ onto $\mathbb{R}^{d}$ satisfies:%
\[
\overline{\mu}=\sum_{\Lambda\in\mathcal{L}}\overline{\mu_{\Lambda}}%
=\delta_{\xi_{0}}.
\]
Since for every primitive module $\Lambda\subset\mathbb{Z}^{d}$ the positive
measure $\overline{\mu_{\Lambda}}$ is supported on $R_{\Lambda}$, and these
sets form a partition of $\mathbb{R}^{d}$, we conclude that $\overline
{\mu_{\Lambda}}=0$ unless $\Lambda=\Lambda_{\xi_{0}}$ and therefore $\mu
=\mu_{\Lambda_{\xi_{0}}}$. Therefore, in order to characterize $\mu$ it
suffices to test it against symbols with Fourier coefficients in $\Lambda
_{\xi_{0}}$. Let $a\in\cC_{c}^{\infty}\left(  T^{\ast}\mathbb{T}^{d}\right)  $
be such a symbol; we can restrict our attention to the case where $a$ is a trigonometric polynomial in $x$. Let $\varphi\in L^{1}\left(  \mathbb{R}\right)  $.
Recall that, by Lemma \ref{Lemma Inv}, the Wigner distributions $w_{h}\left(
t\right)  $ of $S_{h}^{\tau_{h}t}u_{h}$ satisfy%
\[
\int_{\mathbb{R}}\varphi\left(  t\right)  \left\langle w_{h}\left(  t\right)
,a\right\rangle dt=\int_{\mathbb{R}}\varphi\left(  t\right)  \left\langle
w_{h}\left(  0\right)  ,a\circ\phi_{\tau_{h}t}\right\rangle dt+o\left(
1\right)  ;
\]
moreover the Poisson summation formula ensures that the Fourier coefficients
of $u_{h}$ are given by:%
\[
\widehat{u_{h}}\left(  k\right)  =\frac{\left(  \varepsilon_{h}\right)
^{d/2}}{\left(  2\pi\right)  ^{d/2}}\widehat{\rho}\left(  \frac{\varepsilon
_{h}}{h}\left(  hk-\xi_{0}\right)  \right)  e^{-i\left(  k-\xi
_{0}/h\right)  \cdot x_{0}}.
\]
Combining this with the explicit formula (\ref{e:Weylq}) for the Wigner
distribution presented in the appendix we get:%
\begin{equation}%
\begin{split}
\int_{\mathbb{R}}\varphi\left(  t\right)  \left\langle w_{h}\left(  t\right)
,a\right\rangle dt=\frac{\left(  \varepsilon_{h}\right)  ^{d}}{\left(
2\pi\right)  ^{3d/2}}  &  \sum_{k-j\in\Lambda_{\xi_{0}}}\widehat{\varphi
}\left(  \tau_{h}dH\left(  h\frac{k+j}{2}\right)  \cdot\left(  k-j\right)
\right)  \widehat{a}_{j-k}\left(  h\frac{k+j}{2}\right) \\
&  \widehat{\rho}\left(  \frac{\varepsilon_{h}}{h}\left(  hk-\xi_{0}\right)
\right)  \overline{\widehat{\rho}\left(  \frac{\varepsilon_{h}}{h}\left(
hj-\xi_{0}\right)  \right)  }e^{-i\left(  k-j\right)  \cdot x_{0}}+o\left(
1\right)  .
\end{split}
\label{e:wignWP}%
\end{equation}
Now, since $k-j\in\Lambda_{\xi_{0}}$ we can write:%
\begin{align*}
\left\vert dH\left(  h\frac{k+j}{2}\right)  \cdot\left(  k-j\right)
\right\vert  &  =\left\vert \left[  dH\left(  h\frac{k+j}{2}\right)
-dH\left(  \xi_{0}\right)  \right]  \cdot\left(  k-j\right)  \right\vert \\
&  \leq C\left\vert h\frac{k+j}{2}-\xi_{0}\right\vert \left\vert
k-j\right\vert .
\end{align*}
By hypothesis, both $\widehat{\rho}$ and $k\longmapsto\widehat{a_{k}}\left(
\xi\right)  $ are compactly supported, and hence the sum (\ref{e:wignWP}) only involves terms satisfying:%
\[
\frac{\varepsilon_{h}}{h}\left\vert h\frac{k}{2}-\xi_{0}\right\vert \leq
R,\;\frac{\varepsilon_{h}}{h}\left\vert h\frac{j}{2}-\xi_{0}\right\vert \leq
R\text{ and }\left\vert j-k\right\vert \leq R
\]
for some fixed $R$.
This in turn implies%
\[
\left\vert \tau_{h}dH\left(  h\frac{k+j}{2}\right)  \cdot\left(  k-j\right)
\right\vert \leq CR^{2}\frac{\tau_{h}h}{\varepsilon_{h}}.
\]
This shows that the limit of (\ref{e:wignWP}) as $h\longrightarrow0^{+}$
coincides with that of:%
\begin{align*}
&  \frac{\left(  \varepsilon_{h}\right)  ^{d}}{\left(  2\pi\right)  ^{3d/2}%
}\sum_{k-j\in\Lambda_{\xi_{0}}}\widehat{\varphi}\left(  0\right)
a_{j-k}\left(  h\frac{k+j}{2}\right)  \widehat{\rho}\left(  \frac
{\varepsilon_{h}}{h}\left(  hk-\xi_{0}\right)  \right)  \overline
{\widehat{\rho}\left(  \frac{\varepsilon_{h}}{h}\left(  hj-\xi_{0}\right)
\right)  }e^{-i\left(  k-j\right)  \cdot x_{0}}\\
&  =\widehat{\varphi}\left(  0\right)  \left\langle w_{h}\left(  0\right)
,a\right\rangle ,
\end{align*}
which is precisely:%
\[
\widehat{\varphi}\left(  0\right)  a\left(  x_{0},\xi_{0}\right)
=\widehat{\varphi}\left(  0\right)  \lim_{T\rightarrow\infty}\frac{1}{T}%
\int_{0}^{T}a\left(  x_{0}+tdH\left(  \xi_{0}\right)  ,\xi_{0}\right)  dt,
\]
since $a$ has only Fourier modes in $\Lambda_{\xi_{0}}$.
\end{proof}

We next present a slight modification of the previous example in order to
illustrate the two-microlocal nature of the elements of $\widetilde
{\mathcal{M}}\left(  \tau\right)  $. Define now, for $\eta_{0}\in
\mathbb{R}^{d}$:%
\[
u_{h}\left(  x\right)  =\mathbf{P}\left[  v_{h}\left(  x\right)  e^{i\eta
_{0}/\left(  h\tau_{h}\right)  }\right]  ,
\]
where $v_{h}$ was defined in (\ref{e:defv_h}).

\begin{proposition}
\label{prop:Diracmasses} Suppose that $\lim_{h\rightarrow0^{+}}h\tau_{h}=0$
and $\varepsilon_{h}\gg h\tau_{h}$. Suppose moreover that $d^{2}H\left(
\xi_{0}\right)  $ is definite and that $\eta_{0}\in\left\langle \Lambda
_{\xi_{0}}\right\rangle $. Then the Wigner distributions of $S_{h}^{\tau_{h}
t}u_{h}$ converge weakly-$\ast$ in $L^{\infty}\left(  \mathbb{R}
;\mathcal{D}^{\prime}\left(  T^{\ast}\mathbb{T}^{d}\right)  \right)  $ to the
measure:
\[
\mu\left(  t,\cdot\right)  =\mu_{\left(  x_{0}+td^{2}H\left(  \xi_{0}\right)
\eta_{0},\xi_{0}\right)  },\qquad t\in\mathbb{R},
\]
where $\mu_{\left(  x_{0},\xi_{0}\right)  }$ is the uniform orbit measure
defined in (\ref{e:orbitm}).
\end{proposition}

\begin{proof}
The same argumenta we used in the proof of Proposition \ref{prop:pWP} gives
$\mu=\mu_{\Lambda_{\xi_{0}}}$. We claim that $w_{I_{\Lambda_{\xi_{0}}}
,h,R}\left(  0\right)  $ converges to the measure:
\[
\widetilde{\mu}_{\Lambda_{\xi_{0}}}\left(  0,x,\xi,\eta\right)  =\mu_{\left(
x_{0},\xi_{0}\right)  }\left(  x,\xi\right)  \delta_{\eta_{0}}\left(
\eta\right)  .
\]
Assume this is the case, then Proposition \ref{prop:propagation} implies:
\[
\widetilde{\mu}_{\Lambda_{\xi_{0}}}\left(  t,x,\xi,\eta\right)  =\mu_{\left(
x_{0}+td^{2}H\left(  \xi_{0}\right)  \eta_{0},\xi_{0}\right)  }\left(
x,\xi\right)  \delta_{\eta_{0}}\left(  \eta\right)  ,\quad\forall
t\in\mathbb{R},
\]
and, since $\widetilde{\mu}_{\Lambda_{\xi_{0}}}\left(  t,\cdot\right)  $ are
probability measures, it follows from Proposition \ref{p:decomposition} that
$\widetilde{\mu}^{\Lambda_{\xi_{0}}}=0$ and:
\[
\mu_{\Lambda_{\xi_{0}}}\left(  t,\cdot\right)  =\int_{\left\langle
\Lambda_{\xi_{0}}\right\rangle }\widetilde{\mu}^{\Lambda_{\xi_{0}}}\left(
t,\cdot,d\eta\right)  =\mu_{\left(  x_{0}+td^{2}H\left(  \xi_{0}\right)
\eta_{0},\xi_{0}\right)  }.
\]
Let us now prove the claim. Set
\[
\tilde{u}_{h}(x)=v_{h}\left(  x\right)  e^{i\eta_{0}/\left(  h\tau_{h}\right)
}.
\]
Consider $h_{0}>0$ and $\chi\in{\mathcal{C}}_{0}^{\infty}(\R^{d})$ such that
$\chi\tilde{u}_{h}=\tilde{u}_{h}$ for all $h\in(0,h_{0})$ and $\mathbf{P}
\chi^{2}\equiv1$. We now take $a\in{\mathcal{S}}_{\Lambda}^{1}$ and denote by
$\tilde{a}$ the smooth compactly supported function defined on $\R^{d}$ by
$\tilde{a}=\chi^{2}a$. Using the fact that the two-scale quantization admits
the gain $h\tau_{h}$ (see Remark~\ref{rem:gain}),
\begin{align*}
\langle u_{h}\;,\;\Op_{h}^{\Lambda_{\xi_{0}}}(a)u_{h}\rangle_{L^{2}(\T^{d})}
&  =\langle u_{h}\;,\;\Op_{h}^{\Lambda_{\xi_{0}}}(\tilde{a})u_{h}
\rangle_{L^{2}(\R^{d})}\\
&  =\langle\tilde{u}_{h}\;,\;\Op_{h}^{\Lambda_{\xi_{0}}}(a)\tilde{u}
_{h}\rangle_{L^{2}(\R^{d})}+O(h\tau_{h}).
\end{align*}
Therefore, it is possible to lift the computation of the limit of
$w_{I_{\Lambda_{\xi_{0}}},h,R}\left(  0\right)  $ to $T^{\ast}\R^{d}
\times\left\langle \Lambda_{\xi_{0}}\right\rangle $ and, in consequence,
replace sums by integrals. A direct computation gives:
$$\displaylines{\qquad
\langle\tilde{u}_{h},\Op_{h}^{\Lambda_{\xi_{0}}}(a)\tilde{u}_{h}\rangle
_{L^{2}(\R^{d})}   =(2\pi)^{-d}\int_{\mathbb{R}^{3d}}
\mathrm{e}^{i\xi\cdot\left(  x-y\right)  }\overline{\rho}
(x)\rho(y)\hfill\cr\hfill\times
a\left(  x_{0}
+\varepsilon_{h}\frac{x+y}{2},\xi_{0}+\frac{1}{\tau_{h}}\eta_{0}+{\frac
{h}{\varepsilon_{h}}}\xi,\tau_{h}\eta(\xi_{0}+\frac{1}{\tau_{h}}\eta
_{0}+{\frac{h}{\varepsilon_{h}}}\xi)\right)
dxdyd\xi.\qquad
\cr}$$
Note that if $F(\xi)=(\sigma,\eta)$, then
\[
\forall k\in\Lambda,\;\;F(\xi+k)=(\sigma,\eta+k)=F(\xi)+(0,k),
\]
which implies that $dF(\xi)k=(0,k)$ and $d\eta(\xi)k=k$ for all $k\in
\Lambda_{\xi_{0}}$. We deduce $d\eta(\xi_{0})\eta_{0}=\eta_{0}$ since
$\eta_{0}\in\langle\Lambda_{\xi_{0}}\rangle$ and, in view of $\eta(\xi_{0}
)=0$, a Taylor expansion of $\eta(\xi)$ around $\xi_{0}$ gives
\[
\tau_{h}\eta\left(  \xi_{0}+\frac{1}{\tau_{h}}\eta_{0}+{\frac{h}
{\varepsilon_{h}}}\xi\right)  =\eta_{0}+o(1).
\]
Therefore, as $h$ goes to $0$,
\[
\langle\tilde{u}_{h},\Op_{h}^{\Lambda_{{\xi_{0}}}}(a)\tilde{u}_{h}
\rangle\rightarrow a(x_{0},\xi_{0},\eta_{0})=\langle\widetilde{\mu}
_{\Lambda_{\xi_{0}}},a\rangle.
\]

\end{proof}


\subsection{Singular concentration for Hamiltonians with critical points}

We next show by a quasimode construction that for Hamiltonians having a
degenerate critical point (of order $k>2$) and for time scales $\tau_{h}\ll1/h^{k-1}$, the
set $\widetilde{\mathcal{M}}\left(  \tau\right)  $ always contains singular measures.

Suppose $\xi_{0}\in\mathbb{R}^{d}$ is such that:%
\[
dH\left(  \xi_{0}\right)  ,d^{2}H\left(  \xi_{0}\right)  ,...,d^{k-1}H\left(
\xi_{0}\right)  \quad\text{vanish identically.}%
\]
The Hamiltonian $H(\xi)=|\xi|^k$ ($k$ an even integer -- corresponding to the operator $(-\Delta)^{\frac{k}{2}}$) provides such an example
(with $\xi_{0}=0$). Let $u_{h}=\mathbf{P}v_{h}$, where $v_{h}$ is defined in
(\ref{e:defv_h}). If $\varepsilon_{h}\gg h$ it is not hard to see that
\[
\left\Vert H\left(  hD_{x}\right)  u_{h}-H\left(  \xi_{0}\right)
u_{h}\right\Vert _{L^{2}\left(  \mathbb{T}^{d}\right)  }=O\left(
h^{k}/\left(  \varepsilon_{h}\right)  ^{k}\right)  .
\]
Therefore,%
\[
\left\Vert S_{h}^{t}u_{h}-e^{-i\frac{t}{h}H\left(  \xi_{0}\right)  }%
u_{h}\right\Vert _{L^{2}\left(  \mathbb{T}^{d}\right)  }=tO\left(
h^{k-1}/\left(  \varepsilon_{h}\right)  ^{k}\right)  ,
\]
and, it follows that, for compactly supported $\varphi\in L^{1}(\R)$ and $a\in
\mathcal{C}_{c}^{\infty}\left(  T^{\ast}\mathbb{T}^{d}\right)  $,
\[
\int_{\mathbb{R}}\varphi(t)\langle w_{h}(t)\;,\;a\rangle dt=\int_{\mathbb{R}%
}\varphi(t)\left\langle u_{h},\Op_{h}(a)u_{h}\right\rangle _{L^{2}%
(\mathbb{T}^{d})}dt+O\left(  \tau_{h}h^{k-1}/\left(  \varepsilon_{h}\right)
^{k}\right)  .
\]
Choosing $\left(  \varepsilon_{h}\right)  $ tending to zero and such that
$\varepsilon_{h}\gg\left(  \tau_{h}h^{k-1}\right)  ^{1/k },$
the latter quantity converges to $a(x_{0},\xi_{0})\Vert\varphi\Vert
_{L^{1}(\R)}$ as $h\To0^{+}$. In other words,
\[
dt\otimes\delta_{x_{0}}\otimes\delta_{\xi_{0}}\in\widetilde{\mathcal{M}}%
(\tau),
\]
whence $dt\otimes\delta_{x_{0}}\in{\mathcal{M}}(\tau).$


\begin{remark}In the special case of $H(\xi)=|\xi|^k$ ($k$ an even integer), we know that the threshold $\tau_h^H$ is precisely $h^{1-k}$. From the discussion
of \S \ref{s:hierarchy} and previously known results about eigenfunctions of the laplacian, we know that the elements of ${\mathcal{M}}(\tau)$ are absolutely continuous for $\tau_h\gg 1/h^{k-1}$. In the case of $\tau_h= 1/h^{k-1}$, one can still show that elements of ${\mathcal{M}}(\tau)$ are absolutely continuous. This requires some extra work which consists in checking that all our proofs still work in this case for $\tau_h= 1/h^{k-1}$ and $\xi$ in a neighbourhood of $\xi_0=0$, replacing the Hessian $d^2H(\xi_0)$ by $d^k H(\xi_0)$,
and the assumption that the Hessian is definite by the remark that $\left[d^k H(\xi_0).\xi^k=0 \Longrightarrow \xi=0\right]$.

In the general case of a Hamiltonian having a degenerate critical point, the existence of such a threshold, and its explicit determination, is by no means obvious.
\end{remark}

\subsection{The effect of the presence of a subprincipal symbol}

Here we present some remarks concerning how the preceding results may change
when the Hamiltonian $H(hD_{x})$ is perturbed by a small potential $h^{\beta
}V(t,x)$. Suppose $V\in L^{\infty}\left(  \mathbb{R}\times\mathbb{T}%
^{d}\right)  $ and define%
\[
P_{\beta,h}:=H(hD_{x})+h^{\beta}V(t,x),\;\text{with}\;\beta>0\text{,}%
\]
and denote by $S_{\beta,h}^{t}$ the corresponding propagator (starting at
$t=0$):%
\[
S_{\beta,h}^{t}:=\mathrm{e}^{-i{\frac{t}{h}}P_{\beta,h}}.
\]
Let us fix a time scale $\tau=\left(  \tau_{h}\right)  $ that tends to
infinity as $h\longrightarrow0^{+}$. Define $\widetilde{\mathcal{M}}_{\beta
,V}\left(  \tau\right)  $ to be the set of accumulation points of the
time-scaled Wigner distributions
\[
w_{h}^{\beta}(t,\cdot)=w_{S_{\beta,h}^{\tau_{h}t}u_{h}}^{h},
\]
as $\left(  u_{h}\right)  $ varies among all normalised sequences in
$L^{2}(\T^{d})$. For the sake of simplicity, from now on we shall fix the time
scale $\tau_{h}=1/h$. The discussion that follows can be easily adapted to
more general time scales by changing the ranges of values of $\beta$.

\begin{enumerate}
\item $\beta>2$. In this case it can be easily shown that $w_{h}^{\beta}$ and
$w_{h}$ have the same weak-$\ast$ accumulation points in $L^{\infty}\left(
\mathbb{R};\mathcal{D}^{\prime}\left(  T^{\ast}\mathbb{T}^{d}\right)  \right)
$. Therefore, the potential is a negligible perturbation and, in particular,
for every $V\in L^{\infty}\left(  \mathbb{R}\times\mathbb{T}^{d}\right)  $,%
\[
\widetilde{\mathcal{M}}_{\beta,V}\left(  1/h\right)  =\widetilde{\mathcal{M}%
}\left(  1/h\right)  .
\]

\item $\beta=2$. When $H(\xi)=|\xi|^{2}$, the question has been addressed
in~\cite{AnantharamanMacia}.\footnote{In that work, it is assumed that the set
of discontinuity points of $V$ has measure zero.} It turns out that whenever
$V$ is not constant,%
\[
\widetilde{\mathcal{M}}_{2,V}\left(  1/h\right)  \neq\widetilde{\mathcal{M}%
}\left(  1/h\right)  .
\]
In fact, the structure of $\widetilde{\mathcal{M}}_{2,V}\left(  1/h\right)  $
is similar to that of $\widetilde{\mathcal{M}}\left(  1/h\right)  $, but the
propagation law that replaces (\ref{e:mh2}) involves the propagator associated
to the averaged Hamiltonian $|\xi|^{2}+\left\langle V\right\rangle _{\Lambda
}\left(  t,x\right)  $.

\item $\beta<2$. In this case, it is possible to find potentials $V$ for which
Theorem \ref{t:main}(2) fails, \emph{i.e. }such that there exist $\mu
\in\widetilde{\mathcal{M}}_{\beta,V}\left(  1/h\right)  $ such that the
projection of $\mu$ on $x$ is not absolutely continuous with respect to
$dtdx$. The following example is due to Jared Wunsch. On the 2-dimensional
torus, take $H\left(  \xi\right)  =\left\vert \xi\right\vert ^{2}$ and
$V(x_{1},x_{2}):=W(x_{2})$ such that $W(x_{2})=\left(  x_{2}\right)  ^{2}/2$
in $\{|x_{2}|<1/2\}$. Take $\varepsilon\in(0,1)$ and
\[
u_{h}(x,y):=\frac{1}{{\pi}^{1/4}h^{\varepsilon/4}}e^{i\frac{x_{1}}{h}%
}e^{-\frac{\left(  x_{2}\right)  ^{2}}{2h^{\varepsilon}}}\chi(y),
\]
where $\chi$ is a smooth function that is equal to one in $\{|x_{2}|<1/4\}$
and identically equal to $0$ in $\{|x_{2}|>1/2\}$. One checks that
\[
\left(  -h^{2}\Delta+h^{2(1-\varepsilon)}V-1\right)  u_{h}=h^{2-\varepsilon
}u_{h}+\cO(h^{\infty}).
\]
It follows that for $\varphi\in L^{1}(\R)$ and $a\in\mathcal{C}_{c}^{\infty
}\left(  T^{\ast}\mathbb{T}^{2}\right)  $,
\begin{multline*}
\lim_{h\rightarrow0^{+}}\int_{\mathbb{R}}\varphi(t)\left\langle
S_{2(1-\epsilon),h}^{t/h}u_{h},\Op_{h}(a)S_{2(1-\epsilon),h}^{t/h}%
u_{h}\right\rangle _{L^{2}(\mathbb{T}^{2})}dt\\
=\lim_{h\rightarrow0^{+}}\int_{\mathbb{R}}\varphi(t)\left\langle u_{h}%
,\Op_{h}(a)u_{h}\right\rangle _{L^{2}(\mathbb{T}^{2})}dt=\left(
\int_{\mathbb{R}}\varphi\left(  t\right)  dt\right)  \int_{T^{\ast}%
\mathbb{T}^{2}}a\left(  x,\xi\right)  \mu\left(  dx,d\xi\right)  ,
\end{multline*}
and it is not hard to see that $\mu$ is concentrated on $\{x_{2}=0,\xi
_{1}=1,\xi_{2}=0\}$. In particular the image of $\mu$ by the projection to
$\T^{2}$ is supported on $\{x_{2}=0\}$.
\end{enumerate}



\section{Hierarchies of time scales\label{s:hierarchy}}

The following results makes explicit the relation between the sets
$\widetilde{\mathcal{M}}\left(  \tau\right)  $ as the time scale $\left(
\tau_{h}\right)  $ varies.

\begin{proposition}
\label{p:convsm}Let $\left(  \tau_{h}\right)  $ and $\left(  \sigma
_{h}\right)  $ be time scales tending to infinity as $h\longrightarrow0^{+}$
such that $\lim_{h\rightarrow0^{+}}\sigma_{h}/\tau_{h}=0$. Then for every
$\mu\in\widetilde{\mathcal{M}}\left(  \tau\right)  $ and almost every
$t\in\mathbb{R}$ there exist $\mu^{t}\in\overline{\operatorname{Conv}
\widetilde{\mathcal{M}}\left(  \sigma\right)  }$ such that
\begin{equation}
\mu\left(  t,\cdot\right)  =\int_{0}^{1}\mu^{t}\left(  s,\cdot\right)  ds.
\label{e:mutcc}
\end{equation}

\end{proposition}

Before presenting the proof of this result, we shall need two auxiliary lemmas.

\begin{lemma}
\label{l:cconv}Let
$\left(  \sigma_{h}\right)  $ be a time scale tending to infinity as
$h\longrightarrow0^{+}$. Let $\left(v_{h}^{\left(  n\right)  }\right)_{h>0, n\in \N}$ be a normalised family in $L^{2}\left(
\mathbb{T}^{d}\right)$ and define:
\[
w_{h}^{\left(  n\right)  }\left(  t,\cdot\right)  :=w_{S_{h}^{\sigma_{h}
t}v_{h}^{\left(  n\right)  }}^{h}.
\]
Let $c_{h}^{\left(  n\right)  }\geq0$, $n\in \N$, be such that $\sum_{n\in
\N}c_{h}^{\left(  n\right)  }=1$.Then, every weak-$\ast$ accumulation point
in $L^{\infty}\left(  \mathbb{R};\mathcal{D}^{\prime}\left(  T^{\ast
}\mathbb{T}^{d}\right)  \right)  $ of
\begin{equation}
\sum_{n\in I_{h}}c_{h}^{\left(  n\right)  }w_{h}^{\left(  n\right)  }\left(
t,\cdot\right)  \label{e:cconv}
\end{equation}
belongs to $\overline{\operatorname{Conv}\widetilde{\mathcal{M}}\left(
\sigma\right)  }$.
\end{lemma}

\begin{proof}
Suppose (\ref{e:cconv}) possesses an accumulation point $\tilde{\mu}\in
L^{\infty}\left(  \mathbb{R};\mathcal{M}_{+}\left(  T^{\ast}\mathbb{T}
^{d}\right)  \right)  $ that does not belong to $\overline{\operatorname{Conv}
\widetilde{\mathcal{M}}\left(  \sigma\right)  }$. By the Hahn-Banach
theorem
applied to
the convex sets $\left\{  \tilde{\mu}\right\}  $ and $\overline
{\operatorname{Conv}\widetilde{\mathcal{M}}\left(  \sigma\right)  }$
we can ensure the existence of
$\varepsilon>0$, $a\in\mathcal{C}_{c}^{\infty}\left(  T^{\ast}\mathbb{T}
^{d}\right)  $ and $\theta\in L^{1}\left(  \mathbb{R}\right)  $ such that:
\[
\int_{\mathbb{R}}\theta\left(  t\right)  \left\langle \tilde{\mu}\left(
t,\cdot\right)  ,a\right\rangle dt<-\varepsilon<0,
\]
and,
\begin{equation}
\int_{\mathbb{R}}\theta\left(  t\right)  \left\langle \mu\left(
t,\cdot\right)  ,a\right\rangle dt\geq-{\varepsilon\over 3},\quad\forall\mu
\in\overline{\operatorname{Conv}\widetilde{\mathcal{M}}\left(  \sigma\right)
}.\label{e:pos}
\end{equation}
Suppose that $\tilde{\mu}$ is attained through a sequence $\left(
h_{k}\right)  $ tending to zero. For $k>k_{0}$ big enough,
\[
\int_{\mathbb{R}}\theta\left(  t\right)  \sum_{n\in I_{h_{k}}}c_{h_{k}
}^{\left(  n\right)  }\left\langle w_{h_{k}}^{\left(  n\right)  }\left(
t,\cdot\right)  ,a\right\rangle dt\leq-{3\over 2}\varepsilon,
\]
which implies that there exists $n_{k}\in \N$ such that:
\begin{equation}
\int_{\mathbb{R}}\theta\left(  t\right)  \left\langle w_{h_{k}}^{\left(
n_{k}\right)  }\left(  t,\cdot\right)  ,a\right\rangle dt\leq-{3\over 2}\varepsilon
.\label{e:neg}
\end{equation}
Therefore, every accumulation point of $\left(  w_{h_{k}}^{\left(
n_{k}\right)  }\right)  $ also satisfies (\ref{e:neg}) which contradicts
(\ref{e:pos}).
\end{proof}

\begin{lemma}
\label{l:mesint}Let $\tau$, $\sigma$ and $\mu$ be as in Proposition
\ref{p:convsm}. For every $\alpha<\beta$ there exists $\mu_{\alpha,\beta}
\in\overline{\operatorname{Conv}\widetilde{\mathcal{M}}\left(  \sigma\right)
}$ such that
\[
\frac{1}{\beta-\alpha}\int_{\alpha}^{\beta}\mu\left(  t,\cdot\right)
dt=\int_{0}^{1}\mu_{\alpha,\beta}\left(  t,\cdot\right)  dt.
\]

\end{lemma}

\begin{proof}
Let $\mu\in\widetilde{\mathcal{M}}\left(  \tau\right)  $. Then there exist an
$h$-oscillating, normalised sequence $\left(  u_{h}\right)  $ such that, for
every $\theta\in L^{1}\left(  \mathbb{R}\right)  $ and every $a\in
C_{c}^{\infty}\left(  T^{\ast}\mathbb{T}^{d}\right)  $:
\[
\lim_{h\rightarrow0^{+}}\int_{\mathbb{R}}\theta\left(  t\right)  \left\langle
S_{h}^{\tau_{h}t}u_{h},\Op_{h}(a)S_{h}^{\tau_{h}t}u_{h}\right\rangle
dt=\int_{\mathbb{R}}\theta\left(  t\right)  \left\langle \mu\left(
t,\cdot\right)  ,a\right\rangle dt.
\]
Write $N_{h}:=\tau_{h}/\sigma_{h}$; by hypothesis $N_{h}\longrightarrow\infty$
as $h\longrightarrow0^{+}$. Let $\alpha<\beta$, define $L:=\beta-\alpha$ and
put:
\[
\delta_{h}:=\frac{LN_{h}}{\left\lfloor LN_{h}\right\rfloor },\quad t_{n}
^{h}:=\alpha N_{h}+n\delta_{h},
\]
where $\left\lfloor LN_{h}\right\rfloor $ is the integer part of $LN_{h}$.
Then,
\begin{align*}
\frac{1}{L}\int_{\alpha}^{\beta}\left\langle S_{h}^{\tau_{h}t}u_{h}
,\Op_{h}(a)S_{h}^{\tau_{h}t}u_{h}\right\rangle _{L^{2}(\mathbb{T}^{d})}dt  &
=\frac{1}{LN_{h}}\int_{\alpha N_{h}}^{\beta N_{h}}\left\langle S_{h}
^{\sigma_{h}t}u_{h},\Op_{h}(a)S_{h}^{\sigma_{h}t}u_{h}\right\rangle
_{L^{2}(\mathbb{T}^{d})}dt\\
&  =\frac{1}{LN_{h}}\sum_{n=1}^{\left\lfloor LN_{h}\right\rfloor }
\int_{t_{n-1}^{h}}^{t_{n}^{h}}\left\langle S_{h}^{\sigma_{h}t}u_{h}
,\Op_{h}(a)S_{h}^{\sigma_{h}t}u_{h}\right\rangle _{L^{2}(\mathbb{T}^{d})}dt\\
&  =\frac{1}{LN_{h}}\sum_{n=1}^{\left\lfloor LN_{h}\right\rfloor }\int
_{0}^{\delta_{h}}\left\langle S_{h}^{\sigma_{h}t}v_{h}^{\left(  n\right)
},\Op_{h}(a)S_{h}^{\sigma_{h}t}v_{h}^{\left(  n\right)  }\right\rangle
_{L^{2}(\mathbb{T}^{d})}dt,
\end{align*}
where the functions $v_{h}^{\left(  n\right)  }:=S_{h}^{\sigma_{h}t_{n}^{h}
}u_{h}$ form, for each $n\in\mathbb{Z}$, a normalised sequence indexed by
$h>0$. The result then follows by Lemma \ref{l:cconv} and using the fact that
$\delta_{h}\longrightarrow1$ as $h\longrightarrow0^{+}$.
\end{proof}

\begin{proof}
[Proof of Proposition \ref{p:convsm}]Let
$\mu\in\widetilde{\mathcal{M}}\left( \tau\right)  $; an
application of the Lebesgue differentiation theorem gives the
existence of a countable dense set
$S\subset\mathcal{C}_{c}^{\infty}\left(
T^{\ast}\mathbb{T}^{d}\right)  $ and a set $N\subset\mathbb{R}$ of
measure
zero such that, for $a\in S$ and $t\in\mathbb{R}\setminus N$,
\begin{equation}
\lim_{\varepsilon\rightarrow0^{+}}\frac{1}{2\varepsilon}\int_{t-\varepsilon
}^{t+\varepsilon}\int_{T^{\ast}\mathbb{T}^{d}}a\left(  x,\xi\right)
\mu\left(  s,dx,d\xi\right)  ds=\int_{T^{\ast}\mathbb{T}^{d}}a\left(
x,\xi\right)  \mu\left(  t,dx,d\xi\right)  . \label{e:c1}
\end{equation}
Fix $t\in\mathbb{R}\setminus N$; then, for any $\varepsilon>0$ there exist
$\mu_{\varepsilon}^{t}\in\overline{\operatorname{Conv}\widetilde{\mathcal{M}
}\left(  \sigma\right)  }$ such that, for every $a\in\mathcal{C}_{c}^{\infty
}\left(  T^{\ast}\mathbb{T}^{d}\right)  $,
\begin{equation}
\frac{1}{2\varepsilon}\int_{t-\varepsilon}^{t+\varepsilon}\int_{T^{\ast
}\mathbb{T}^{d}}a\left(  x,\xi\right)  \mu\left(  s,dx,d\xi\right)
ds=\int_{0}^{1}\int_{T^{\ast}\mathbb{T}^{d}}a\left(  x,\xi\right)
\mu_{\varepsilon}^{t}\left(  s,dx,d\xi\right)  ds. \label{e:c2}
\end{equation}
Note that $\overline{\operatorname{Conv}\widetilde{\mathcal{M}}\left(
\sigma\right)  }$ is sequentially compact for the weak-$\ast$ topology,
therefore, there exist a sequence $\left(  \varepsilon_{n}\right)  $ tending
to zero and a $\mu^{t}\in\overline{\operatorname{Conv}\widetilde{\mathcal{M}
}\left(  \sigma\right)  }$ such that $\mu_{\varepsilon_{n}}^{t}$ converges
weakly-$\ast$ to $\mu^{t}$. Identities (\ref{e:c1}) and (\ref{e:c2}) ensure
that $\mu\left(  t,\cdot\right)  =\int_{0}^{1}\mu^{t}\left(  s,\cdot\right)
ds$.
\end{proof}

\begin{remark}
\label{r:cchosc}
Projecting on $x$ in identity (\ref{e:mutcc}) we deduce that given $\nu
\in\mathcal{M}\left(  \tau\right)  $ there exist $\nu^{t}\in\mathcal{M}\left(
\sigma\right)  $ such that:
\[
\nu\left(  t,\cdot\right)  =\int_{0}^{1}\nu^{t}\left(  s,\cdot\right)  ds.
\]
This, together with the fact that elements of $\mathcal{M}\left(  1/h\right)
$ are absolutely continuous imply the conclusion of Theorem \ref{t:main}(2)
when $\tau_{h}\gg1/h$.
\end{remark}

Denote by $\widetilde{\mathcal{M}}\left(  \infty\right)  $ the set of
weak-$\ast$ limit points of sequences of Wigner distributions $\left(
w_{u_{h}}\right)  $ corresponding to sequences $\left(  u_{h}\right)  $
consisting of normalised eigenfunctions of $H\left(  hD_{x}\right)  $. We now
focus on a family of time scales $\tau$ for which the structure of
$\widetilde{\mathcal{M}}\left(  \tau\right)  $ can be described in terms of
the closed convex hull of $\widetilde{\mathcal{M}}\left(  \infty\right)  $.
Given a measurable subset $O\subseteq\mathbb{R}^{d}$, we define:
\[
\tau_{h}^{H}\left(  O\right)  :=h\sup\left\{  \left\vert H\left(  hk\right)
-H\left(  hj\right)  \right\vert ^{-1}\;:\;H\left(  hk\right)  \neq H\left(
hj\right)  ,\;hk,hj\in h\mathbb{Z}^{d}\cap O\right\}  .
\]
Note that the scale $\tau_{h}^{H}$ defined in the introduction coincides with
$\tau_{h}^{H}\left(  \mathbb{R}^{d}\right)  $. The following holds.

\begin{proposition}
\label{p:muconv}Let $O\subseteq\mathbb{R}^{d}$ be an open set such that
$\tau_{h}^{H}\left(  O\right)  $ tends to infinity as $h\longrightarrow0^{+}$.
Suppose $\left(  \tau_{h}\right)  $ is a time scale such that $\lim
_{h\rightarrow0^{+}}\tau_{h}^{H}\left(  O\right)  /\tau_{h}=0$. If $\mu
\in\widetilde{\mathcal{M}}\left(  \tau\right)  $ is obtained trough a sequence
whose semiclassical measure satisfies $\mu_{0}\left(  \mathbb{T}^{d}
\times\left(  \mathbb{R}^{d}\setminus O\right)  \right)  =0$ then $\mu
\in\overline{\operatorname{Conv}\widetilde{\mathcal{M}}\left(  \infty\right)
}$.
\end{proposition}

\begin{proof} As in~(\ref{e:expli}),
for $a\in\mathcal{C}_{c}^{\infty}\left(  T^{\ast}\mathbb{T}
^{d}\right)  $ and $\theta\in L^{1}\left(  \mathbb{R}\right)  $, we write:
\begin{multline}
\int_{\mathbb{R}}\theta\left(  t\right)  \left\langle w_{h}\left(  t\right)
,a\right\rangle dt=\label{e:wignerbigt}\\
=\frac{1}{(2\pi)^{d/2}}\sum_{h,j\in\mathbb{Z}^{d}}\widehat{\theta}\left(
\tau_{h}\frac{H\left(  hk\right)  -H\left(  hj\right)  }{h}\right)
\widehat{u}_{h}(k)\overline{\widehat{u}_{h}(j)}\widehat{a}_{j-k}\left(
\frac{h}{2}(k+j)\right)  .\nonumber
\end{multline}
Our assumptions on the semiclassical measure of the initial data implies that,
for a.e. $t\in\mathbb{R}$:
\[
\mu\left(  t,\mathbb{T}^{d}\times\left(  \mathbb{R}^{d}\setminus O\right)
\right)  =0.
\]
Suppose that $\mu$ is obtained through the normalised sequence $\left(
u_{h}\right)  $. Suppose that $a\in\mathcal{C}_{c}^{\infty}\left(
\mathbb{T}^{d}\times O\right)  $ and that $\operatorname*{supp}\widehat
{\theta}$ is compact. For $0<h<h_{0}$ small enough,
\[
\tau_{h}\frac{H\left(  hk\right)  -H\left(  hj\right)  }{h}\notin
\operatorname*{supp}\widehat{\theta},\quad\forall hk,hj\in O\text{ such that
}H\left(  hk\right)  \neq H\left(  hj\right)  .
\]
Therefore, for such $h$, $a$ and $\theta$,
\begin{align*}
\int_{\mathbb{R}}\theta\left(  t\right)  \left\langle w_{h}\left(  t\right)
,a\right\rangle dt  &  =\frac{\widehat{\theta}\left(  0\right)  }{(2\pi
)^{d/2}}\sum_{\substack{kh,hj\in O\\H\left(  hk\right)  =H\left(  hj\right)
}}\widehat{u}_{h}(k)\overline{\widehat{u}_{h}(j)}\widehat{a}_{j-k}\left(
\frac{h}{2}(k+j)\right) \\
&  =\widehat{\theta}\left(  0\right)  \sum_{E_{h}\in H\left(  h\mathbb{Z}
^{d}\right)  \cap H\left(  O\right)  }\left\langle P_{E_{h}}u_{h}
,\Op_{h}(a)P_{E_{h}}u_{h}\right\rangle _{L^{2}(\mathbb{T}^{d})},
\end{align*}
where $P_{E_{h}}$ stands for the orthogonal projector onto the eigenspace
associated to the eigenvalue $E_{h}$. This can be rewritten as:
\[
\int_{\mathbb{R}}\theta\left(  t\right)  \left\langle w_{h}\left(  t\right)
,a\right\rangle dt=\widehat{\theta}\left(  0\right)  \sum_{E_{h}\in H\left(
h\mathbb{Z}^{d}\right)  \cap H\left(  O\right)  }c_{h}^{E_{h}}\left\langle
w_{v_{h}^{E_{h}}}^{h},a\right\rangle ,
\]
where
\[
v_{h}^{E_{h}}:=\frac{P_{E_{h}}u_{h}}{\left\Vert P_{E_{h}}u_{h}\right\Vert
_{L^{2}\left(  \mathbb{T}^{d}\right)  }},\quad\text{and\quad}c_{h}^{E_{h}
}:=\left\Vert P_{E_{h}}u_{h}\right\Vert _{L^{2}\left(  \mathbb{T}^{d}\right)
}^{2}.
\]
Note that $v_{h}^{E_{h}}$ are eigenfunctions of $H\left(  hD_{x}\right)  $ and
the fact that $\left(  u_{h}\right)  $ is normalised implies:
\[
\sum_{E_{h}\in H\left(  h\mathbb{Z}^{d}\right)  \cap H\left(  O\right)  }
c_{h}^{E_{h}}=1.
\]
We conclude by applying (a straightforward adaptation of) Lemma \ref{l:cconv}
to $v_{h}^{E_{h}}$ and $c_{h}^{E_{h}}$.
\end{proof}

\begin{corollary}
Suppose $\tau_{h}^{H}:=\tau_{h}^{H}\left(  \mathbb{R}^{d}\right)
\longrightarrow\infty$ as $h\longrightarrow0^{+}$ and that $\left(  \tau
_{h}\right)  $ is a time scale such that $\tau_{h}^{H}\ll\tau_{h}$. Then
\[
\widetilde{\mathcal{M}}\left(  \tau\right)  =\overline{\operatorname{Conv}
\widetilde{\mathcal{M}}\left(  \infty\right)  }.
\]

\end{corollary}

\begin{proof}
The inclusion $\widetilde{\mathcal{M}}\left(  \tau\right)  \subseteq
\overline{\operatorname{Conv}\widetilde{\mathcal{M}}\left(  \infty\right)  }$
is a consequence of the previous result with $O=\mathbb{R}^{d}$. The converse
inclusion can be proved by
reversing the steps of the proof of Proposition \ref{p:muconv}.
\end{proof}

\begin{remark}
Proposition \ref{p:conv} is a direct consequence of this result.
\end{remark}

\section{Appendix: Basic properties of Wigner distributions and semi-classical
measures}

\label{sec:Ap1}

In this Appendix, we review basic properties of Wigner distributions and
semiclassical measures. Recall that we have defined $w_{u_{h}}^{h}$ for
$u_{h}\in L^{2}\left(  \mathbb{T}^{d}\right)  $ as:
\begin{equation}
\int_{T^{\ast}\mathbb{T}^{d}}a(x,\xi)w_{u_{h}}^{h}(dx,d\xi)=\left\langle
u_{h},\Op_{h}(a)u_{h}\right\rangle _{L^{2}(\mathbb{T}^{d})},\qquad
\mbox{ for all }a\in\mathcal{C}_{c}^{\infty}(T^{\ast}\mathbb{T}^{d}
),\label{e:defWD}
\end{equation}
Start noticing that (\ref{e:defWD}) admits the more explicit expression:
\begin{equation}
\int_{T^{\ast}\mathbb{T}^{d}}a(x,\xi)w_{u_{h}}^{h}(dx,d\xi)=\frac{1}
{(2\pi)^{d/2}}\sum_{k,j\in\Z^{d}}\widehat{u_{h}}(k)\overline{\widehat{u_{h}
}(j)}\widehat{a}_{j-k}\left(  \frac{h}{2}(k+j)\right)  ,\label{e:Weylq}
\end{equation}
where $\widehat{u}_{h}(k):=\int_{\mathbb{T}^{d}}u_{h}(x)\frac{e^{-ik.x}}
{(2\pi)^{d/2}}dx$ and $\widehat{a}_{k}(\xi):=\int_{\mathbb{T}^{d}}
a(x,\xi)\frac{e^{-ik.x}}{(2\pi)^{d/2}}dx$ denote the respective Fourier
coefficients of $u_{h}$ and $a$, with respect to the variable $x\in
\mathbb{T}^{d}$.

By the Calder\'{o}n-Vaillancourt theorem \cite{CV}, the norm of $\Op_{h}(a)$
is uniformly bounded in~$h$: indeed, there exists an integer $K_{d}$, and a
constant $C_{d}>0$ (depending on the dimension $d$) such that, if $a$ is a
smooth function on $T^{\ast}\IT^{d}$, with uniformly bounded derivatives,
then
\[
\norm{\Op_1(a)}_{L^{2}(\IT^{d})\To L^{2}(\IT^{d})}\leq C_{d}\sum_{\alpha
\in\IN^{2d},|\alpha|\leq K_{d}}\sup_{T^{\ast}\IT^{d}}|\partial^{\alpha
}a|=:C_{d}M\left(  a\right)  .
\]
A proof in the case of $L^{2}(\IR^{d})$ can be found in
\cite{DimassiSjostrand}. As a consequence of this, equation (\ref{e:defWD})
gives:%
\[
\left\vert \int_{T^{\ast}\mathbb{T}^{d}}a(x,\xi)w_{u_{h}}^{h}(dx,d\xi
)\right\vert \leq C_{d}\left\Vert u_{h}\right\Vert _{L^{2}\left(
\mathbb{T}^{d}\right)  }^{2}M\left(  a\right)  ,\qquad\mbox{ for all }a\in
\mathcal{C}_{c}^{\infty}(T^{\ast}\mathbb{T}^{d}).
\]
Therefore, if $w_{h}(t,\cdot):=w_{S_{h}^{\tau_{h}t}u_{h}}^{h}$ for some
function $h\longmapsto\tau_{h}\in\mathbb{R}_{+}$ and $\left(  u_{h}\right)  $
is bounded in $L^{2}\left(  \mathbb{T}^{d}\right)  $ one has that $\left(
w_{h}\right)  $ is uniformly bounded in $L^{\infty}\left(  \mathbb{R}
;\mathcal{D}^{\prime}\left(  T^{\ast}\mathbb{T}^{d}\right)  \right)  $. Let us consider
$\mu\in L^{\infty}\left(  \mathbb{R};\mathcal{D}^{\prime}\left(  T^{\ast
}\mathbb{T}^{d}\right)  \right)  $  an accumulation point of $\left(
w_{h}\right)  $ for the weak-$\ast$ topology.

It follows from standard results on the Weyl quantization that $\mu$ enjoys
the following properties~:

\begin{enumerate}
\item $\mu\in L^{\infty}(\R;\cM_{+}(T^{\ast}\mathbb{T}^{d}))$, meaning that
for almost all $t$, $\mu(t,\cdot)$ is a positive measure on $T^{\ast
}\mathbb{T}^{d}$.

\item The unitary character of $S_{h}^{t}$ implies that $\int_{T^{\ast
}\mathbb{T}^{d}}\mu(t,dx,d\xi)$ does not depend on $t$; from the normalization
of $u_{h}$, we have $\int_{T^{\ast}\mathbb{T}^{d}}\mu(\tau,dx,d\xi)\leq1$, the
inequality coming from the fact that $T^{\ast}\mathbb{T}^{d}$ is not compact,
and that there may be an escape of mass to infinity. Such escape does not
occur if and only if $\left(  u_{h}\right)  $ is $h$-oscillating, in which
case $\mu\in L^{\infty}\left(  \mathbb{R};\mathcal{P}\left(  T^{\ast
}\mathbb{T}^{d}\right)  \right)  $.

\item If $\tau_{h}\To\infty$ as $h\To0^{+}$ then the measures $\mu(t,\cdot)$
are invariant under $\phi_{s}$, for almost all $t$ and all $s$.

\item Let $\bar{\mu}$ be the measure on $\mathbb{R}^{d}$ image of $\mu
(t,\cdot)$ under the projection map $(x,\xi)\longmapsto\xi$. Then $\bar{\mu}$
does not depend on $t$. Moreover, if $\overline{\mu_{0}}$ stands for the image
under the same projection of any semiclassical measure corresponding to the
sequence of initial data $\left(  u_{h}\right)  $ then $\bar{\mu}
=\overline{\mu_{0}}$.
\end{enumerate}

For the reader's convenience, we next prove statements (3) and (4)  (see also
\cite{MaciaAv} for a proof of these results in the context of the
Schr\"{o}dinger flow $e^{iht\Delta}$ on a general Riemannian manifold). Let us
begin with the invariance through the Hamiltonian flow. We set
\[
a_{s}(x,\xi):=a(x+sdH(\xi),\xi)=a\circ\phi_{s}(x,\xi).
\]
The symbolic calculus for Wey's quantization implies:
\begin{align*}
{\frac{d}{ds}}S_{h}^{s}\Op_{h}(a_{s})S_{h}^{-s} &  =S_{h}^{s}\Op_{h}
(\partial_{s}a_{s})S_{h}^{-s}-{\frac{i}{h}}S_{h}^{s}\left[  H(hD)\;,\;\Op_{h}
(a_{s})\right]  S_{h}^{-s}\\
&  =O(h^{2}).
\end{align*}
Therefore, $S_{h}^{s}\Op_{h}(a_{s})S_{h}^{-s}=\Op_{h}(a)+O(h^{2})$ and for
$\theta\in L^{1}(\R)$,
\begin{align*}
\int_{\mathbb{R}}\theta(t)\left\langle w_{h}\left(  t\right)  ,a\right\rangle
dt &  =\int_{\mathbb{R}}\theta(t)\la u_{h}\;,\;S_{h}^{-\tau_{h}t}
\Op_{h}(a)S_{h}^{\tau_{h}t}u_{h}\ra dt\\
&  =\int_{\mathbb{R}}\theta(t)\la u_{h}\;,\;S_{h}^{-\tau_{h}(t-s/\tau_{h}
)}\Op_{h}(a\circ\phi_{s})S_{h}^{\tau_{h}(t-s/\tau_{h})}u_{h}\ra dt+O(h^{2})\\
&  =\int_{\mathbb{R}}\theta(t+s/\tau_{h})\la u_{h}\;,\;S^{-\tau_{h}t}
\Op_{h}(a\circ\phi_{s})S^{\tau_{h}t}u_{h}\ra dt+O(h^{2})\\
&  =\int_{\mathbb{R}}\theta(t+s/\tau_{h})\la w_{h}\left(  t\right)
,a\circ\phi_{s}\ra dt+O(h^{2}).
\end{align*}
Since $\Vert\theta(\cdot+s/\tau_{h})-\theta\Vert_{L^{1}}\To0$ (recall that we
have assumed that  $\tau_{h}\To\infty$ as $h\To0^{+}$) we obtain
\[
\int_{\mathbb{R}}\theta(t)\left\langle w_{h}\left(  t\right)  ,a\right\rangle
dt-\int_{\mathbb{R}}\theta(t)\la w_{h}\left(  t\right)  ,a\circ\phi_{s}\ra
dt\To0,\text{ as }h\To0^{+},
\]
whence the invariance under $\phi_{s}$.

Let us now prove property (4). Consider $\overline{\mu}$ the image of $\mu$ by
the projection $(x,\xi)\longmapsto\xi$, we have for $a\in{\mathcal{C}}
_{0}^{\infty}(\R^{d})$ :
\begin{align*}
\left\langle w_{h}\left(  t\right)  ,a\left(  \xi\right)  \right\rangle
-\left\langle w_{u_{h}}^{h},a\left(  \xi\right)  \right\rangle  &  =\int
_{0}^{t}\frac{d}{ds}\langle w_{h}(s)\;,\;a(\xi)\rangle ds\\
&  =\int_{0}^{t}\la u_{h}\;,\;\frac{d}{ds}\left(  S_{h}^{-\tau_{h}s}
\Op_{h}(a)S_{h}^{\tau_{h}s}\right)  u_{h}\ra ds\\
&  =0,
\end{align*}
as ${\frac{d}{ds}}S_{h}^{-s}\Op_{h}(a(\xi))S_{h}^{s}=0$ (for $a$ only
depending on $\xi$ we have $\Op_{h}(a)=a(hD_{x})$, which commutes with
$H(hD_{x})$). Therefore, taking limits we find, for every $\theta\in
L^{1}\left(  \mathbb{R}\right)  $:
\[
\int_{\mathbb{R}}\theta\left(  t\right)  \int_{T^{\ast}\mathbb{T}^{d}}a\left(
\xi\right)  \mu\left(  t,dx,d\xi\right)  =\left(  \int_{\mathbb{R}}
\theta\left(  t\right)  dt\right)  \int_{T^{\ast}\mathbb{T}^{d}}a\left(
\xi\right)  \mu_{0}\left(  dx,d\xi\right)  ,
\]
where $\mu_{0}$ is any accumulation point of $\left(  w_{u_{h}}^{h}\right)  $.
As a consequence of this, we find that $\overline{\mu}$ does not depend on $t$
and:
\[
\overline{\mu}\left(  \xi\right)  =\int_{\mathbb{T}^{d}}\mu_{0}\left(
dy,\xi\right)  .
\]


\def\cprime{$'$} \def\cprime{$'$}

\end{document}